\documentclass[11pt,a4paper]{article}
\usepackage[T1]{fontenc}
\usepackage[margin=25mm]{geometry}
\usepackage{amsmath,amssymb,amsthm,mathtools}
\usepackage{microtype,enumitem,placeins,tikz}
\usetikzlibrary{calc}
\usepackage[hidelinks]{hyperref}
\hypersetup{pdftitle={Total eight-colorings without the Four-Color Theorem},
  pdfsubject={Complete proof revision using Hamiltonian duals}}
\setlist[enumerate]{itemsep=3pt,topsep=5pt}
\numberwithin{equation}{section}
\newtheorem{theorem}{Theorem}[section]
\newtheorem{lemma}[theorem]{Lemma}
\newtheorem{proposition}[theorem]{Proposition}
\newtheorem{corollary}[theorem]{Corollary}
\theoremstyle{definition}

\newtheorem{question}[theorem]{Question}
\theoremstyle{remark}
\newtheorem{remark}[theorem]{Remark}

\newcommand{\chit}{\chi^{\prime\prime}}
\title{Total eight-colorings of maximal planar graphs\protect\\
of maximum degree six\protect\\[6pt]}

\author{Enqiang Zhu \protect\\
\vspace{0.1cm}
{\footnotesize{The Institute of Computing Technology, Guangzhou University, Guangzhou, 510006, China}}
}

\date{}

\begin{document}
\maketitle

\begin{abstract}
We prove that every finite simple maximal planar graph of maximum
degree at most six admits a total eight-coloring.
We first show that every proper vertex four-coloring of a subcubic
graph extends to a total five-coloring. A curvature inequality for
boundary walks then gives a matching decomposition of four-connected
triangulations with degrees five and six, extending their prescribed
vertex four-colorings to total eight-colorings. In a minimum
counterexample, two local boundary colorings force the degree-four
vertices to have disjoint neighborhoods. A characterization of the
local extension obstructions permits these vertices to be restored
simultaneously. The proof uses Hamiltonian-connectedness of
four-connected planar graphs and Kardo\v{s}'s Hamiltonicity theorem for three-connected
cubic planar graphs.

\end{abstract}

\section{Introduction}
\label{intro:section}

\subsection{The problem}

For a positive integer $k$, a \emph{total $k$-coloring} of a graph $G$
is a map $\psi:V(G)\cup E(G)\to\{1,\ldots,k\}$ such that
adjacent vertices, adjacent edges, and a vertex and an incident edge
receive different colors. Not all $k$ colors need be used; the
color set may equivalently be any set of cardinality $k$.
The minimum number of colors in a total coloring is the
\emph{total chromatic number}, denoted by $\chit(G)$. For
nonempty $G$, a vertex of maximum degree together with its incident edges
forms a set of $\Delta(G)+1$ mutually conflicting elements.  Thus
$\chit(G)\geq\Delta(G)+1$. In his 1965 doctoral thesis,
Behzad~\cite{behzad1965} proposed the Total Coloring Conjecture,
which asserts the complementary upper bound
\begin{equation}\label{intro:tcc}
 \chit(G)\leq\Delta(G)+2
\end{equation}
for every finite simple graph.

The incidence constraints couple the vertex and edge colorings:
each edge must avoid the colors of its endpoints. Our main result
establishes the conjectured bound for maximal planar graphs of
maximum degree six.

\begin{theorem}\label{main:theorem}
Every finite simple maximal planar graph $G$ with $\Delta(G)\leq6$
satisfies \mbox{$\chit(G)\leq8$}.
\end{theorem}

For $\Delta(G)=6$, the bound is exactly \eqref{intro:tcc}.
The cases $\Delta(G)\leq5$ already follow from the stronger
seven-color bound of Kostochka~\cite{kostochka1996}.

The proof uses both the local and global structure of a triangulation.
The cyclic link of a vertex describes the faces exposed by its
deletion. Euler's formula gives
$\sum_{v\in V(G)}(6-d_G(v))=12$, with nonnegative summands when
$\Delta(G)\leq6$. Applied to a connected region, the same count
relates its degree deficit to its boundary walk. This yields the
cut inequality in Lemma~\ref{core:curvature} and, ultimately,
the matching required to color the core.

The maximality assumption is substantive.  Completing a plane graph
to a triangulation can raise its maximum degree, so
Theorem~\ref{main:theorem} does not imply an eight-color theorem for
all planar graphs of maximum degree six.  The broader problem is
described as unresolved in recent work on additional forbidden
configurations~\cite{sufangzhu2025,zhuzhouxu2026}.

\subsection{Background and motivation}

The Total Coloring Conjecture has long been known for subcubic graphs;
Feng and Lin~\cite{fenglin2013} give a concise proof and references to
the original work.  Kostochka~\cite{kostochka1996} proved that every
multigraph of maximum degree five has total chromatic number at most
seven.  For planar graphs of maximum degree six, a number of results
impose restrictions on local configurations.  Zhu and Xu
\cite{zhuxu2017}, for example, prove the existence of total eight-colorings under
the exclusion of a specified fan configuration.  More recent
preprints impose other forbidden subgraphs or restrictions on
four-cycles~\cite{sufangzhu2025,zhuzhouxu2026}.  These hypotheses are
not assumed here.

Results for particular triangulation families also illustrate why the
induction class matters.  Zhou, Zhao, Ma and Xu~\cite{zhouetal2022}
establish the conjectured bound for recursive maximal planar graphs,
obtained by repeatedly inserting a vertex into a triangular face and
joining it to the three boundary vertices. That construction supplies a
degree-three deletion order.  A triangulation whose minimum degree is
four need not admit such an order.  In the present proof the
degree-three reduction is only the beginning; the principal difficulty
is the interaction between vertices of degrees four, five and six.

One general approach combines vertex and edge coloring:
every four-colorable Class~1 graph satisfies
$\chit(G)\leq\Delta(G)+2$; see Cranston and West
\cite[Theorem~4.7]{cranstonwest2013}.  Here Class~1 means that the
edges admit a proper coloring with $\Delta(G)$ colors. Our argument
requires only three edge-disjoint matchings whose deletion leaves a
subcubic graph. The extension theorem for that remainder preserves
every supplied proper vertex four-coloring, so its total coloring
can be combined with three new colors on the matchings. No
Class~1 hypothesis is needed.

Reconstruction creates a second difficulty. Deleting a degree-four
vertex and triangulating its hole gives a smaller graph of maximum
degree at most six, but an arbitrary total coloring of that graph
may fail to extend while the exterior is fixed.
Theorem~\ref{wheel:local} identifies the two local obstructions.
One can be removed by recoloring a neighbor of the deleted vertex;
the other requires a change farther outside the wheel. We control
these changes by first constructing a total coloring in which all
old vertices use a common four-element color set.

\subsection{Proof strategy}

Three extension results organize the proof.

Theorem~\ref{ext:main} extends any proper vertex four-coloring
of a subcubic graph to a total five-coloring. In a bridgeless
subcubic graph, a matching covering all degree-three vertices leaves
only paths, cycles and isolated vertices. At a bridge, the
color imposed by the already colored side must be accommodated.
The key local statement, Proposition~\ref{ext:single-port}, permits
every legal color on one outgoing bridge.  It combines a
near-perfect matching with a propagation rule for two-element
edge lists.

Theorem~\ref{core:main} extends any proper vertex
four-coloring of a four-connected sphere triangulation with degrees
five and six to a total eight-coloring.  Such a graph has twelve
vertices of degree five.  A spanning path between two of them gives
two matchings. After deleting the path, we find a third matching
covering every remaining degree-four vertex. The boundary inequality
in Lemma~\ref{core:curvature} and the two-edge-cut analysis in
Lemma~\ref{core:two-cut} establish the required matching condition.
Assigning a separate color to each matching leaves a subcubic graph
to which Theorem~\ref{ext:main} applies.

Theorem~\ref{restore:main} restores an independent family of
degree-four vertices whose induced rims are pairwise disjoint and
whose neighbors have original degree six. Its input is a
total eight-coloring of the deleted graph with all vertices in one
four-element color set. Two-color interchanges in the total graph
first remove the alternating obstructions. The restriction on old
vertex colors ensures that the relevant components are paths,
cycles or isolated vertices. The remaining uniform obstructions
are then removed one at a time, using the current exterior vertex
colors. This order keeps the restorations compatible.

The link between these statements is structural.  Elementary
reductions first show that a minimum counterexample would be
four-connected and that every degree-four vertex would have only
degree-six neighbors. Theorem~\ref{wheel:local} and
Lemma~\ref{shell:witnesses} then force eight exterior witness demands.
Distinctness of their exterior positions and two proper boundary
colorings yield degree bounds. Consequently the degree-four vertices are at
pairwise distance at least three.  Removing all of them and adding
one diagonal in each hole yields precisely the core required above.
The simplicity, degree bounds and four-connectivity of this core are
proved in Theorem~\ref{shell:core}.

\subsection{Dependencies and outline}

Hall's theorem (Theorem~\ref{prelim:hall}) is used for the small extension
problems.  Tutte's matching theorem~\cite{tutte1954} is used through
the required-vertex version in Lemma~\ref{core:required}, for which we include a proof.  The spanning
path in the core is supplied by Theorem~\ref{core:hamiltonian-connected},
the Hamiltonian-connectedness result of Thomassen~\cite{thomassen1983}
for four-connected planar graphs, with the correction to its proof
by Chiba and Nishizeki~\cite{chiba1986}. Finally, Kardo\v{s}'s
Hamiltonicity theorem~\cite{kardos2020} is applied to the cubic dual.
Lemma~\ref{dual:two-trees} then constructs a proper vertex four-coloring
by two-coloring two induced trees. This use of planar duality and tree
coloring does not invoke the Four-Color Theorem. Kardo\v{s}'s theorem
is itself computer assisted, and is the external computational dependency
of this revision. Its proof is cited rather than reproduced. Standard
polyhedral duality and the Jordan curve theorem are used in the new
section; neither requires the Four-Color Theorem. The extension and
structural arguments require no new enumeration or search certificate.

After the preliminaries, Sections~\ref{core:matching-tools}--\ref{core:section}
develop the matching and prescribed-color extension tools.
Sections~\ref{red:section}--\ref{shell:section} construct the core;
Sections~\ref{restore:section} and~\ref{main:section} restore the
deleted vertices and complete the proof.
Section~\ref{scope:section} discusses prescribed vertex colors,
the scope of the external Hamiltonicity input, simpler subclass proofs,
and extensions beyond triangulations.

\section{Preliminaries}
\label{prelim:section}

\subsection{Graphs, embeddings and cuts}
\label{prelim:graphs}

All graphs are finite and, unless explicitly stated otherwise, simple.
For a graph $F$, write $V(F)$ and
$E(F)$ for its vertex and edge sets, $N_F(v)$ for the neighbors of
$v$, $d_F(v)$ for its degree, $V_i(F)=\{v:d_F(v)=i\}$, and
$n_i(F)=|V_i(F)|$.
For an edge $e=xy$, write $V(e)=\{x,y\}$ for its endpoint set.
Two distinct edges are adjacent if they share an endpoint.
A vertex is incident with an edge if it is one of that edge's endpoints.
The induced subgraph on a vertex set $S$ is $F[S]$; $F-S$ deletes
the vertices of $S$ and their incident edges.  If $M\subseteq E(F)$,
$F-M$ deletes the edges of $M$ and retains all vertices. We write
$\delta_F(S)=\{xy\in E(F):x\in S,\ y\notin S\}$
for the \emph{edge cut} determined by $S$: it consists of all
edges with exactly one endpoint in $S$. If $S$ is nonempty and
proper, deleting all edges of $\delta_F(S)$ leaves no path
between $S$ and $V(F)\setminus S$. An individual edge belonging
to this set need not be a \emph{bridge}, that is, an edge whose
deletion increases the number of connected components.
For example, if $F$ is a triangle and $S$ consists of one
vertex, then $\delta_F(S)$ contains two edges, neither of which
is a bridge.
A matching is a set of edges no two of which share an endpoint.
An empty matching and isolated vertices are permitted throughout.

For an integer $k\geq1$, a graph $F$ is \emph{$k$-connected}
(or \emph{$k$-vertex-connected}) if $|V(F)|\geq k+1$ and $F-S$
is connected for every $S\subseteq V(F)$ with $|S|<k$.
In particular, a \emph{four-connected} graph has at least five
vertices and remains connected after deletion of any set of at most
three vertices. Connectivity is vertex connectivity unless stated
otherwise.

A planar graph is \emph{maximal planar} if no edge can be added
between two existing vertices while preserving planarity.  For order
at least three, its sphere embedding is a triangulation: each face
has boundary a three-cycle.  Conversely a sphere triangulation
is maximal planar.  We call a triangular embedding on at least
four vertices a \emph{sphere triangulation}, and treat $K_3$ separately.
Let $\mathcal T_6$ be the sphere triangulations of maximum degree at
most six together with $K_3$.  The graphs of order below three will
be handled directly. We keep the sphere embedding fixed and make
each local modification within the specified disk.

The \emph{link} of a vertex of a sphere triangulation is the cycle
of its neighbors in their cyclic order.  Its vertices are distinct,
and two consecutive link vertices form a face with the given vertex.
A subgraph is \emph{induced} if it contains every edge of the ambient
graph whose two endpoints lie in the subgraph.  Accordingly, the link
cycle is induced precisely when no edge joins two nonconsecutive link
neighbors; such an edge would be a chord of the link.

We fix the terminology for the local configurations used below.
Let $v$ have degree four, with distinct neighbors $u_0,u_1,u_2,u_3$.
The subgraph with these five vertices and the four edges $vu_i$
is a \emph{four-star}. Its \emph{center} is $v$, and its four
edges are called \emph{spokes}. If $u_0u_1u_2u_3u_0$ is a cycle,
the union of this cycle and the four-star is a \emph{wheel}.
The cycle is its \emph{rim}, and the vertices $u_i$ are called
\emph{rim vertices}. Relative to this specified wheel,
a \emph{nonrim edge} is an edge of $G-v$ other than the four
rim edges. Additional edges between nonconsecutive rim vertices
are allowed unless the rim is explicitly required to be induced.
For a degree-four vertex of a sphere triangulation, its cyclic
link supplies this rim.

Let $C=abca$ be a three-cycle in a fixed sphere embedding of $G$.
Its drawing is a simple closed curve, whose complement in the sphere
has two open connected regions, called its \emph{open sides}.
The cycle $C$ is a \emph{separating triangle} if each open side
contains at least one vertex of $G$ other than $a,b,c$.
In a plane drawing, this means that there is a vertex strictly
inside $C$ and another strictly outside $C$.
A triangle is \emph{facial} if it bounds a face. In a sphere
triangulation, a side of a three-cycle containing no vertices is a
face; hence a three-cycle is separating precisely when it is not
facial.

A \emph{triangulated disk} is a finite embedded
triangulation of a closed topological disk: its boundary is a specified
simple cycle (the boundary vertices are distinct, apart from the
usual identification of the first and last vertex), and every
two-dimensional cell in the interior of that boundary is a triangle.
The boundary itself is not counted as an additional face.  At a
boundary vertex its link inside the disk is a simple path between the
two boundary neighbors.

For a sphere triangulation $F$ with $n$ vertices, $m$ edges and
$f$ faces, Euler's formula and $3f=2m$ give $m=3n-6$.
Summing degrees therefore gives
\begin{equation}\label{prelim:defect}
 \sum_{v\in V(F)}(6-d_F(v))=12.
\end{equation}
If all degrees belong to $\{4,5,6\}$, this becomes
$2|V_4(F)|+|V_5(F)|=12$.  For an induced connected region, its
facial boundary may instead be a walk with repeated vertices or
edges.  In Section~\ref{core:section}, all lengths of such walks
count edge occurrences, including repeated occurrences.

\subsection{Coloring conventions}
\label{prelim:coloring}

We refer to a vertex or an edge as an \emph{element}.  Two elements
\emph{conflict} if a proper total coloring requires them to have
different colors.  Thus two vertex-elements conflict when adjacent,
two edge-elements when their edges are adjacent, and a vertex-element and an
edge-element when incident.  The \emph{total graph} has these elements
as its vertices and these conflicts as its edges.  This language
will only be needed to specify two-color interchanges precisely.

A partial total coloring is proper on its colored elements.
Deleting an edge discards its color but retains its endpoint colors.
Restoring a deleted vertex requires colors for both the vertex and
its incident edges, respecting every total-coloring constraint.
In particular, four distinct available spoke colors need not leave a
color for a degree-four center.

In the eight-color arguments, $\Omega$ denotes the set of eight
available colors.  For the core coloring we take
$\Omega=\{1,2,3,4,5,6,7,8\}$. The symbol $B$ denotes the fixed
vertex color set $\{1,2,3,4\}$, as in Theorem~\ref{ext:main}, where
$c$ is a fifth color. In a local wheel argument, we instead use
$\mathcal{H}=\{\psi(u_0),\psi(u_1),\psi(u_2),\psi(u_3)\}$
for the set of colors on the four rim vertices in the supplied
coloring $\psi$. This set depends on that wheel and coloring
and may have fewer than four elements. Its complement in the
eight-element color set is denoted by $R=\Omega\setminus\mathcal{H}$.
A \emph{spoke list} omits the color of the old
endpoint and its already colored incident edges, but does not yet
impose the color of the uncolored center.  Its exact meaning is
given in Section~\ref{wheel:lists}.

When a two-color list is regarded algebraically in
Section~\ref{ext:section}, identify the four colors with
$\mathbb F_2^2$.  For instance, one may identify $1,2,3,4$ with
$(0,0),(0,1),(1,0),(1,1)$, with addition in each coordinate modulo
two. These four vectors sum to $(0,0)$. Thus if an edge has endpoint
colors $a,b$ and its two remaining colors in $B$ are $d,e$, then
$a+b=d+e$. We use this identity to describe forced choices along
paths.

\subsection{Systems of distinct representatives}

Let $(A_\lambda)_{\lambda\in\Lambda}$ be a finite indexed family
of finite sets. A \emph{system of distinct representatives}
is a choice of one element $a_\lambda\in A_\lambda$ for each
$\lambda\in\Lambda$, with $a_\lambda\neq a_\mu$ whenever
$\lambda\neq\mu$. Different indices may specify equal sets;
they still require different representatives and are counted
separately when a subfamily is selected.
We use the following finite form of Hall's theorem~\cite{hall1935}.

\begin{theorem}\label{prelim:hall}
Let $\Lambda$ be a finite index set, and let
$(A_\lambda)_{\lambda\in\Lambda}$ be a family of finite sets.
This family has a system of distinct representatives if and only if
\[
 \left|\bigcup_{\lambda\in J}A_\lambda\right|\geq |J|
 \qquad\text{for every }J\subseteq\Lambda.
\]
The union indexed by the empty set is understood to be empty.
\end{theorem}

The inequalities in Theorem~\ref{prelim:hall} are called
\emph{Hall's conditions}. They are necessary because the $|J|$
representatives indexed by $J$ are distinct elements of the union;
the theorem asserts their sufficiency.

\section{Vertex colors from a Hamiltonian dual}
\label{dual:section}

This section supplies the initial vertex coloring independently of the
Four-Color Theorem. The only substantial new external input is the
following published theorem.

\begin{theorem}[Kardo\v{s}~\cite{kardos2020}]
\label{dual:kardos}
Every three-connected cubic planar graph whose faces have length at
most six has a Hamiltonian cycle.
\end{theorem}

The statement is Theorem~1 of the cited work. Its published proof uses
finite computer verification; the theorem is not an assumption about
the total coloring under consideration.

\begin{lemma}\label{dual:two-trees}
Let $T$ be a finite simple sphere triangulation on at least four
vertices. If its planar dual $D=T^*$ has a Hamiltonian cycle, then
$V(T)$ admits a partition $A\sqcup B$ such that both $T[A]$ and $T[B]$
are trees. In particular, $T$ has a proper vertex coloring with the
four colors $1,2,3,4$.
\end{lemma}

\begin{proof}
Let $C$ be a Hamiltonian cycle of $D$. Its embedded image is a simple
closed curve, and the Jordan curve theorem divides the sphere into
two open disks with common boundary $C$. Every vertex of $D$ lies on
$C$. The dual is cubic, because every face of $T$ is a triangle.
Consequently each vertex of $D$ is incident with exactly one edge
outside $C$, so $E(D)\setminus E(C)$ is a matching. Each such edge
has its two distinct endpoints on $C$ and its interior in exactly one
of the disks. Within either disk these edges are mutually disjoint
noncrossing chords.

Fix one disk and suppose that it contains $m$ chords. Insert its
chords one at a time. Before a chord is inserted, its interior lies
in one region of the current subdivision and its endpoints lie on
that region's boundary. That region is a disk, and the chord splits
it into two disks. This proves inductively that the final number of
regions is $m+1$ and that all regions are disks.

Form a graph $J$ with one vertex for each final region and one edge
for each chord, joining the two regions on its sides. This graph is
connected: starting with a single region, inserting a chord replaces
one region vertex by two joined vertices, distributing its former
adjacencies between them. Thus $J$ has $m+1$ vertices and $m$ edges
and is connected, so it is a tree. This also covers $m=0$, when $J$
is the one-vertex tree.

Every face of $D$ lies on exactly one side of $C$. Let $A$ and $B$
be the sets of vertices of $T$ corresponding, under duality, to the
faces on the respective sides. An edge of $D$ on $C$ separates a
face from each side; an edge outside $C$ is a chord and separates two
faces on the same side. Hence the two region-adjacency graphs just
constructed are exactly the induced graphs $T[A]$ and $T[B]$.
Both are trees. Two-color $T[A]$ with $1,2$ and two-color $T[B]$
with $3,4$. Edges within either part have differently colored ends,
and edges between the parts have ends in disjoint color sets.
This is the claimed proper vertex four-coloring.
\end{proof}

\begin{corollary}\label{dual:vertex-four}
Every finite simple sphere triangulation $T$ with $\Delta(T)\le6$
has a proper vertex four-coloring obtained without invoking the
Four-Color Theorem.
\end{corollary}

\begin{proof}
By the convention for sphere triangulations, $|V(T)|\ge4$.
The graph $T$ is a simple three-connected plane graph. By standard
polyhedral duality, its dual $D$ is also simple and three-connected.
It is cubic, and the face of $D$ corresponding to $v\in V(T)$ has
length $d_T(v)\le6$. Thus all hypotheses of
Theorem~\ref{dual:kardos} hold. Apply that theorem and then
Lemma~\ref{dual:two-trees}.
\end{proof}

The corollary supplies vertex colors only. No total-color extension
is implicit in it. The next sections prove the independent matching,
extension, and restoration statements needed to obtain eight colors
for vertices and edges together.

\section{Matchings covering prescribed vertices}
\label{core:matching-tools}

A matching \emph{covers} a vertex if one of its edges is incident with
that vertex. We shall need to cover specified vertices while leaving
the other vertices unconstrained. For a graph $F$, a set
$U\subseteq V(F)$, and $X\subseteq V(F)$, let
\[
 q_U(F-X)=\#\{C:\ C\text{ is an odd-order component of }F-X,
                       \ V(C)\subseteq U\}.
\]
Thus an odd component containing even one vertex outside $U$ is not
counted. The following version of the matching criterion will serve
both the edge-coloring argument and the construction of the core.

\begin{lemma}\label{core:required}
Let $F$ be a graph and $U\subseteq V(F)$. There is a matching
covering $U$ if and only if
\begin{equation}\label{core:required-condition}
 q_U(F-X)\leq |X|\qquad\text{for every }X\subseteq V(F).
\end{equation}
\end{lemma}

\begin{proof}
Suppose first that a matching covers $U$. In an odd component of
$F-X$ lying wholly in $U$, an internal matching cannot cover every
vertex. A matching edge must therefore leave that component and end
in $X$. Distinct counted components require distinct matched vertices
of $X$, proving \eqref{core:required-condition}.

For the converse, write $W=V(F)\setminus U$. Join every pair of
nonadjacent vertices in $W$, so that $W$ induces a clique.
If $|V(F)|$ is odd, add a new vertex $z$ adjacent to every
vertex of $W$. Denote the resulting even-order graph by $F^+$. Every
added edge has both ends outside $U$, the set of vertices required
to be covered by the matching.

We verify Tutte's perfect-matching criterion for $F^+$
\cite{tutte1954}. Take $Y\subseteq V(F^+)$, put
$X=Y\cap V(F)$, and let $\varepsilon=1$ if $z$ was added and belongs
to $Y$, and $\varepsilon=0$ otherwise. A component of $F-X$ lying
wholly in $U$ remains a component of $F^+-Y$. Every other old
component contains a vertex of $W\setminus X$. Since these vertices
form a clique, all such components merge into at most one component.
If $z$ remains but $W\setminus X$ is empty, the isolated vertex $z$
is itself the only possible additional component. It follows that
\[
 o(F^+-Y)\leq q_U(F-X)+1\leq |X|+1,
\]
where $o$ denotes the number of odd components. If $\varepsilon=1$,
then $|X|+1=|Y|$. If $\varepsilon=0$, then $|X|=|Y|$; the order of
$F^+$ is even, so $o(F^+-Y)$ and $|Y|$ have the same parity. In this
case the bound $o(F^+-Y)\leq |Y|+1$ improves to
$o(F^+-Y)\leq |Y|$. Thus Tutte's condition holds for every $Y$.

Choose a perfect matching of $F^+$ and discard its added edges. All
discarded edges avoid $U$, so the remaining edges form a matching
of $F$ covering every vertex of $U$.
\end{proof}

\section{Four-coloring extensions in subcubic graphs}
\label{ext:section}

The principal coloring tool preserves the colors of \emph{all}
vertices. This is the form needed when an edge decomposition is
colored one part at a time: the vertex coloring must remain proper
also on the parts whose edges have already been removed.

\begin{theorem}\label{ext:main}
Let $Q$ be a graph with $\Delta(Q)\leq3$. Every proper
vertex coloring $\phi:V(Q)\longrightarrow B=\{1,2,3,4\}$ extends
to a total coloring of $Q$ with the five colors $B\cup\{c\}$,
where $c\notin B$.
\end{theorem}

Theorem~\ref{ext:main} preserves the supplied coloring and requires no
planar embedding. Its proof first uses a matching to reduce a
bridgeless graph to paths and cycles, then removes terminal blocks.
For a terminal block $D$, write $rz$ for its unique bridge to the
rest of the graph, with $r\in V(D)$ and $z\notin V(D)$, and call
$rz$ its \emph{outgoing bridge}. An \emph{outside coloring} is a
total coloring of the graph outside $D$. It fixes the color of $z$
and the colors of the edges already incident with $z$, restricting
the available colors on $rz$. Propositions~\ref{ext:marked-port}
and~\ref{ext:single-port} show that every remaining legal color
can be realized inside $D$.

Throughout this section, $\phi$ denotes the supplied proper
vertex coloring into $B=\{1,2,3,4\}$; it need not use all four
colors. For an edge $xy$ of a graph carrying this fixed
coloring, set
\begin{equation}\label{ext:lists}
 A_\phi(xy)=B\setminus\{\phi(x),\phi(y)\},\qquad
 L_\phi(xy)=A_\phi(xy)\cup\{c\}.
\end{equation}
Properness of $\phi$ implies $|A_\phi(xy)|=2$ and
$|L_\phi(xy)|=3$. A proper edge coloring from the lists $L_\phi$
is exactly a total coloring extending $\phi$. We abbreviate
$A_\phi(e)$ and $L_\phi(e)$ to $A(e)$ and $L(e)$ when the fixed
coloring is clear.

\subsection{Two-color lists on paths and cycles}

\begin{lemma}\label{ext:degree-two}
Let $F$ have maximum degree at most two, and let
$\phi:V(F)\to B=\{1,2,3,4\}$ be a proper vertex coloring.
Then $F$ has a proper edge coloring from the lists $A_\phi$
in \eqref{ext:lists}. On a path, either color in the list of its
first edge may be prescribed.
\end{lemma}

\begin{proof}
Every nontrivial component is a path or a cycle. On a path, give the
first edge either permitted color. Each later edge has a two-element
list, so it has a color different from that of its predecessor.
This colors the whole path greedily.

Consider a cycle. If all its edge lists equal a fixed two-set $S$,
then the colors of the ends of every edge are the two elements of
$B\setminus S$. The vertex colors alternate, so the cycle is even.
Alternating the two members of $S$ colors its edges properly.

Otherwise, choose consecutive unequal edge lists and number the
edges cyclically as $e_1,\ldots,e_k$ so that
$A_\phi(e_1)\neq A_\phi(e_k)$. Give $e_1$ a color in
$A_\phi(e_1)\setminus A_\phi(e_k)$ and greedily color $e_2,\ldots,e_k$.
The color of $e_1$ does not belong to $A_\phi(e_k)$, so the last edge
automatically differs from the first. Isolated vertices need no
edge color.
\end{proof}

We next keep track of the only circumstance in which a prescribed
first edge color determines the last edge color. Fix any bijection
between $B$ and the additive group $\mathbb F_2^2$. All sums of colors
in this section refer to that identification. The four elements sum
to zero, and $x+x=0$ for every color $x$. The choice of bijection
does not impose any additional condition on $\phi$.

\begin{lemma}\label{ext:propagation}
Let $v_0v_1\cdots v_k$ be a path, where $k\geq1$, and let
$\phi$ be the fixed proper vertex four-coloring into $B$ from this
section.
Prescribe a color $x\in A_\phi(v_0v_1)$ on the first edge. For
$1\leq i\leq k$, let $R_i$ be the set of colors attainable on
$v_{i-1}v_i$ in a proper $A_\phi$-edge-coloring of the prefix path
$v_0v_1\cdots v_i$, with first color $x$. Each $R_i$ is nonempty.
If $|R_i|=2$, then $R_j=A_\phi(v_{j-1}v_j)$ for every $j\geq i$.
If $R_k=\{y\}$, then
\begin{equation}\label{ext:transport}
 y+\phi(v_k)=x+\phi(v_1).
\end{equation}
Consequently, if a color $t\in B$ satisfies
\begin{equation}\label{ext:avoidance}
 x+\phi(v_1)\neq t+\phi(v_k),
\end{equation}
there is such a path coloring whose last edge color differs from
$t$.
\end{lemma}

\begin{proof}
The recurrence for the attainable sets is
\[
 R_{i+1}=\{a\in A_\phi(v_iv_{i+1}):
                   a\neq b\text{ for some }b\in R_i\}.
\]
A two-element list always contains a color different from a given
one, so these sets are nonempty. If $|R_i|=2$, then for every
$a\in A_\phi(v_iv_{i+1})$ we have
$|R_i\setminus\{a\}|\geq |R_i|-1=1$.
Thus there exists $b\in R_i$ with $b\neq a$, and the recurrence
gives $a\in R_{i+1}$. Since this holds for every color in the
next list, $R_{i+1}=A_\phi(v_iv_{i+1})$, and the two-element
property persists.

If $R_k$ is a singleton, every preceding $R_i$ must therefore also
be a singleton, say $R_i=\{x_i\}$. For $1\leq i<k$, the recurrence
then gives $R_{i+1}=A_\phi(v_iv_{i+1})\setminus\{x_i\}$.
Since $A_\phi(v_iv_{i+1})$ has two members and
$R_{i+1}=\{x_{i+1}\}$, we must have
$A_\phi(v_iv_{i+1})=\{x_i,x_{i+1}\}$.
By the definition of $A_\phi$, these are precisely the two colors
in $B$ other than $\phi(v_i)$ and $\phi(v_{i+1})$.
Thus $x_i,x_{i+1},\phi(v_i),\phi(v_{i+1})$ are the four distinct
elements of $B$. Their sum in $\mathbb F_2^2$ is zero, so
$x_i+x_{i+1}=\phi(v_i)+\phi(v_{i+1})$.
Adding $x_i$ to both sides and using $x_i+x_i=0$ gives the first
identity below; adding $\phi(v_{i+1})$ gives the second:
\[
 x_{i+1}=x_i+\phi(v_i)+\phi(v_{i+1}),\qquad
 x_{i+1}+\phi(v_{i+1})=x_i+\phi(v_i).
\]
Iteration gives \eqref{ext:transport}; it also covers the case $k=1$.
Finally, if $R_k$ has two members, at least one differs from $t$.
If it has one member, \eqref{ext:transport} and
\eqref{ext:avoidance} show that this member differs from $t$. The
definition of attainability supplies a path coloring realizing the
chosen last color.
\end{proof}

At a vertex $u$ of a cycle whose edges are colored from the lists $A_\phi$, the two incident edge
colors are distinct members of $B\setminus\{\phi(u)\}$. We call
the remaining member of that three-set the \emph{missing color at
$u$}. Lemma~\ref{ext:cycle-missing} describes exactly the flexibility that will
be used at the attachment of a path.

\begin{lemma}\label{ext:cycle-missing}
Let $C$ be a cycle and let
$\phi:V(C)\to B$ be a proper vertex four-coloring.  Let $u$
be a vertex of $C$.
\begin{enumerate}
\item If the two neighbors of $u$ on $C$ have the same color $b$,
then every proper $A_\phi$-edge-coloring of $C$ has missing color $b$
at $u$.
\item If their colors $b,d$ differ, there is a proper
    $A_\phi$-edge-coloring with missing color $b$ at $u$, and another with
missing color $d$ there.
\end{enumerate}
\end{lemma}

\begin{proof}
Existence of a proper $A_\phi$-edge-coloring follows from
Lemma~\ref{ext:degree-two}. In the first case, the two incident
edge lists both equal $B\setminus\{\phi(u),b\}$. Their edges
must use its two members, leaving precisely $b$ missing.

In the second case, put $a=\phi(u)$ and let $e$ be the fourth
color, so that $a,b,d,e$ are distinct. Open the cycle at $u$ into
a path $v_0v_1\cdots v_k$, regarding the two copies
$v_0=v_k=u$ as different endpoints. Here
$\phi(v_1)=b$, $\phi(v_{k-1})=d$, and the first and last lists are
$\{d,e\}$ and $\{b,e\}$.

Prescribe $d$ on the first edge. If $e$ were unattainable on the
last edge, the last attainable set would be $\{b\}$. Equation
\eqref{ext:transport} would give $b+a=d+b$, contradicting
$a\neq d$. Thus the endpoint pair $(d,e)$ is
attainable. Its two entries differ, so identifying the endpoints
closes the cycle properly, and the missing color at $u$ is $b$.

Next prescribe $e$ on the first edge. If $b$ were unattainable at
the last edge, its attainable set would be $\{e\}$, whence
\eqref{ext:transport} would give $e+a=e+b$, contradicting
$a\neq b$. The pair $(e,b)$ is therefore attainable
and closes to a cycle coloring missing $d$ at $u$. The temporary
duplication of $u$ introduced no additional condition: closing the
path makes its first and last edges adjacent, so they must differ
in color, as they do in both constructions.
\end{proof}

\subsection{Matching states in bridgeless graphs}

We write $V_3(D)=\{v\in V(D):d_D(v)=3\}$. In this section, a
``cubic vertex of $D$'' means a vertex in $V_3(D)$, and a graph is
called cubic when every one of its vertices has degree three. An edge is a
\emph{bridge} if its deletion increases the number of connected
components. A graph is \emph{bridgeless} if it has no bridge.

\begin{lemma}\label{ext:bridgeless-matching}
Let $D$ be a connected bridgeless graph whose vertex
degrees are two or three. Then $D$ has a matching covering $V_3(D)$.
Consequently every prescribed proper vertex four-coloring
$\phi:V(D)\to B$ extends to a total coloring with $B\cup\{c\}$.
\end{lemma}

\begin{proof}
Apply Lemma~\ref{core:required} with required set $V_3(D)$. For
$X=\varnothing$, its count is zero: if some vertex has degree two,
the unique component contains an optional vertex; otherwise $D$ is
cubic and has even order.

Now let $X\neq\varnothing$, and let $C$ be an odd component of
$D-X$ lying wholly in $V_3(D)$. Write $\delta_D(C)$ for its edge
cut. The degree sum gives
\[
 |\delta_D(C)|=3|V(C)|-2|E(C)|,
\]
which is odd. The cut is nonempty by connectedness, and cannot
have size one because $D$ is bridgeless. It has at least three
edges, all ending in $X$. Summing over the counted components gives
\[
 3q_{V_3(D)}(D-X)\leq3|X|.
\]
Lemma~\ref{core:required} supplies a matching $N$ covering
$V_3(D)$.

Color the edges of $N$ with $c$. The graph $D-N$ has maximum
degree at most two, so Lemma~\ref{ext:degree-two} colors its edges
from $A_\phi$. Each vertex meets at most one edge of color $c$, and
that color occurs on neither the other edges nor the vertices.
The two edge colorings together extend the prescribed $\phi$.
\end{proof}

For attachment to an already colored graph, a slightly more
precise statement is needed. Lemma~\ref{ext:two-states} distinguishes between
leaving a particular degree-two vertex uncovered and covering it.

\begin{lemma}\label{ext:two-states}
Let $D$ satisfy the hypotheses of
Lemma~\ref{ext:bridgeless-matching}, and let $r\in V(D)$ be a vertex
of degree two.
\begin{enumerate}
\item There is a matching avoiding $r$ and covering $V_3(D)$.
\item If $D$ has a degree-two vertex other than $r$, there is a
matching covering $V_3(D)\cup\{r\}$.
\end{enumerate}
\end{lemma}

\begin{proof}
For a set $C\subseteq V(D)$, write
$\delta_D(C)=\{xy\in E(D):x\in C,\ y\notin C\}$; this is the
\emph{$D$-cut} of $C$.  Recall from the definition of
$q_U$ that a counted odd component is an odd-order component whose
vertex set is contained in the required set $U$.

For the first assertion apply Lemma~\ref{core:required} to $D-r$
with required set $V_3(D)$. For $X\subseteq V(D-r)$, a counted
component $C$ is a component of $(D-r)-X=D-(X\cup\{r\})$.
We use its cut $\delta_D(C)$ in the original graph $D$, including
any edges from $C$ to $r$. Every vertex of $C$ has degree three
in $D$, so this cut is odd,
nonempty and cannot have size one because $D$ is bridgeless; hence
it has size at least three.  Every edge of this $D$-cut has its
endpoint outside $C$ in $X\cup\{r\}$, whose total degree in $D$ is
at most $3|X|+2$. Thus
\[
 3q_{V_3(D)}((D-r)-X)\leq3|X|+2.
\]
Since the count is an integer, it is at most $|X|$. The resulting
matching in $D-r$ has the required properties.

For the second assertion use $D$ itself and required set
$V_3(D)\cup\{r\}$. If $X=\varnothing$, the other degree-two
vertex is optional, so the count is zero. If $X\neq\varnothing$,
each counted component not containing $r$ consists entirely of
vertices of degree three in the original graph $D$ and has an odd
$D$-cut of size at least three.  At most one counted component
contains $r$; its $D$-cut is nonempty and even, and has size at least
two. Consequently
\[
 3q_{V_3(D)\cup\{r\}}(D-X)-1\leq3|X|.
\]
Integrality again gives the required inequality.
\end{proof}

\begin{proposition}
\label{ext:marked-port}
Let $D$ be connected and bridgeless, with degrees
two or three. Let $r$ have degree two and suppose there is another
degree-two vertex. Add a new vertex $z$ and the edge $rz$, and
prescribe a proper vertex four-coloring
$\phi:V(D+rz)\to B$. For each $t\in L_\phi(rz)$, there is a proper
$L_\phi$-edge-coloring of $D+rz$ in
which $rz$ has color $t$.
\end{proposition}

\begin{proof}
If $t=c$, choose the matching in
Lemma~\ref{ext:two-states}(1), color its edges $c$, and color
the residual graph from $A_\phi$ using
Lemma~\ref{ext:degree-two}. No internal edge of color $c$ meets
$r$, so adding $rz$ with color $c$ causes no conflict.

If $t\in A_\phi(rz)$, choose the matching in
Lemma~\ref{ext:two-states}(2) and again color it $c$. In the
residual graph, $r$ has degree one. Its component is a path, and
its first edge $e$ has a color in $A_\phi(e)\setminus\{t\}$. Prescribe
such a color and extend along the path by
Lemma~\ref{ext:degree-two}, and color the other components by the
same lemma. The three edge colors at $r$ are now $c$, $t$, and a
member of $B$ different from $t$. They are distinct. Membership
of $t$ in $L_\phi(rz)$ also excludes the colors of both endpoints.
\end{proof}

The outside vertex $z$ has degree one in the local graph of
Proposition~\ref{ext:marked-port}. In an application to a larger
graph, we first choose $t$ to avoid the colored edges at $z$.

\subsection{One degree-two vertex}

It remains to handle a bridgeless block with only one degree-two
vertex. The missing second optional vertex is replaced by the
freedom to choose which degree-three vertex a near-perfect matching
leaves uncovered.

\begin{lemma}\label{ext:near-cubic-structure}
Let $D$ be a connected bridgeless graph, and let
$r\in V(D)$ satisfy $d_D(r)=2$ and $d_D(v)=3$ for every
$v\in V(D)\setminus\{r\}$.
Then $D-r$ is connected, $D$ is not bipartite, and $D-u$ has a
perfect matching for every $u\in V(D)$.
\end{lemma}

\begin{proof}
If a vertex $v$ were a cut vertex, every component of $D-v$ would
have at least two edges to $v$: a component with just one such
edge would make that edge a bridge. Two components would require
$d_D(v)\geq4$, a contradiction. Thus $D$ has no cut vertex, and
in particular $D-r$ is connected.

If $D$ were bipartite with parts $U,W$ and $r\in U$, counting
edges from the two parts would give $3|U|-1=|E(D)|=3|W|$,
which is impossible modulo three.

The degree sum $3|V(D)|-1$ is even, so $|V(D)|$ is odd. Fix
$u\in V(D)$ and $X\subseteq V(D)\setminus\{u\}$. Let $q$ be
the number of odd components of $D-(X\cup\{u\})$. Every such
component has a nonempty cut of size at least two. Moreover,
\[
 |\delta_D(C)|=3|V(C)|-\mathbf1_{r\in V(C)}-2|E(C)|.
\]
Unless $C$ contains $r$, its cut is odd and therefore has at
least three edges. There is at most one exception. All edges in
these $D$-cuts have their endpoint outside the corresponding
component in $X\cup\{u\}$, giving
\[
 3q-1\leq\sum_C|\delta_D(C)|\leq3|X|+3.
\]
Hence $q-|X|\leq4/3$. The graph $D-u$ has even order, so
$q-|X|$ is even, and it follows that $q\leq |X|$. This verifies
Tutte's criterion for $D-u$ for every $X$, proving the final
assertion.
\end{proof}

Let $D,r$ satisfy the hypotheses of
Lemma~\ref{ext:near-cubic-structure}, and fix a proper vertex
four-coloring $\phi:V(D)\to B=\{1,2,3,4\}$.
Use the identification from Section~\ref{prelim:coloring} of $B=\{1,2,3,4\}$ with
$\mathbb F_2^2$ that assigns the vectors
$(0,0),(0,1),(1,0),(1,1)$ to the four colors in that order.
Addition of color vectors is coordinatewise modulo two.
Fix a nonzero vector $\kappa\in\mathbb F_2^2$. The pairs
$\{a,a+\kappa\}$, for $a\in\mathbb F_2^2$, partition the four
colors into two disjoint pairs, since $(a+\kappa)+\kappa=a$.
For example, if $\kappa=(0,1)$, the two pairs are
\[
 \{(0,0),(0,1)\},\qquad \{(1,0),(1,1)\},
\]
which correspond to $\{1,2\}$ and $\{3,4\}$ in the original
color labels. For two distinct colors $b,d$, belonging to the
same pair is equivalent to $d=b+\kappa$, or equivalently
$b+d=\kappa$; this is the relation used in Lemma~\ref{ext:witness}.
An edge will be
called \emph{crossing for this partition} if its endpoint colors
belong to different pairs. This use of ``crossing'' concerns only
a partition of the colors, and has no geometric meaning.

\begin{lemma}\label{ext:witness}
Let $D$ be a connected bridgeless graph with exactly
one vertex $r$ of degree two and with every other vertex of degree
three.
Fix a proper vertex
four-coloring $\phi:V(D)\to B=\{1,2,3,4\}$, and use the above
identification of $B$ with $\mathbb F_2^2$.
For each nonzero $\kappa\in\mathbb F_2^2$, there is a vertex
$u\in V(D)\setminus\{r\}$ with one of the following neighborhood
color patterns under this fixed coloring $\phi$:
\begin{enumerate}
\item its three neighbors receive pairwise distinct colors under $\phi$;
\item the colors assigned by $\phi$ to its neighbors are exactly
two distinct colors $b,d\in B$, with
$b+d\neq\kappa$.
\end{enumerate}
\end{lemma}

\begin{proof}
Fix $\kappa\neq0$ and suppose that neither alternative occurs
under $\phi$ at any vertex other than $r$. Every vertex other than $r$
then has a monochromatic neighborhood, or its neighborhood uses
exactly two colors $b,d$ satisfying $b+d=\kappa$.

Write $P_0,P_1$ for the two pairs determined by $\kappa$, and fix
$v\in V(D)\setminus\{r\}$. If the neighborhood of $v$ is
monochromatic with color $b$, every incident edge has endpoint
colors $\phi(v)$ and $b$. All three edges are noncrossing if
these colors lie in the same pair, and crossing otherwise.

In the other case, the neighbors of $v$ receive exactly two
distinct colors $b,d$ with $b+d=\kappa$. Then
$d=b+(b+d)=b+\kappa$, so $\{b,d\}=\{b,b+\kappa\}$ is one
of the two pairs $P_0,P_1$.
Since $\phi$ is a proper vertex coloring, $\phi(v)$ differs from
both $b$ and $d$. It therefore belongs to the other pair, and
all three edges incident with $v$ are crossing.
In either case, each vertex $v\neq r$ has a common crossing
status for its incident edges.

For an edge $xy$ of $D-r$, the common crossing statuses at $x$
and $y$ agree, since both equal the status of $xy$.
Since $D-r$ is connected by
Lemma~\ref{ext:near-cubic-structure}, the status is common to all
vertices other than $r$. Every edge of $D$ meets such a vertex,
so all edges have that same status. If all are crossing, the
two pairs give a bipartition of $D$. If all are noncrossing,
connectedness forces all vertex colors to lie in a single pair,
and properness of the vertex coloring again makes $D$ bipartite.
Both conclusions contradict Lemma~\ref{ext:near-cubic-structure}.
\end{proof}

The hypothesis that $r$ is the only degree-two vertex cannot simply
be replaced by the condition that all degrees are two or three.
For example, consider $K_{2,3}$ with parts $\{x_1,x_2\}$ and
$\{s,y_1,y_2\}$, and give these vertices the coloring
\[
 \psi(x_1)=1,\quad \psi(x_2)=2,\quad
 \psi(s)=\psi(y_1)=3,\quad \psi(y_2)=4.
\]
This is a proper vertex four-coloring of a connected
bridgeless graph of maximum degree three. For $\kappa=(0,1)$,
every neighborhood uses either $\{1,2\}$ or $\{3,4\}$; both pairs
have color-vector sum $\kappa$. Thus neither alternative of
Lemma~\ref{ext:witness} occurs. In the proof above, the stronger
degree assumption supplies the non-bipartiteness of $D$ through
Lemma~\ref{ext:near-cubic-structure}.

Choose $u\neq r$ and a perfect matching $N$ of $D-u$. In the
residual graph $F=D-N$,
\begin{equation}\label{ext:lollipop-degrees}
 d_F(u)=3,\qquad d_F(r)=1,\qquad d_F(v)=2
       \quad(v\notin\{u,r\}).
\end{equation}
The vertices $u,r$ lie in the same component, since each component
has an even number of vertices of odd degree. Every other component
is a cycle. The component containing $u,r$ consists of a simple
cycle $C$ through $u$ and a nontrivial simple path $P$ from $u$ to $r$,
meeting that cycle only at $u$.

To verify the cycle-and-path description, start at
$r$ along its only edge and continue through vertices of degree
two. The walk cannot revisit a vertex before reaching $u$: a
closed portion attached to the preceding path would give a
degree-three vertex other than $u$, and $r$ itself has degree one.
Finiteness therefore makes the walk reach $u$ along a simple path.
The other two edges at $u$ continue through degree-two vertices
and return to $u$, forming a simple cycle. All degrees in the
component are now exhausted. Write $P$ as
$u=v_0,v_1,\ldots,v_k=r$, where $k\geq1$.
Figure~\ref{fig:port} illustrates this cycle-and-path component.
It also shows the edge $rz$ to a new outside vertex $z$; this edge
and its prescribed color $t$ will be considered in
Proposition~\ref{ext:single-port}.

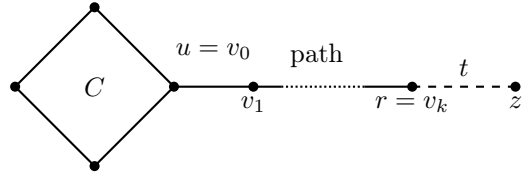
\begin{figure}[!htbp]
\centering
\begin{tikzpicture}[scale=1.05, every node/.style={font=\small}]
 \coordinate (u) at (0,0);
 \coordinate (a) at (-1,1);
 \coordinate (b) at (-2,0);
 \coordinate (d) at (-1,-1);
 \coordinate (vone) at (1,0);
 \coordinate (r) at (3,0);
 \coordinate (z) at (4.3,0);
 \draw[thick] (u)--(a)--(b)--(d)--(u)--(vone);
 \draw[thick] (vone)--(1.35,0);
 \draw[densely dotted,thick] (1.35,0)--(2.4,0);
 \draw[thick] (2.4,0)--(r);
 \draw[dashed,thick] (r)--node[above] {$t$}(z);
 \foreach \pt in {u,a,b,d,vone,r,z} \fill (\pt) circle (1.8pt);
 \node at (0.5,0.45) {$u=v_0$};
 \node[below] at (vone) {$v_1$};
 \node[below] at (r) {$r=v_k$};
 \node[below] at (z) {$z$};
 \node at (-1,0) {$C$};
 \node[above] at (1.8,0.12) {path};
\end{tikzpicture}
\caption{The component of $D-N$ containing $u$ and $r$, shown
schematically.  Its cycle and path can have other lengths.  The dashed
edge is the outgoing bridge, whose color $t$ is prescribed.  Edges of
the matching $N$ and the other cycle components are omitted.
Proposition~\ref{ext:single-port} chooses the exposed vertex $u$
so that the last path-edge color can be made different from $t$.}
\label{fig:port}
\end{figure}
\FloatBarrier

\begin{lemma}
\label{ext:initial-colour}
Let $C$ be a simple cycle and let $P$ be a simple path with vertex
sequence $u=v_0,v_1,\ldots,v_k=r$, where $k\geq1$ and
$V(C)\cap V(P)=\{u\}$.
Fix a proper vertex four-coloring
$\phi:V(C\cup P)\to B=\{1,2,3,4\}$, and for each edge $xy$ of
$C\cup P$ put $A_\phi(xy)=B\setminus\{\phi(x),\phi(y)\}$.
Use the above identification of $B$ with $\mathbb F_2^2$ for
sums of colors. All references to the neighbors of $u$ below
mean its three neighbors in $C\cup P$, with their colors under
this fixed $\phi$. Then the following hold.
\begin{enumerate}
\item If the three neighbors of $u$ have distinct colors, either
member of $A_\phi(uv_1)$ can be assigned to $uv_1$ in a proper
$A_\phi$-edge-coloring of $C\cup\{uv_1\}$.
\item If those neighbors use exactly two distinct colors $b,d\in B$,
there is a proper $A_\phi$-edge-coloring of $C\cup\{uv_1\}$
such that the color $x$ of $uv_1$ satisfies
\begin{equation}\label{ext:first-difference}
 x+\phi(v_1)=b+d.
\end{equation}
\end{enumerate}
\end{lemma}

Lemma~\ref{ext:initial-colour} colors the cycle and the first path
edge; propagation will color the remaining path edges.

\begin{proof}
In the first case, let $b,d$ be the two cycle-neighbor colors and
$e=\phi(v_1)$ the path-neighbor color. Since these are distinct
and all differ from $\phi(u)$, the list of $uv_1$ is $\{b,d\}$.
Apply Lemma~\ref{ext:cycle-missing}(2) to the restricted coloring
$\phi|_{V(C)}$: either $b$ or $d$ can be made missing at $u$
on the cycle. Give $uv_1$ this missing color.

In the second case, suppose first that the cycle-neighbors have
the same color. Interchanging the names $b,d$ if necessary, call
that color $b$; the path-neighbor then has color $d$.
Choose a proper $A_\phi$-edge-coloring of the cycle using
Lemma~\ref{ext:degree-two}. By Lemma~\ref{ext:cycle-missing}(1),
it misses $b$ at $u$, so set $x=b$.
Otherwise the cycle-neighbors have
different colors $b,d$, and the path-neighbor has one of them,
say $b$. Use Lemma~\ref{ext:cycle-missing}(2) to color the cycle
with $d$ missing at $u$, and set $x=d$.
In either situation $x$ is the neighbor color different from
$\phi(v_1)$, proving \eqref{ext:first-difference}. The chosen
color belongs to the edge list and differs from the colors of both
cycle edges incident with $u$, by construction.
\end{proof}

\begin{proposition}
\label{ext:single-port}
Let $D$ have one vertex $r$ of degree two and all other vertices
of degree three, and suppose that $D$ is connected
and bridgeless. Add a new vertex $z$ and the edge $rz$, and
prescribe a proper vertex four-coloring
$\phi:V(D+rz)\to B$. For every $t\in L_\phi(rz)$, there is a proper
$L_\phi$-edge-coloring of $D+rz$
assigning color $t$ to $rz$.
\end{proposition}

\begin{proof}
First suppose $t=c$. By
Lemma~\ref{ext:near-cubic-structure}, choose a perfect matching
$N$ of $D-r$ and give it color $c$. The graph $D-N$ has degree
two everywhere, including at $r$, so it is a union of cycles.
Color these cycles from $A_\phi$ by Lemma~\ref{ext:degree-two}, and
give $rz$ color $c$. There is no internal matching edge at $r$,
so this is a proper extension.

Now suppose $t\in A_\phi(rz)$, and put
$\kappa=\phi(r)+t\neq0$.
Choose a vertex $u\neq r$ supplied by
Lemma~\ref{ext:witness}, applied to the restricted coloring
$\phi|_{V(D)}$ and this $\kappa$. Write
$T_\phi(v)=\{\phi(w):w\in N_D(v)\}$ for the set of neighbor
colors of a vertex $v$. The selected vertex $u$ satisfies
\begin{equation}\label{ext:chosen-witness}
\begin{aligned}
 &|T_\phi(u)|=3,\\
 &\text{or}\quad T_\phi(u)=\{b,d\}
       \quad\text{with }b\neq d\text{ and }b+d\neq\kappa.
\end{aligned}
\end{equation}
Choose a perfect matching $N$ of $D-u$, which exists for this
selected vertex by Lemma~\ref{ext:near-cubic-structure}.
The component of $D-N$ containing
$u,r$ has the cycle-and-path description above. All three original
edges at $u$ remain, so the neighborhood color pattern used to
choose $u$ is unchanged.
In applying Lemma~\ref{ext:initial-colour} below, we use the
restriction of the prescribed $\phi$ to this component $C\cup P$.

Suppose first that the three neighbors of $u$ have pairwise
distinct colors. Write $e_i=v_{i-1}v_i$ and
$A_i=A_\phi(e_i)$ for $1\leq i\leq k$. Each $A_i$ has exactly two
members, since $\phi$ gives distinct colors to the endpoints of
$e_i$. By Lemma~\ref{ext:initial-colour}(1), for each choice
$x\in A_1$ there is a proper $A_\phi$-edge-coloring of $C$
in which neither of the two cycle edges at $u$ has color $x$.
The first path-edge color can therefore range over all of $A_1$;
the compatible cycle coloring may depend on that choice.

Construct the path coloring backwards. First choose
$a_k\in A_k\setminus\{t\}$. Then, successively for
$i=k-1,k-2,\ldots,1$, choose
\begin{equation}\label{ext:backward-choice}
 a_i\in A_i\setminus\{a_{i+1}\}.
\end{equation}
Each choice excludes at most one member of a two-element list. Thus
\[
 a_i\in A_i\quad(1\leq i\leq k),\qquad
 a_i\neq a_{i+1}\quad(1\leq i<k),\qquad a_k\neq t.
\]
Define $\psi(e_i)=a_i$. Membership in $A_i$ excludes both
endpoint vertex colors, and the inequalities between consecutive
edge colors make $\psi$ a proper $A_\phi$-edge-coloring of the
path. Its first color $x=a_1$ belongs to $A_1$, so
Lemma~\ref{ext:initial-colour}(1) supplies a proper
$A_\phi$-edge-coloring $\gamma$ of $C$ in which both cycle
edges incident with $u$ have colors different from $x$.
Since $V(C)\cap V(P)=\{u\}$, these are exactly the conditions
needed to combine $\gamma$ and $\psi$ into a proper
$A_\phi$-edge-coloring of $C\cup P$. Its last path-edge color
differs from $t$. This also covers $k=1$: choose
$a_1\in A_1\setminus\{t\}$ and then choose the compatible cycle
coloring.

The same construction works for each $s\in A_k$: a proper path
coloring $\psi_s$ with $\psi_s(e_k)=s$ and a compatible cycle
coloring can be chosen together. Fixing earlier edges may restrict
the last color; for example, if $A_{k-1}=A_k=\{1,2\}$ and the
preceding edge has color $1$, the last edge must have color $2$.

Suppose now that the selected vertex $u$ has exactly two neighbor
colors $b,d$. By its choice in \eqref{ext:chosen-witness},
$b+d\neq\kappa=\phi(r)+t$. By
Lemma~\ref{ext:initial-colour}(2), color the cycle and first path
edge so that its color $x$ satisfies
\[
 x+\phi(v_1)=b+d\neq\kappa=t+\phi(r).
\]
Lemma~\ref{ext:propagation} extends this fixed first color along
the path with final color different from $t$, leaving the cycle
coloring unchanged. Thus both cases give an $A_\phi$-edge-coloring
of the entire component with the desired avoidance at $r$.

Color all other components of $D-N$, which are cycles, from $A_\phi$.
Give $N$ color $c$ and $rz$ color $t$. At every vertex other
than $r$, the residual edges are properly colored from $A_\phi$ and
the possible matching edge has the distinct color $c$. At $r$
there is exactly one matching edge and one residual edge. Their
colors are $c$ and a member of $B$ different from $t$, respectively;
adding the bridge of color $t$ is therefore proper. Its color
avoids both endpoint vertex colors by the definition of $A_\phi(rz)$,
and $z$ has no other incident edge in the local graph. Every
total-coloring condition is satisfied.
\end{proof}

The choice of $u$ depends on the desired bridge color. This is
essential: the color sum transported along a path may be forced,
but Lemma~\ref{ext:witness} supplies an exposed vertex so that the
forced value, if there is one, cannot obstruct that bridge color.
For example, under the fixed identification with $\mathbb F_2^2$,
$1+2=3+4$. If $t=3$ and $\phi(r)=4$, a vertex whose neighbor
color set is exactly $\{1,2\}$ cannot serve as the selected $u$.
Lemma~\ref{ext:witness} guarantees another vertex satisfying
one of the two alternatives in \eqref{ext:chosen-witness}.

\subsection{Removing terminal blocks}

Let $Q$ be a connected graph, and let
$\mathcal B(Q)$ denote its set of bridges. Delete these edges but
keep every vertex; thus the resulting spanning subgraph is
$Q^\circ=\bigl(V(Q),E(Q)\setminus\mathcal B(Q)\bigr)$.
Write $D_1,\ldots,D_m$ for the connected components of $Q^\circ$,
including isolated vertices. We call these components the
\emph{bridge blocks} of $Q$.

We first show that every $D_i$ is bridgeless. Let $e=xy$ be any
edge of $D_i$. Since $e$ is not a bridge of $Q$, there is a simple
$x$--$y$ path in $Q-e$. Together with $e$, this path forms a
cycle $C$ in $Q$. No edge of $C$ can be a bridge of $Q$: after
any one of its edges is removed, the rest of $C$ still connects
the two endpoints of that edge. Consequently every edge of $C$
survives in $Q^\circ$. Since $C$ is connected and contains $e$,
it lies entirely in the component $D_i$. Hence $e$ lies on a
cycle in $D_i$ and is not a bridge of $D_i$. This holds for
every edge of $D_i$. A block with at least two vertices therefore
has minimum degree at least two, since an edge incident with a
degree-one vertex would be a bridge.

To define the \emph{block tree} $T$, replace each
block $D_i$ by a single vertex $d_i$. For each original bridge
$e=xy$, with $x\in V(D_i)$ and $y\in V(D_j)$, put one edge
$\bar e=d_i d_j$ in $T$, keeping a distinct edge for each bridge
at this stage. The indices $i$ and $j$ are different: otherwise
a path from $x$ to $y$ inside $D_i$ would avoid $e$, contradicting
the assumption that $e$ is a bridge. Thus $T$ has no loops.
This contraction discards internal block edges and retains the
original bridges between the contracted vertices.
A path in $Q$ gives a walk between the corresponding
vertices of $T$, so $T$ is connected.

To prove that $T$ is a tree, fix an original bridge $e=xy$ and
its corresponding edge $\bar e=d_i d_j$. Suppose that
$T-\bar e$ contains a path from $d_i$ to $d_j$. Each edge of
this path represents a bridge of $Q$ other than $e$.
In the first block, join $x$ to the endpoint of the first such
bridge; in each intermediate block, join the endpoint of the
arriving bridge to the endpoint of the departing bridge; and
in the last block, join the final endpoint to $y$. All these
internal paths exist because the blocks are connected. A path
of length zero suffices if its two endpoints coincide.
Concatenating these internal paths with the intervening bridges
gives an $x$--$y$ walk in $Q-e$, and hence an $x$--$y$ path in
$Q-e$. This contradicts the fact that $e$ is a bridge of $Q$.
Thus every edge of $T$ is a bridge. In particular, $T$ has no
parallel edges (one would provide an alternative path for the
other), and it has no cycle. Being connected, $T$ is a tree.

For example, if two vertex-disjoint triangles are joined by a
single edge, that joining edge is the only bridge. The two
triangles are the bridge blocks, and $T$ consists of two
vertices joined by one edge. More generally, a \emph{leaf block}
is a block $D_i$ for which $d_i$ is a leaf of $T$. Such a block
has exactly one edge of $Q$ joining it to the remaining blocks;
this edge is its unique outgoing bridge.

\begin{proof}[Proof of Theorem~\ref{ext:main}]
Treat connected components independently. For a connected graph
$Q$, induct on the number of bridge blocks. If there is one block,
then $Q$ is an isolated vertex or a nontrivial bridgeless graph.
In the latter case its minimum degree is at least two, and
Lemma~\ref{ext:bridgeless-matching} applies.

Suppose there are at least two bridge blocks, and let $D$ be a
leaf block of the block tree. Write $rz$ for its unique outgoing
bridge, with $r\in V(D)$ and $z\notin V(D)$. Delete all vertices
of $D$, so $Q'=Q-V(D)$. The tree obtained from $T$ by deleting
the vertex corresponding to $D$ is still connected. The
remaining blocks and their joining bridges therefore form a
connected graph $Q'$. Every internal edge of a remaining block
still lies on a cycle in that block, and every surviving bridge
of $Q$ is still a bridge in $Q'$: a path avoiding it in $Q'$
would also be such a path in $Q$. Thus the bridge blocks of
$Q'$ are exactly those of $Q$ other than $D$. In particular,
$Q'$ has one fewer bridge block. It is also simple and
subcubic, since it is a subgraph of $Q$. The restriction of
$\phi$ is still a proper vertex four-coloring of $Q'$, so
induction gives a total coloring of $Q'$ using the
color set $B\cup\{c\}$ while preserving this vertex coloring.

At most two edges of $Q'$ meet $z$, whereas $L_\phi(rz)$ has three
members. Choose a color $t\in L_\phi(rz)$ different from their edge
colors. If $D$ is a singleton, assign $t$ to $rz$ and we are
done. Otherwise $D$ is connected and bridgeless with minimum
degree at least two. Because $r$ has the outgoing bridge and
$Q$ is subcubic, $d_D(r)=2$. Every other vertex has degree two
or three inside $D$. If some other vertex has degree two,
Proposition~\ref{ext:marked-port} extends the chosen bridge color
$t$ over $D$; if not, Proposition~\ref{ext:single-port} does so.

These extensions preserve every vertex color and change no edge
of $Q'$. The only edge between the two pieces is $rz$. At $z$
its color was chosen to avoid the exterior edges, and at $r$
its color is compatible with the internal edges by the applicable
extension proposition. Its two endpoint colors are excluded by its
list. All other adjacency and incidence constraints involve elements
within one colored piece. The two colorings therefore give the
required coloring of $Q$.
\end{proof}

Deleting a leaf block preserves the class of subcubic graphs,
including the lower degrees created at the outside endpoint. Each
restoration prescribes a color on only one outgoing bridge.

\begin{corollary}\label{ext:three-matchings}
Let $H$ be a graph with a proper vertex coloring
$\phi:V(H)\to\{1,2,3,4\}$. Suppose $M_1,M_2,M_3$ are pairwise
edge-disjoint matchings such that
\[
 \Delta\bigl(H-(M_1\cup M_2\cup M_3)\bigr)\leq3.
\]
Then $\phi$ extends to a total coloring of $H$ with eight colors.
\end{corollary}

\begin{proof}
Give every edge of $M_i$ color $4+i$, for $i=1,2,3$. Apply
Theorem~\ref{ext:main} to the residual graph with the restricted
vertex coloring and with $c=8$. Its edges use only
$\{1,2,3,4,8\}$. The matching colors are different from all
residual-edge and vertex colors; within each matching the edges
are nonadjacent. The vertex coloring was proper on all of $H$
and has not changed. Hence all adjacency and incidence conditions
hold in the combined coloring.
\end{proof}

\section{Three matchings in the core}
\label{core:section}

We now construct the decomposition required by
Corollary~\ref{ext:three-matchings}. The first two matchings come
from a spanning path. The remaining matching must cover a
specified set of degree-four vertices. The boundary-walk inequality
in Lemma~\ref{core:curvature} identifies certain two-edge cuts as
pairs of bridges in Lemma~\ref{core:two-cut}. Those bridges yield
the count required by Lemma~\ref{core:required}, even when an
induced region has repeated vertices or edges on its boundary.
The connectivity hypothesis in the following theorem uses the
definition of $k$-vertex-connectivity in Section~\ref{prelim:graphs}
with $k=4$.

\begin{theorem}\label{core:main}
Let $H$ be a four-connected sphere triangulation in
which every vertex has degree five or six. Every proper vertex
coloring $\phi:V(H)\to\{1,2,3,4\}$ extends to a total coloring
with colors $\{1,2,3,4,5,6,7,8\}$.
\end{theorem}

\begin{corollary}\label{core:four-palette}
Every graph $H$ satisfying the hypotheses of
Theorem~\ref{core:main} has a total eight-coloring in which all
vertices receive colors from one fixed four-element color set.
\end{corollary}

\begin{proof}
Corollary~\ref{dual:vertex-four} supplies the proper vertex coloring
required by Theorem~\ref{core:main}, using Kardo\v{s}'s theorem and
the explicit construction in Lemma~\ref{dual:two-trees}.
\end{proof}

Theorem~\ref{core:main} takes a supplied coloring as input. Its
matching and extension proof is independent of the method used to
produce that coloring. In this revision, Corollary~\ref{core:four-palette}
depends on the Hamiltonian-dual construction of Section~\ref{dual:section},
and hence on Kardo\v{s}'s published computer-assisted theorem, not on the
Four-Color Theorem. None of the arguments below requires a new computation.

\subsection{Deleting a spanning path}

A \emph{linear forest} is a graph each of whose components is a
path, possibly an isolated vertex. We first record a connectivity
property of triangulations.

\begin{lemma}\label{core:linear-forest}
If $T$ is a sphere triangulation on at least four
vertices and $R$ is a linear forest contained in $T$, then
$T-E(R)$ is connected.
\end{lemma}

\begin{proof}
Suppose, for a contradiction, that $T-E(R)$ is disconnected,
and let $S$ be a nonempty proper union of its component vertex
sets. Every edge of $\delta_T(S)$ joins $S$ to its
complement and must belong to $R$; otherwise it would remain
in $T-E(R)$ and join two of its components. The set
$\delta_T(S)$ is nonempty because $T$ is connected. Define
the spanning subgraph $J=\bigl(V(T),\delta_T(S)\bigr)$.
Thus $J$ retains all vertices of $T$ and only the edges joining
$S$ to $V(T)\setminus S$; some vertices may be isolated.
Since $E(J)\subseteq E(R)$ and $R$ is a linear forest,
$d_J(w)\leq2$ for every $w\in V(T)$. This degree counts only
the cut edges incident with $w$.

Take an arbitrary edge $uv\in\delta_T(S)$, with $u\in S$
and $v\notin S$. Write $uvx$ and $uvy$ for the two
triangular faces incident with $uv$. Their third vertices
$x,y$ are distinct.
Indeed, if their third vertices coincided, the same triangle
would bound both faces; the sphere would then contain no further
vertex, contrary to $|V(T)|\geq4$.

If $x,y$ both belonged to $S$, the three edges $vu,vx,vy$ would
belong to $\delta_T(S)$, giving $d_J(v)\geq3$. If both lay
outside $S$, the three edges $uv,ux,uy$ would belong to
$\delta_T(S)$, giving $d_J(u)\geq3$. Both possibilities
contradict the degree bound for $J$. Thus exactly one of $x,y$
belongs to $S$. Interchanging their names if necessary, assume
$x\in S$ and $y\notin S$. The two triangular faces then give
$vx,uy\in\delta_T(S)$ and $ux,vy\notin\delta_T(S)$.
The vertex $u$ is therefore incident with the two distinct cut
edges $uv,uy$, and $v$ with the two distinct cut edges $uv,vx$.
Since both degrees in $J$ are at most two,
$d_J(u)=d_J(v)=2$.

For any non-isolated vertex $w$ of $J$, choose an edge of $J$
incident with $w$ and apply the preceding argument to that edge.
It follows that $d_J(w)=2$. Consequently every vertex of $J$
has degree either zero or two, and at least one edge is present.

Choose a simple path $w_0w_1\cdots w_k$ in $J$ with the maximum
possible number of vertices. Such a path exists because $J$ is
finite, and $k\geq1$ because $J$ has an edge. Since
$d_J(w_0)=2$, the vertex $w_0$ has a neighbor $z$ different
from $w_1$. If $z$ were not on the path, then
$zw_0w_1\cdots w_k$ would be a longer simple path in $J$.
Hence $z=w_j$ for some $j\geq2$; here $z\neq w_0$ because
$J$ has no loops, and $z\neq w_1$ by its choice. The edges
of $w_0w_1\cdots w_j$ together with $w_jw_0$ therefore form
a cycle in $J$. All these edges belong to $R$, since
$E(J)\subseteq E(R)$. This contradicts the fact that every
component of the linear forest $R$ is a path and therefore
contains no cycle. The assumption that $T-E(R)$ is disconnected
is thus impossible.
\end{proof}

A \emph{Hamiltonian path} in a graph $F$ is a simple path
containing every vertex of $F$, also called a \emph{spanning path}.
The graph $F$ is \emph{Hamiltonian-connected} if, for every pair
of distinct vertices $a,b\in V(F)$, there is a Hamiltonian path
whose endpoints are $a$ and $b$; the endpoints need not be adjacent.
We use the following theorem of Thomassen~\cite{thomassen1983},
with the correction to its proof by Chiba and
Nishizeki~\cite{chiba1986}.

\begin{theorem}\label{core:hamiltonian-connected}
Every four-connected planar graph is
Hamiltonian-connected. Explicitly, if $F$ is such a graph
and $a,b\in V(F)$ are distinct, then $F$ contains a simple
path $P$ with endpoints $a,b$ and with $V(P)=V(F)$.
\end{theorem}

For any maximal planar graph $F$ with $|V(F)|\geq3$, the degree
identity \eqref{prelim:defect}, also valid for $K_3$, gives
\begin{equation}\label{core:twelve-five}
 \sum_{v\in V(F)}(6-d_F(v))
 =6|V(F)|-2|E(F)|=12.
\end{equation}

Now fix a graph $H$ satisfying the hypotheses of
Theorem~\ref{core:main}: $H$ is a four-connected
sphere triangulation, and $d_H(v)\in\{5,6\}$ for every
$v\in V(H)$. In particular, $H$ is maximal planar, so
\eqref{core:twelve-five} applies with $F=H$. Each degree-five
vertex contributes one to the sum, and each degree-six vertex
contributes zero. Hence $|V_5(H)|=12$, and we may choose
two distinct vertices $a,b\in V_5(H)$.
Since $H$ is four-connected and planar, applying
Theorem~\ref{core:hamiltonian-connected} to $H$ and this
specified pair $a,b$ therefore gives a spanning path
\[
 P=(v_1,v_2,\ldots,v_n),\qquad
 v_1=a,\quad v_n=b,\quad n=|V(H)|,
\]
where the vertices $v_1,\ldots,v_n$ are distinct and exhaust
$V(H)$, and $v_iv_{i+1}\in E(H)$ for $1\leq i<n$.
Thus the two ends of $P$ are the chosen degree-five vertices,
and every other vertex of $H$ is internal on $P$.

Partition the edges of $P$ alternately into matchings $M_1,M_2$,
and put $K=H-E(P)$. By Lemma~\ref{core:linear-forest}, $K$ is
connected. Every internal vertex of $P$ loses two edges, and its
two ends lose one. Consequently
\begin{equation}\label{core:degree-classes}
 U=\{v:d_K(v)=4\}=V_6(H)\cup\{a,b\},\qquad
 W=\{v:d_K(v)=3\}=V_5(H)\setminus\{a,b\}.
\end{equation}
In particular $|W|=10$, and only two vertices of $U$ had degree
five in $H$. Our task is to find a matching of $K$ covering $U$.

\subsection{Curvature and boundary walks}

In a triangulation $H$ with degrees five and six, put
$p(S)=|S\cap V_5(H)|$. The following inequality uses only the
maximum degree of the deleted subgraph, so it applies more generally
than to the spanning path chosen above.

\begin{lemma}\label{core:curvature}
Let $H$ be a sphere triangulation with degrees five
and six, let $R$ be a subgraph of maximum degree at most two, and
set $K=H-E(R)$. Suppose that $S$ is a nonempty proper vertex set,
$|S|\geq2$, and both $H[S]$ and $H[V(H)\setminus S]$ are
connected. Then
\begin{equation}\label{core:cut-curvature}
 p(S)+|\delta_K(S)|\geq6.
\end{equation}
\end{lemma}

\begin{proof}
Work in the fixed embedding of $H$ on $\mathbb S^2$, and put
$\bar S=V(H)\setminus S$. Let $\Gamma_S$ and
$\Gamma_{\bar S}$ be the drawings of $H[S]$ and $H[\bar S]$,
including their vertex points and edge arcs. These drawings are
disjoint: their vertex sets are disjoint, and embedded edges meet
other edges or vertices only at common endpoints. Since
$\Gamma_{\bar S}$ is nonempty and connected, it lies in a unique
connected component $\Omega_S$ of $\mathbb S^2\setminus\Gamma_S$.
Thus $\Omega_S$ is the face of $H[S]$ containing all vertices and
edges of $H[\bar S]$. It need not be a face of $H$, and its boundary
need not be a simple cycle.

For example, in Figure~\ref{fig:face-containment}, $H$ is the
icosahedral graph and $S=\{a,b,c\}$ is a facial triangle. The
remaining nine vertices induce a connected graph drawn in the
exterior face $\Omega_S$ of the black triangle. This face is defined
by $H[S]$ alone; the overlaid orange edges do not subdivide it.

% BEGIN_FACE_CONTAINMENT_FIGURE
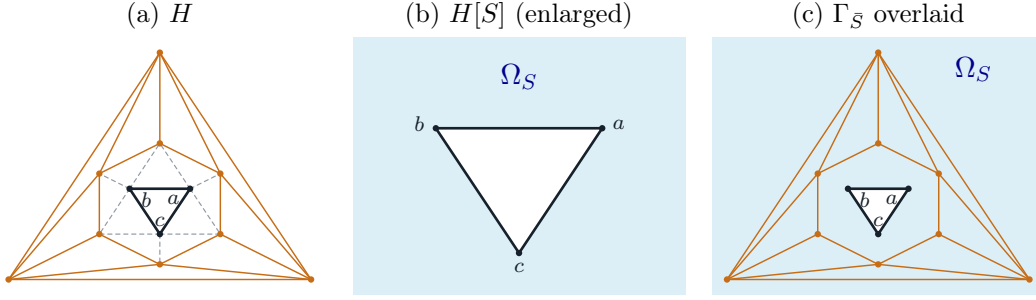
\begin{figure}[!htbp]
\centering
\begin{tikzpicture}[x=1cm,y=1cm,line cap=round,line join=round]
\definecolor{faceInk}{HTML}{17212B}
\definecolor{faceOrange}{HTML}{C66D16}
\definecolor{faceGrey}{HTML}{929BA4}
\definecolor{faceBlue}{HTML}{DCEFF7}
\draw[faceGrey,dash pattern=on 1.6pt off 1.4pt,line width=.45pt] (2.60000,1.40000) -- (3.00000,0.80000);
\draw[faceGrey,dash pattern=on 1.6pt off 1.4pt,line width=.45pt] (2.60000,1.40000) -- (3.00000,1.60000);
\draw[faceGrey,dash pattern=on 1.6pt off 1.4pt,line width=.45pt] (2.60000,1.40000) -- (2.20000,2.00000);
\draw[faceGrey,dash pattern=on 1.6pt off 1.4pt,line width=.45pt] (1.80000,1.40000) -- (2.20000,2.00000);
\draw[faceGrey,dash pattern=on 1.6pt off 1.4pt,line width=.45pt] (1.80000,1.40000) -- (1.40000,1.60000);
\draw[faceGrey,dash pattern=on 1.6pt off 1.4pt,line width=.45pt] (1.80000,1.40000) -- (1.40000,0.80000);
\draw[faceGrey,dash pattern=on 1.6pt off 1.4pt,line width=.45pt] (2.20000,0.80000) -- (1.40000,0.80000);
\draw[faceGrey,dash pattern=on 1.6pt off 1.4pt,line width=.45pt] (2.20000,0.80000) -- (2.20000,0.40000);
\draw[faceGrey,dash pattern=on 1.6pt off 1.4pt,line width=.45pt] (2.20000,0.80000) -- (3.00000,0.80000);
\draw[faceOrange,line width=.55pt] (2.20000,3.20000) -- (0.20000,0.20000);
\draw[faceOrange,line width=.55pt] (0.20000,0.20000) -- (4.20000,0.20000);
\draw[faceOrange,line width=.55pt] (4.20000,0.20000) -- (2.20000,3.20000);
\draw[faceOrange,line width=.55pt] (3.00000,1.60000) -- (2.20000,2.00000);
\draw[faceOrange,line width=.55pt] (2.20000,2.00000) -- (1.40000,1.60000);
\draw[faceOrange,line width=.55pt] (1.40000,1.60000) -- (1.40000,0.80000);
\draw[faceOrange,line width=.55pt] (1.40000,0.80000) -- (2.20000,0.40000);
\draw[faceOrange,line width=.55pt] (2.20000,0.40000) -- (3.00000,0.80000);
\draw[faceOrange,line width=.55pt] (3.00000,0.80000) -- (3.00000,1.60000);
\draw[faceOrange,line width=.55pt] (2.20000,3.20000) -- (3.00000,1.60000);
\draw[faceOrange,line width=.55pt] (2.20000,3.20000) -- (2.20000,2.00000);
\draw[faceOrange,line width=.55pt] (2.20000,3.20000) -- (1.40000,1.60000);
\draw[faceOrange,line width=.55pt] (0.20000,0.20000) -- (1.40000,1.60000);
\draw[faceOrange,line width=.55pt] (0.20000,0.20000) -- (1.40000,0.80000);
\draw[faceOrange,line width=.55pt] (0.20000,0.20000) -- (2.20000,0.40000);
\draw[faceOrange,line width=.55pt] (4.20000,0.20000) -- (2.20000,0.40000);
\draw[faceOrange,line width=.55pt] (4.20000,0.20000) -- (3.00000,0.80000);
\draw[faceOrange,line width=.55pt] (4.20000,0.20000) -- (3.00000,1.60000);
\draw[faceInk,line width=1pt] (2.60000,1.40000) -- (1.80000,1.40000);
\draw[faceInk,line width=1pt] (1.80000,1.40000) -- (2.20000,0.80000);
\draw[faceInk,line width=1pt] (2.20000,0.80000) -- (2.60000,1.40000);
\fill[faceOrange] (2.20000,3.20000) circle (1.2pt);
\fill[faceOrange] (0.20000,0.20000) circle (1.2pt);
\fill[faceOrange] (4.20000,0.20000) circle (1.2pt);
\fill[faceOrange] (3.00000,1.60000) circle (1.2pt);
\fill[faceOrange] (2.20000,2.00000) circle (1.2pt);
\fill[faceOrange] (1.40000,1.60000) circle (1.2pt);
\fill[faceOrange] (1.40000,0.80000) circle (1.2pt);
\fill[faceOrange] (2.20000,0.40000) circle (1.2pt);
\fill[faceOrange] (3.00000,0.80000) circle (1.2pt);
\fill[faceInk] (2.60000,1.40000) circle (1.2pt);
\fill[faceInk] (1.80000,1.40000) circle (1.2pt);
\fill[faceInk] (2.20000,0.80000) circle (1.2pt);
\node[font=\scriptsize,text=faceInk] at (2.38000,1.26000) {$a$};
\node[font=\scriptsize,text=faceInk] at (2.02000,1.26000) {$b$};
\node[font=\scriptsize,text=faceInk] at (2.20000,0.97200) {$c$};
\node[font=\small] at (2.20000,3.69000) {(a) $H$};
\fill[faceBlue] (4.75000,0) rectangle (9.15000,3.4);
\fill[white] (8.05000,2.20000) -- (5.85000,2.20000) -- (6.95000,0.55000) -- cycle;
\draw[faceInk,line width=1pt] (8.05000,2.20000) -- (5.85000,2.20000);
\draw[faceInk,line width=1pt] (5.85000,2.20000) -- (6.95000,0.55000);
\draw[faceInk,line width=1pt] (6.95000,0.55000) -- (8.05000,2.20000);
\fill[faceInk] (8.05000,2.20000) circle (1.2pt);
\fill[faceInk] (5.85000,2.20000) circle (1.2pt);
\fill[faceInk] (6.95000,0.55000) circle (1.2pt);
\node[font=\scriptsize,text=faceInk] at (8.27000,2.25500) {$a$};
\node[font=\scriptsize,text=faceInk] at (5.63000,2.25500) {$b$};
\node[font=\scriptsize,text=faceInk] at (6.95000,0.36300) {$c$};
\node[text=blue!55!black] at (6.95000,2.88200) {$\Omega_S$};
\node[font=\small] at (6.95000,3.69000) {(b) $H[S]$ (enlarged)};
\fill[faceBlue] (9.50000,0) rectangle (13.90000,3.4);
\fill[white] (12.10000,1.40000) -- (11.30000,1.40000) -- (11.70000,0.80000) -- cycle;
\draw[faceOrange,line width=.55pt] (11.70000,3.20000) -- (9.70000,0.20000);
\draw[faceOrange,line width=.55pt] (9.70000,0.20000) -- (13.70000,0.20000);
\draw[faceOrange,line width=.55pt] (13.70000,0.20000) -- (11.70000,3.20000);
\draw[faceOrange,line width=.55pt] (12.50000,1.60000) -- (11.70000,2.00000);
\draw[faceOrange,line width=.55pt] (11.70000,2.00000) -- (10.90000,1.60000);
\draw[faceOrange,line width=.55pt] (10.90000,1.60000) -- (10.90000,0.80000);
\draw[faceOrange,line width=.55pt] (10.90000,0.80000) -- (11.70000,0.40000);
\draw[faceOrange,line width=.55pt] (11.70000,0.40000) -- (12.50000,0.80000);
\draw[faceOrange,line width=.55pt] (12.50000,0.80000) -- (12.50000,1.60000);
\draw[faceOrange,line width=.55pt] (11.70000,3.20000) -- (12.50000,1.60000);
\draw[faceOrange,line width=.55pt] (11.70000,3.20000) -- (11.70000,2.00000);
\draw[faceOrange,line width=.55pt] (11.70000,3.20000) -- (10.90000,1.60000);
\draw[faceOrange,line width=.55pt] (9.70000,0.20000) -- (10.90000,1.60000);
\draw[faceOrange,line width=.55pt] (9.70000,0.20000) -- (10.90000,0.80000);
\draw[faceOrange,line width=.55pt] (9.70000,0.20000) -- (11.70000,0.40000);
\draw[faceOrange,line width=.55pt] (13.70000,0.20000) -- (11.70000,0.40000);
\draw[faceOrange,line width=.55pt] (13.70000,0.20000) -- (12.50000,0.80000);
\draw[faceOrange,line width=.55pt] (13.70000,0.20000) -- (12.50000,1.60000);
\draw[faceInk,line width=1pt] (12.10000,1.40000) -- (11.30000,1.40000);
\draw[faceInk,line width=1pt] (11.30000,1.40000) -- (11.70000,0.80000);
\draw[faceInk,line width=1pt] (11.70000,0.80000) -- (12.10000,1.40000);
\fill[faceOrange] (11.70000,3.20000) circle (1.2pt);
\fill[faceOrange] (9.70000,0.20000) circle (1.2pt);
\fill[faceOrange] (13.70000,0.20000) circle (1.2pt);
\fill[faceOrange] (12.50000,1.60000) circle (1.2pt);
\fill[faceOrange] (11.70000,2.00000) circle (1.2pt);
\fill[faceOrange] (10.90000,1.60000) circle (1.2pt);
\fill[faceOrange] (10.90000,0.80000) circle (1.2pt);
\fill[faceOrange] (11.70000,0.40000) circle (1.2pt);
\fill[faceOrange] (12.50000,0.80000) circle (1.2pt);
\fill[faceInk] (12.10000,1.40000) circle (1.2pt);
\fill[faceInk] (11.30000,1.40000) circle (1.2pt);
\fill[faceInk] (11.70000,0.80000) circle (1.2pt);
\node[font=\scriptsize,text=faceInk] at (11.88000,1.26000) {$a$};
\node[font=\scriptsize,text=faceInk] at (11.52000,1.26000) {$b$};
\node[font=\scriptsize,text=faceInk] at (11.70000,0.97200) {$c$};
\node[text=blue!55!black] at (12.96000,2.98000) {$\Omega_S$};
\node[font=\small] at (11.70000,3.69000) {(c) $\Gamma_{\bar S}$ overlaid};
\end{tikzpicture}
\caption{A concrete example for Lemma~\ref{core:curvature}.
The graph $H$ is the icosahedral graph, with all twelve vertices
of degree five, and $S=\{a,b,c\}$ is the vertex set of its central
facial triangle. Black shows $H[S]$, orange shows $H[\bar S]$,
and the gray dashed edges join the two vertex sets.
In (b), only $H[S]$ is retained and the view is enlarged:
the white triangular interior is one face, and the blue exterior
is the other face, $\Omega_S$.
In (c), the orange drawing is overlaid only to locate
$\Gamma_{\bar S}$ inside $\Omega_S$; the orange edges are not
part of $H[S]$ and do not subdivide its face.
The blue face continues beyond the picture, including beyond
the outer orange triangle.}
\label{fig:face-containment}
\end{figure}
\FloatBarrier
% END_FACE_CONTAINMENT_FIGURE

Choose a point of $\Omega_S$ outside the drawing of $H$ as the
point at infinity. In the resulting plane representation,
$\Omega_S$ is the outer face of $H[S]$.

Every face of $H[S]$ other than $\Omega_S$ is a facial triangle of $H$.
Indeed, consider an edge $e=uw$ with $w\in\bar S$ and a point
$z$ in its interior. The subarc from $z$ to $w$ avoids
$\Gamma_S$: embedded edges cannot cross or pass through other
vertices, and the subarc omits $u$, the only endpoint that
could lie in $S$. Since $w\in\Omega_S$, this puts $z$ in
$\Omega_S$ as well. Thus the interiors of all edges having
an endpoint in $\bar S$ lie in $\Omega_S$.
Every other edge of $H$ already belongs to the induced graph
$H[S]$, and all vertices outside $S$ lie in $\Omega_S$.
No additional vertex or edge of $H$ therefore subdivides any
other face of $H[S]$. Such a face remains a face of $H$ and
has a triangular boundary.

Write $n=|S|$, $e=|E(H[S])|$, and let $\ell$ be the length of
the boundary walk of $\Omega_S$, the face containing
$\Gamma_{\bar S}$. Repeated edge occurrences are counted separately.
In particular we do not assume that this boundary is a simple
cycle. If $f$ is the total number of faces of the connected plane
graph $H[S]$, including its outer face $\Omega_S$, then its
$f-1$ bounded faces each have boundary length three, whereas
the outer boundary walk has length $\ell$. Euler's formula
and face-edge counting therefore give
\[
 f=e-n+2,\qquad 2e=3(f-1)+\ell.
\]
Hence
\begin{equation}\label{core:boundary-count}
 e=3n-3-\ell,\qquad
 |\delta_H(S)|=6n-p(S)-2e=6+2\ell-p(S).
\end{equation}

Every edge of $R$ leaving $S$ starts at a vertex on this boundary
walk: its interior lies in $\Omega_S$. There are
at most $\ell$ distinct boundary vertices, and each is incident
with at most two edges of $R$. Therefore
$|\delta_R(S)|\leq2\ell$. Since the $R$-edges have been deleted,
\[
 |\delta_K(S)|=|\delta_H(S)|-|\delta_R(S)|
 \geq 6-p(S),
\]
as claimed. The use of the boundary walk, rather than a boundary
cycle, includes bridges and repeated vertices in $H[S]$.
\end{proof}

\begin{lemma}\label{core:two-cut}
Let $H$ be a sphere triangulation with degrees
five and six, let $R$ be a subgraph of $H$ with
$\Delta(R)\leq2$, and put $K=H-E(R)$. Suppose that $K$
is connected, every vertex of $K$ has degree three or four,
and $|V_4(K)\cap V_5(H)|\leq3$.
Let $S$ be a nonempty proper subset of $V(H)$ such that
$K[S]$ is connected, every vertex of $S$ has degree four
in $K$, and $|\delta_K(S)|=2$.
Then $H[V(H)\setminus S]$ is disconnected, and $K-S$ has
exactly two connected components, each joined to $S$ by
exactly one edge. In particular, the two edges of
$\delta_K(S)$ are bridges of $K$.
\end{lemma}

\begin{proof}
If $S=\{v\}$, then $|\delta_K(S)|=d_K(v)=4$, because
$K$ is simple. This contradicts $|\delta_K(S)|=2$, so
$|S|\geq2$. Also $K[S]$ is a connected spanning subgraph
of $H[S]$, and hence $H[S]$ is connected. Since every
vertex of $S$ belongs to $V_4(K)$, the global bound on
$V_4(K)\cap V_5(H)$ gives
$p(S)=|S\cap V_5(H)|\leq |V_4(K)\cap V_5(H)|\leq3$.

Suppose, for a contradiction, that $H[V(H)\setminus S]$ is
connected. Both induced sides are then connected and $|S|\geq2$,
so Lemma~\ref{core:curvature} applies to $H,R,K,S$ and gives
\[
 6\leq p(S)+|\delta_K(S)|\leq3+2,
\]
a contradiction. Therefore $H[V(H)\setminus S]$ is
disconnected. The graph $K-S$ is obtained from this induced
subgraph by deleting edges while keeping all its vertices,
so $K-S$ is disconnected as well.

Let $C_1,\ldots,C_t$ be the connected components of $K-S$,
where $t\geq2$. For each $i$, let $m_i$ be the number of
edges of $K$ joining $C_i$ to $S$. The connectedness of
$K$ gives $m_i\geq1$, and every edge of $\delta_K(S)$ is
counted in exactly one $m_i$. Thus
\[
 2=|\delta_K(S)|=\sum_{i=1}^t m_i\geq t\geq2.
\]
It follows that $t=2$ and $m_1=m_2=1$. Deleting the
unique edge joining either $C_i$ to $S$ separates $C_i$
from the rest of $K$, since no edge joins different
components of $K-S$. Hence both edges of $\delta_K(S)$
are bridges of $K$.
\end{proof}

\subsection{Covering the degree-four vertices}

\begin{theorem}\label{core:third-matching}
Let $H$ be a sphere triangulation with degrees five
and six. Let $R$ be a subgraph of maximum degree at most two, and
suppose that $K=H-E(R)$ is connected with degrees three and four.
Put $U=\{v:d_K(v)=4\}$. If $K$ has a degree-three vertex and
at most three vertices of $U$ have degree five in $H$, then $K$
has a matching covering $U$.
\end{theorem}

\begin{proof}
We verify Lemma~\ref{core:required}. For $X=\varnothing$, the
unique component contains an optional degree-three vertex, and
$q_U(K)=0$.

Fix $X\neq\varnothing$. A counted odd component $C$ of $K-X$
consists entirely of degree-four vertices. Its cut is nonempty
by connectedness and has even size, since
\begin{equation}\label{core:even-cut}
 |\delta_K(C)|=4|V(C)|-2|E(C)|.
\end{equation}
Thus each such cut has size two or at least four. Let $q_2$ be
the number of counted components with cut size two.
For each such component $C$, apply Lemma~\ref{core:two-cut}
with $S=V(C)$. This is a nonempty proper subset of $V(K)$,
because $C$ is nonempty and $X\neq\varnothing$.
Moreover, $K[S]=C$ is connected, every vertex of $S$ has
degree four in $K$, and $|\delta_K(S)|=2$ by the choice
of $C$. The hypotheses of the present theorem include all
the required assumptions on $H,R,K$.
Lemma~\ref{core:two-cut} therefore shows that the two edges
of $\delta_K(C)$ are bridges of $K$.

If $q_2=0$, every counted component sends at least four edges
to $X$. These edge incidences are distinct, and vertices of
$X$ have degree at most four, so $4q_U(K-X)\leq4|X|$.
This proves the required inequality in that case.

Suppose $q_2>0$, and denote the $q_2$ counted components of
$K-X$ with cut size two by $C_1,\ldots,C_{q_2}$.
Define
\[
 Z=\bigcup_{j=1}^{q_2}V(C_j),\qquad
 K_0=K-Z=K[V(K)\setminus Z].
\]
The graph $K_0$ deletes the selected component vertex sets from
$K$ and retains $X$. Let $D_1,\ldots,D_m$
be all the connected components of $K_0$, including any
isolated vertices.

Every $C_j$ is a component of $K-X$, so $X\cap Z=\varnothing$.
Thus $\varnothing\neq X\subseteq V(K_0)$, and $m\geq1$.
Each $D_i$ has a nonempty vertex set.

If the $C_j$ are all the components of $K-X$, then
$Z=V(K)\setminus X$ and $K_0=K[X]$. This induced graph
retains the edges within $X$ and need not be connected;
the $D_i$ are precisely its connected components.

The vertex sets of
$C_1,\ldots,C_{q_2},D_1,\ldots,D_m$ partition $V(K)$.
Define the contracted graph $T$, initially as a multigraph
without loops.
For each part introduce a distinct new vertex, and put
$V(T)=\{c_1,\ldots,c_{q_2}\}\cup\{d_1,\ldots,d_m\}$.
Define the map $\pi:V(K)\longrightarrow V(T)$ by
\[
 \pi(v)=
 \begin{cases}
 c_j,&v\in V(C_j),\\
 d_i,&v\in V(D_i).
 \end{cases}
\]
For every edge $uv\in E(K)$ with $\pi(u)\neq\pi(v)$,
include a distinct edge $e_{uv}$ of $T$ whose endpoints
are $\pi(u)$ and $\pi(v)$. Discard the edges $uv$ with
$\pi(u)=\pi(v)$. Thus
$E(T)=\{e_{uv}:uv\in E(K),\ \pi(u)\neq\pi(v)\}$.
The contraction is performed in $K$, with the $C_j$ still present.
It identifies each connected part to its representative, discards
its internal edges, and retains distinct edges between parts.
Parallel edges are allowed at this stage.

There is no edge of $K$ between two distinct $C_j$, since
they are distinct components of $K-X$. There is also no edge
of $K$ between two distinct $D_i$: such an edge would survive
in $K_0$ and join two of its components.
Consequently, every edge of $T$ joins a vertex $c_j$ to a
vertex $d_i$. More precisely, the original edges retained
in this contraction are exactly
\[
 \{uv\in E(K):\pi(u)\neq\pi(v)\}
   =\bigcup_{j=1}^{q_2}\delta_K(C_j),
\]
where the sets on the right are pairwise disjoint.
Thus each retained edge belongs to $\delta_K(C_j)$ for a
unique selected component $C_j$. Applying
Lemma~\ref{core:two-cut} with $S=V(C_j)$, as verified above,
shows that this edge is a bridge of $K$.
In particular, even if $\delta_K(D_i)$ has four or more edges,
each joins $D_i$ to a selected $C_j$ and is a bridge by the
application of the lemma to that $C_j$.

There may also be components $C$ of $K-X$ with
$|\delta_K(C)|\geq4$. None is among the selected $C_j$.
Such a component survives intact in $K_0$ and is contained
in one $D_i$. Every edge $uv\in\delta_K(C)$, with
$u\in V(C)$, has $v\in X$, because $C$ is a component
of $K-X$. Both endpoints survive in $K_0$, so the edge
$uv$ also survives there and places $v$ in the same $D_i$
as $u$. Consequently, $\delta_K(C)\subseteq E(D_i)$.
These edges are internal to $D_i$ and are discarded in
the contraction; their individual bridge status is immaterial.

The graph $T$ is connected because $K$ is connected.
To prove that every edge $e_{uv}$ of $T$ is a bridge,
suppose that a path in $T-e_{uv}$ joins $\pi(u)$ to $\pi(v)$.
Replace each edge of this path by its corresponding edge of
$K$. Within each part visited by the path, join the necessary
endpoints by a path in that connected part; in the first
and last parts, also join to $u$ and $v$, respectively.
The resulting walk joins $u$ to $v$ in $K$ without using
$uv$. Indeed, its edges between parts correspond to edges
other than $e_{uv}$, and its edges within parts cannot be $uv$.
This contradicts the fact that $uv$ is a bridge of $K$.
Hence every edge of $T$ is a bridge. In particular, $T$
has neither parallel edges nor a cycle, so it is a tree.

Each $c_j$ has degree two because its incident edges
correspond exactly to the two edges of $\delta_K(C_j)$.
Since every edge of $T$ has exactly one endpoint among the
$c_j$, we have
\[
 |E(T)|=\sum_{j=1}^{q_2}d_T(c_j)=2q_2,
 \qquad |V(T)|=q_2+m.
\]
The tree identity $|E(T)|=|V(T)|-1$ therefore gives
\begin{equation}\label{core:tree-count}
 m=q_2+1.
\end{equation}

If $K_0=K[X]$, its $m$ nonempty components partition $X$, so
$m\leq|X|$ and $q_U(K-X)=q_2=m-1\leq|X|-1$.
The following count also covers this case.

For $1\leq i\leq m$, put $x_i=|X\cap V(D_i)|$, and let $s_i$
be the number of edges of $K$ joining $D_i$ to one of
$C_1,\ldots,C_{q_2}$. Every such edge has its endpoint in
$D_i$ in $X$, since each $C_j$ is an entire component of
$K-X$. The contracted tree
has more than one vertex, so $s_i\geq1$ and hence $x_i\geq1$.
The sets $X\cap V(D_i)$ partition $X$.

Let $q_i$ be the number of the other counted odd components of
$K-X$ that lie in $D_i$. For each such component $C$,
the set $\delta_K(C)$ has at least four edges.
All of those edges end in $X\cap V(D_i)$: none can lead
to a removed component of $K-X$, and otherwise it would connect
two different components of $K_0$. The $s_i$ edges of $K$
joining $D_i$ to the selected components $C_j$ already
consume $s_i$ incidences at its
$x_i$ vertices of $X$. Therefore $4q_i\leq4x_i-s_i$.
Since $s_i\geq1$ and $q_i$ is integral, $q_i\leq x_i-1$.
Summing and using \eqref{core:tree-count} yields
\begin{align*}
 q_U(K-X)
   &=q_2+\sum_{i=1}^m q_i\\
   &\leq q_2+\sum_{i=1}^m(x_i-1)\\
   &=q_2+|X|-(q_2+1)=|X|-1.
\end{align*}
Thus the required inequality holds for every $X$, and
Lemma~\ref{core:required} supplies the matching covering $U$.
\end{proof}

The point of the tree count is that a two-edge component costs
less than four incidences at $X$, but its two bridge edges split
the remaining graph into additional connected components.
Every such component has at least one incidence at $X$ used by
an edge to a removed component. The resulting bound
$q_i\leq x_i-1$ offsets the deficiency of all two-edge components.

\begin{proof}[Proof of Theorem~\ref{core:main}]
Choose the spanning path $P$ with degree-five ends, and its
alternating matchings $M_1,M_2$, as above. The graph $K=H-E(P)$
is connected by Lemma~\ref{core:linear-forest}, and its degree
classes satisfy \eqref{core:degree-classes}. In particular it
has ten degree-three vertices, and only two of its degree-four
vertices had degree five in $H$. Apply
Theorem~\ref{core:third-matching} with $R=P$ to obtain a matching
$M_3$ of $K$ covering every degree-four vertex.

The matchings $M_1,M_2,M_3$ are pairwise edge-disjoint. In
$Q=H-(M_1\cup M_2\cup M_3)=K-M_3$,
every former degree-four vertex has degree three. Every former
degree-three vertex has degree two or three. Therefore
$\Delta(Q)\leq3$. Apply Corollary~\ref{ext:three-matchings} to
the originally supplied proper coloring $\phi$ of $H$. It
colors the three matchings with $5,6,7$, respectively, and the
residual edges with colors in $\{1,2,3,4,8\}$, preserving all
vertex colors. This is the asserted total eight-coloring.
\end{proof}

To apply Theorem~\ref{core:main} to a general triangulation, we
need a reduction preserving simplicity, the degree bound and
four-connectivity, followed by the restoration of removed vertices
in Theorem~\ref{restore:main}. Corollary~\ref{core:four-palette}
supplies the four-element vertex color set required for that
restoration.

\section{Reduction to four-connectivity}
\label{red:section}

The induction hypothesis asserts the existence of a proper total
eight-coloring, without prescribing vertex colors.
Theorem~\ref{core:main} supplies the stronger extension property for the
reduced core. Separating triangles will be handled by a structural
classification and an explicit coloring.

\subsection{Small graphs and degree-three reduction}

The empty graph, $K_1$, and $K_2$ are immediate.  For $K_3$, give the
vertices the colors $0,1,2$, and give each edge the color of its opposite
vertex.  For $K_4$, label the vertices $0,1,2,3$ and use the five colors
in $\mathbb Z/5\mathbb Z$, the residue classes of integers modulo
five, represented by $0,1,2,3,4$. Arithmetic on these color names
is evaluated modulo five.

Define the total coloring by
\[
 \varphi(i)=2i\pmod5,\qquad
 \varphi(ij)=i+j\pmod5
 \quad (i,j\in\{0,1,2,3\},\ i\neq j).
\]
The four vertex colors, in vertex order $0,1,2,3$, are
$0,2,4,1$, so they are distinct.  Edges $ij$ and $ik$ incident
with the same vertex have distinct colors: equality would give
$i+j\equiv i+k\pmod5$, hence $j\equiv k\pmod5$, which for these
labels implies $j=k$.  Finally, the color of $ij$ cannot equal
the color of either endpoint: either $i+j\equiv2i\pmod5$ or
$i+j\equiv2j\pmod5$ would imply $i=j$.  Thus this is a proper
total coloring using five colors.

\begin{lemma}\label{red:degree-three}
Let $G$ be a graph with maximum degree at most six.  Suppose
that $v$ has three neighbors $a,b,c$ and that $ab\in E(G)$.  Every total
eight-coloring of $G-v$ extends to a total eight-coloring of $G$ after
changing, if necessary, only the old edge color on $ab$.  All old vertex
colors can be retained.
\end{lemma}

\begin{proof}
Let $\Omega$ be the set of eight colors and let $\varphi$ be the
given coloring.  For $x\in\{a,b,c\}$, put
\[
 L_x=\Omega\setminus
 \bigl(\{\varphi(x)\}\cup
       \{\varphi(e):e\in E(G-v),\ x\in V(e)\}\bigr).
\]
The forbidden colors at $x$ are its vertex color and the colors of
its $d_G(x)-1$ old incident edges.  They are pairwise distinct, so
$|L_x|=8-d_G(x)\ge2$.  If
$|L_a\cup L_b\cup L_c|\ge3$, Theorem~\ref{prelim:hall} supplies three distinct
representatives: the inequalities for one or two lists are automatic
from their sizes.  These representatives properly color the three
new edges at their old endpoints and at their common endpoint $v$.

Otherwise the three lists are one two-set $\{\alpha,\beta\}$.  Let
$\gamma=\varphi(ab)$; this color is outside that set.  Recolor $ab$
with $\alpha$, which was missing at both its endpoints, and set
\[
 \varphi(va)=\gamma,\qquad
 \varphi(vb)=\beta,\qquad
 \varphi(vc)=\alpha.
\]
The color $\gamma$ has just been released at $a$, while $\beta$ remains
available at $b$ and $\alpha$ remains available at $c$.  The two edges
now colored $\alpha$ are disjoint.  Thus this assignment is proper.

In either case the neighbors of $v$ and its three incident edges use
at most six colors.  Any color outside those colors can be assigned
to $v$.
\end{proof}

If $v$ has degree three in a sphere triangulation, its link is a
triangle.  Deleting $v$ replaces its three incident faces by that
triangle and introduces neither an edge nor an identification.  Except
for the already handled transition $K_4\to K_3$, the result is a smaller
sphere triangulation.  Its degrees only decrease.  Consequently
Lemma~\ref{red:degree-three} gives a valid induction step within
$\mathcal T_6$.

\subsection{Rigidity of a triangular disk}

Each closed side of a separating triangle is a triangulated disk.
A vertex strictly inside such a disk has all its incident edges
in the disk. We begin by classifying the disks that arise when
their boundary degrees are at most four.

For $k\ge2$, define the \emph{octahedral chain} $O_k$ as follows.  Its
vertices are $v_{i,j}$, where $0\le i<k$ and $j\in\mathbb Z/3\mathbb Z$.
Each set $\{v_{i,0},v_{i,1},v_{i,2}\}$ spans a triangle, called a layer.
Between consecutive layers $i,i+1$, the edge $v_{i,j}v_{i+1,\ell}$ is
present precisely when $j\ne\ell$.  There are no other edges.  The
successive pairs of layer triangles bound triangulated annuli with faces
\[
 v_{i,j}v_{i,j+1}v_{i+1,j+2},\qquad
 v_{i,j}v_{i+1,j+1}v_{i+1,j+2}
 \quad (j\in\mathbb Z/3\mathbb Z).
\]
Capping the two end triangles gives a sphere triangulation.  Its end
vertices have degree four, and all its other vertices have degree six.
In particular $O_2$ is the octahedron.  Removing an end face, without
removing its vertices or edges, gives an \emph{octahedral-chain disk}.

\begin{lemma}\label{red:disk}
Let $D$ be a triangulated topological disk with boundary $abc$
and at least one interior vertex.  Suppose
\[
 d_D(a),d_D(b),d_D(c)\le4,
 \qquad 4\le d_D(x)\le6
 \quad\text{for every interior vertex }x.
\]
Then $D$ is an octahedral-chain disk.  In particular each of its three
boundary vertices has degree exactly four in $D$.
\end{lemma}

\begin{proof}
Every boundary edge of the triangulated disk belongs to exactly one
triangular face in $D$. Consider the face containing $ab$, and write
it as $abx$. Its third vertex $x$ is distinct from $a,b$.
If $x=c$, this face has the entire boundary $abc$ of $D$ as its
boundary, so its closure is all of $D$. This would leave no interior
vertex, contrary to the hypothesis. Hence $x$ lies strictly inside
$D$, and both $a$ and $b$ have an interior neighbor.
The same argument applies to the faces containing $bc$ and $ca$.
Thus each boundary vertex has at least one interior neighbor.
It has at most two, because it already has its two boundary
neighbors and its degree in $D$ is at most four.

Suppose first that $a$ has only one interior neighbor $x$.  Its link
path is $b,x,c$, and $abx,acx$ are faces.  Removing those two faces and
$a$ leaves the disk bounded by $bcx$.  If that disk is a face, $x$ has
exactly the three neighbors $a,b,c$ in $D$, a contradiction.  Otherwise
let $y$ be the interior third vertex of its face on $bc$.  In $D$, the
vertex $b$ already has the three neighbors $a,c,x$ and can have at
most one further neighbor.  This further neighbor must be $y$.
Likewise $c$ has $y$ as its only neighbor strictly inside $bcx$.
The boundary links in the smaller disk therefore force the three faces
$bcy,bxy,cxy$.  Their union is the entire disk bounded by $bcx$:
their interior edges are paired and their only boundary is $bcx$.
It follows that $y$ has only the three neighbors $b,c,x$, again a
contradiction.  Every boundary vertex therefore has exactly two interior
neighbors.

Let $x_{ab},x_{bc},x_{ca}$ be the third vertices of the faces on the
three boundary edges.  They are pairwise distinct.  For example,
$x_{ab}=x_{bc}$ would identify the first and last interior entries in
the simple link path at $b$; that path could then have only one interior
entry, contrary to the preceding conclusion.  The links are consequently
\[
 \begin{array}{c|c}
 \text{boundary vertex}&\text{link path in }D\\ \hline
 a&b,x_{ab},x_{ca},c\\
 b&c,x_{bc},x_{ab},a\\
 c&a,x_{ca},x_{bc},b.
 \end{array}
\]
These links force the inner triangle $x_{ab}x_{bc}x_{ca}$ and the six faces of
an octahedral annulus between that triangle and $abc$.  The listed
paths are the complete boundary fans, so everything not in this annulus
lies in the disk bounded by the inner triangle.  The new boundary is
thus a genuine simple triangle, proved from the links.

Each new boundary vertex has exactly two neighbors on the preceding
boundary, and hence degree at most $6-2=4$ in the remaining disk.
Strictly interior vertices lose no incident edge.  If the remaining
disk has an interior vertex, it satisfies the same hypotheses and the
argument repeats.  Each step replaces the boundary by three distinct
interior vertices, decreasing the number of strictly interior vertices
by three.  The process therefore terminates at a face.  The successive
annuli form exactly an octahedral chain: after a layer is labeled, label
the next by pairing each vertex with its unique nonadjacent vertex in
the preceding layer.  The initial disk has an interior layer, so the
chain has at least two layers.
\end{proof}

\begin{proposition}\label{red:separator}
If a sphere triangulation $G$ satisfies $4\le d_G(v)\le6$ for every
vertex and has a separating triangle, then $G\cong O_k$ for some $k\ge3$.
\end{proposition}

\begin{proof}
Among the closed disks on both sides of all separating triangles,
choose one, denoted $D$, with as few interior vertices as possible.
Write its boundary as $abc$; its interior is nonempty.

Each of $a,b,c$ has a neighbor strictly inside $D$. Indeed, the face
on $ab$ has a third vertex $x\ne c$: otherwise its closure would be
all of $D$, leaving no interior vertex. Thus $x$ is an interior
neighbor of $a,b$, and the face on $bc$ gives one for $c$.
In fact each boundary vertex has at least two interior neighbors.
If $a$ had just one such neighbor $x$, its faces in $D$ would be
$abx,acx$.  If $x$ were the only interior vertex, it would have ambient
degree three, since no edge at an interior vertex can leave the disk.
Otherwise the remaining disk $bcx$ would have an interior vertex.  Its
other side contains $a$, so $bcx$ would be a separating triangle with a
smaller nonempty side.  Both alternatives are impossible.

For a boundary vertex $t$, let $r_t$ and $s_t$ count its neighbors in
the two open sides, with $r_t$ referring to $D$.  We have
\[
 2+r_t+s_t=d_G(t)\le6,\qquad r_t\ge2.
\]
Thus $s_t\le2$.  The opposite closed side has boundary degrees at most
four, and its interior vertices have degrees between four and six.
By Lemma~\ref{red:disk} it is a chain disk, in particular $s_t=2$ at
each boundary vertex.  The displayed inequality then forces $r_t=2$.
Lemma~\ref{red:disk} applies also to $D$.

The two chain disks concatenate along their common layer. The labels of
a new layer are determined by the unique nonadjacent pairing with the
preceding layer,
starting with any labeling of the common triangle.  Each side has at
least one interior layer, so their union has at least three layers.
\end{proof}

\begin{proposition}
\label{red:chain-colouring}
Every $O_k$ has a total coloring with at most seven colors.  For
$O_2$, five colors suffice.
\end{proposition}

\begin{proof}
Use representatives $j\in\{0,1,2\}$ and give $v_{i,j}$ color $j$.
Give a layer edge $v_{i,j}v_{i,\ell}$ color $3-j-\ell$, the unique
member of $\{0,1,2\}$ different from its endpoint colors.  For a cross
edge $v_{i,j}v_{i+1,\ell}$, let $d\in\{1,2\}$ represent
$\ell-j$ modulo three, and put
\[
 \varphi(v_{i,j}v_{i+1,\ell})
   =3+2(i\bmod2)+(d-1).
\]
Successive annuli thus alternate between the pairs $\{3,4\}$ and
$\{5,6\}$.  Adjacent vertices have different second indices, so their
colors differ.  At $v_{i,j}$, the vertex and the two layer edges have
the three distinct colors $0,1,2$.  The two outgoing cross edges have
the two distinct colors of their annulus, because their differences
$d$ are $1$ and $2$; the same holds for the two incoming edges.  When
both pairs are present, their color sets are disjoint.  All cross-edge
colors avoid $0,1,2$.  This verifies every vertex and edge conflict.
If $k=2$ there is only one annulus and only colors $0,1,2,3,4$ occur.
\end{proof}

Combining Propositions~\ref{red:separator} and~\ref{red:chain-colouring},
every sphere triangulation $G$ with $4\le d_G(v)\le6$ for all vertices
and a separating triangle satisfies $\chit(G)=7$.
Indeed, it is $O_k$ with $k\ge3$, so it has a vertex of degree six.
That vertex and its six incident edges require seven distinct colors,
and the displayed construction attains this bound.
Neither the structural classification nor the explicit coloring uses
the Four-Color Theorem.

\subsection{Connectivity and minimality}

\begin{lemma}\label{red:connectivity}
A sphere triangulation on at least four vertices is
three-connected.  On at least five vertices it is four-connected if and
only if it has no separating triangle.
\end{lemma}

\begin{proof}
To avoid a vertex $a$ in a path, replace every passage through $a$ by
a path in its cyclic link.  This proves that $G-a$ is connected.  To
avoid a second vertex $b$, use the link of $b$ after deleting $a$ if
necessary.  A cycle with at most one vertex removed is connected, so
the same rerouting proves that $G-\{a,b\}$ is connected.

Suppose $S=\{a,b,c\}$ is a cut.  Every component of $G-S$ has a
neighbor at each vertex of $S$; otherwise its neighborhood would lie
in a set of at most two vertices, contradicting what was just proved.
In particular the link of $a$ contains vertices from at least two
components.  If either $b$ or $c$ were absent from that link, deleting
$b,c$ would remove at most one link vertex, leaving a connected path
or cycle joining those components.  Hence $a$ is adjacent to both $b$
and $c$.  Repeating at $b$ and $c$ proves that $abc$ is a triangle.
It cannot be facial: if it were, $b,c$ would be consecutive in the
link of $a$, and deleting those two consecutive vertices would again
leave a connected path containing neighbors from distinct components.
An embedded triangle with an empty side bounds a face; hence this
nonfacial triangle is separating.  Conversely, deletion of the three
vertices of a separating triangle disconnects its two nonempty sides.
\end{proof}

\begin{corollary}\label{red:first-minimal}
If some member of $\mathcal T_6$ has no total eight-coloring, a
counterexample of minimum order is four-connected, has degrees in
$\{4,5,6\}$, and has every triangle facial.
\end{corollary}

\begin{proof}
The small colorings above exclude order at most four.  A three-vertex
could be deleted within $\mathcal T_6$, and the resulting coloring
could be restored by Lemma~\ref{red:degree-three}, contradicting
minimality.  Thus the minimum degree is at least four.  A separating
triangle would make the graph an octahedral chain by
Proposition~\ref{red:separator}, and Proposition~\ref{red:chain-colouring}
would color it directly.  There is therefore no separating triangle.
Lemma~\ref{red:connectivity} gives four-connectivity.  Finally a
nonfacial triangle in a sphere triangulation is separating.
\end{proof}

Agreement between two total colorings on the three boundary vertices
and three boundary edges would not control collisions between edges on
opposite sides at a common boundary vertex. The explicit chain coloring
resolves the separating case without this compatibility requirement.

\subsection{A degree-four vertex with a low-degree neighbor}

We first give the exact Hall conditions for completing a four-star.
Suppose $\Delta(G)\le6$, the neighborhood of $v$ is exactly
$\{u_0,u_1,u_2,u_3\}$ with these four vertices distinct, and a proper
total coloring $\varphi$ of $G-v$ uses colors from $\Omega$, where $|\Omega|=8$.
For each $i\in\{0,1,2,3\}$, define
\begin{align*}
 L_i&=\Omega\setminus\bigl(\{\varphi(u_i)\}\cup
             \{\varphi(e):e\in E(G-v),\ u_i\in V(e)\}\bigr),\\
 R&=\Omega\setminus\{\varphi(u_0),\varphi(u_1),
                        \varphi(u_2),\varphi(u_3)\}.
\end{align*}
The list $L_i$ consists of the colors allowed on the new edge $vu_i$
by the already colored vertex $u_i$ and its old incident edges.
The list $R$ consists of the colors allowed on the restored vertex
$v$ by its four already colored neighbors. Define
$U=L_0\cup L_1\cup L_2\cup L_3=\bigcup_{i=0}^3L_i$.
The union of all five lists, including the list for $v$, is $R\cup U$.
Each $L_i$ has size $8-d_G(u_i)\ge2$, and $|R|\ge4$.

The five uncolored elements $v,vu_0,vu_1,vu_2,vu_3$ are pairwise
conflicting: the four edges share the endpoint $v$, and $v$ is
incident with each edge. Restoring them while keeping all old
colors fixed is therefore equivalent to choosing five pairwise
distinct colors
\[
 c_\ast\in R,\qquad c_i\in L_i\quad(i=0,1,2,3).
\]
Indeed, every proper extension provides such choices. Conversely,
assigning $c_\ast$ to $v$ and $c_i$ to $vu_i$ avoids conflicts
with old elements by the definitions of $R,L_i$, and avoids
conflicts among the five new elements by their pairwise distinctness.

Apply Hall's theorem (Theorem~\ref{prelim:hall}) with
\[
 \Lambda=\{\ast,0,1,2,3\},\qquad
 A_\ast=R,\qquad A_i=L_i\quad(i=0,1,2,3).
\]
The symbol $\ast$ is an additional index for the vertex list $R$.
The theorem says that the required five distinct colors exist if
and only if
\[
 \left|\bigcup_{\lambda\in J}A_\lambda\right|\ge |J|
 \qquad\text{for every }J\subseteq\{\ast,0,1,2,3\}.
\]
Here $J$ selects a subfamily of the five indexed lists. If
$\ast\notin J$, write $J=I$; if $\ast\in J$, write
$J=\{\ast\}\cup I$, where in both cases $I\subseteq\{0,1,2,3\}$.
Thus the complete set of Hall conditions consists of the following
two requirements for every $I\subseteq\{0,1,2,3\}$:
\[
 \left|\bigcup_{i\in I}L_i\right|\ge |I|,
 \qquad
 \left|R\cup\bigcup_{i\in I}L_i\right|\ge |I|+1.
\]
Under the bounds $|L_i|\ge2$ and $|R|\ge4$, these Hall conditions
reduce to
\begin{equation}\label{red:hall}
 \begin{gathered}
 \left|\bigcup_{i\in I}L_i\right|\ge3
 \quad\text{for every }I\subseteq\{0,1,2,3\}\text{ with }|I|=3,\\
 |U|\ge4,\qquad |R\cup U|\ge5.
 \end{gathered}
\end{equation}
Here $I$ is a set of indices specifying which spoke lists are
selected; it is not a set of colors. The first inequality must hold
for all four choices
\[
 I=\{0,1,2\},\quad I=\{0,1,3\},\quad
 I=\{0,2,3\},\quad I=\{1,2,3\}.
\]
For example, the choice $I=\{0,1,3\}$ requires
$|L_0\cup L_1\cup L_3|\ge3$.
To justify the reduction, a subfamily consisting of one or two
spoke lists satisfies Hall's condition automatically, since each
list has at least two colors. A subfamily containing $R$ and at
most three spoke lists has at most four members and a union of size
at least four, so it also satisfies Hall's condition automatically.
The empty subfamily satisfies it trivially. The remaining subfamilies
are precisely three spoke lists, all four spoke lists, and all
five lists; their conditions are exactly \eqref{red:hall}.

If an old edge $u_i u_j$ of color $q$ is recolored with
$p\in L_i\cap L_j$, the operation is proper and the two affected lists
become $(L_i\setminus\{p\})\cup\{q\}$ and
$(L_j\setminus\{p\})\cup\{q\}$.  All such operations below use the
lists in the current coloring.

\begin{lemma}\label{red:low-neighbour}
Let $G\in\mathcal T_6$, and let $v$ have degree four and a neighbor of
degree at most five.  Deleting $v$ and inserting an absent diagonal of
its link gives a smaller graph $H\in\mathcal T_6$.  Every total
eight-coloring of $H$ can be restored to $G$ by discarding the inserted
diagonal and recoloring at most one edge of the link.  All old vertex
colors remain fixed.
\end{lemma}

\begin{proof}
The simple cyclic link is $u_0u_1u_2u_3u_0$.  Its two diagonals cannot
both be edges: together with $v$ they would give a $K_5$, contradicting
planarity.  Delete $v$ and insert an absent diagonal in the resulting
quadrilateral face.  This introduces neither a loop nor a parallel edge
and restores a triangulation.  A diagonal endpoint loses its edge to
$v$ and gains the diagonal, so its degree is unchanged; the other two
rim vertices lose one edge. Thus $H\in\mathcal T_6$ and $|V(H)|<|V(G)|$.

Restrict the supplied coloring of $H$ to $G-v$, and use the lists just
defined. Choose an index $k$ for which $d_G(u_k)\le5$, as guaranteed
by the hypothesis. Then $|L_k|=8-d_G(u_k)\ge3$. We will ensure
\eqref{red:hall} by at most one recoloring.

In each of the following cases, $L_i$ and $U$ refer to the
lists and their union before the proposed recoloring. We write
$L'_i$ for the resulting lists and $U'=\bigcup_{i=0}^3L'_i$
for their union. The list $R$ is unchanged, since only an edge
is recolored and all old vertex colors remain fixed.

First suppose $|U|\le3$. Since $L_k\subseteq U$, we have
$3\le |L_k|\le |U|\le3$, so $|U|=|L_k|=3$ and $L_k=U$.
Define $S:=U=L_0\cup L_1\cup L_2\cup L_3$.
Thus $S$ is a fixed three-element color set, with $L_i\subseteq S$
for every $i$ and $L_k=S$ before the recoloring.
Choose one of the two rim edges $u_ku_j$ whose old color $q$ satisfies
$|R\cup S\cup\{q\}|\ge5$.  Such an edge exists: if $|R\cup S|\ge5$
either works; otherwise this union has size four and the two distinct
old edge colors, both outside $S$, cannot both belong to its single
element outside $S$.  Let $h,\ell$ be the other two indices.
The sets $L_j$ and $L_h\cup L_\ell$ each have at least two elements in
the three-set $S$, so choose
$p\in L_j\cap(L_h\cup L_\ell)$ and recolor $u_ku_j$ with $p$.
After this recoloring, the lists are
\[
 \begin{aligned}
 L'_k&=(S\setminus\{p\})\cup\{q\},&
 L'_j&=(L_j\setminus\{p\})\cup\{q\},\\
 L'_h&=L_h,& L'_\ell&=L_\ell.
 \end{aligned}
\]
To check the first inequality of \eqref{red:hall}, consider each
subfamily of three updated lists.
If $I\subseteq\{0,1,2,3\}$ has $|I|=3$ and $k\in I$, then
\[
 \left|\bigcup_{i\in I}L'_i\right|\ge |L'_k|=3.
\]
The only three-element index set not containing $k$ is
$I=\{j,h,\ell\}$. For this choice, $q$ belongs to $L'_j$,
whereas $q\notin S$ and $L_h\cup L_\ell\subseteq S$.
Moreover, $|L_h\cup L_\ell|\ge2$, since each unchanged list
has at least two colors. Therefore
\[
 |L'_j\cup L'_h\cup L'_\ell|
 \ge |\{q\}\cup L_h\cup L_\ell|
 =1+|L_h\cup L_\ell|\ge3.
\]
Finally, $p$ survives in an unchanged list by its choice in
$L_h\cup L_\ell$, while $L'_k$ contains $S\setminus\{p\}$
and $q$. Thus the union of all four updated lists is
$U':=\bigcup_{i=0}^3L'_i=S\cup\{q\}$.
It has size four, and $|R\cup U'|\ge5$ by the choice of $q$.

Next suppose $|U|\ge4$ but $|R\cup U|\le4$.  Then $R=U$ is a
four-set.  At least one of the two rim edges at $u_k$ has old color
$q\notin R$: their distinct old colors avoid $L_k$, whereas
$|R\setminus L_k|\le1$.  Write this edge as $u_ku_j$.  Since
$|L_k|\ge3$ and $|L_j|\ge2$ inside the four-set $R$, choose
$p\in L_k\cap L_j$ and recolor that edge with $p$.
Let $h,\ell$ be the other two indices. The resulting lists are
\[
 \begin{aligned}
 L'_k&=(L_k\setminus\{p\})\cup\{q\},&
 L'_j&=(L_j\setminus\{p\})\cup\{q\},\\
 L'_h&=L_h,& L'_\ell&=L_\ell.
 \end{aligned}
\]
Since $p\in L_k\subseteq R$ and $q\notin R$, we have
$|L'_k|=|L_k|\ge3$. Every subfamily of three updated edge
lists containing $L'_k$ therefore has a union of size at least
three. For the other three updated lists, $q\in L'_j\setminus R$
and $L_h\cup L_\ell\subseteq R$, so
\[
 |L'_j\cup L'_h\cup L'_\ell|
 \ge 1+|L_h\cup L_\ell|\ge3.
\]
Furthermore, the only old color that can disappear from the
union is $p$, and $q$ is added. Hence
\[
 U'\supseteq (U\setminus\{p\})\cup\{q\},
 \qquad |U'|\ge4.
\]
Every updated list is contained in $R\cup\{q\}$, and
$q\in U'$, so $R\cup U'=R\cup\{q\}$ has size five.

Finally suppose $|U|\ge4$ and $|R\cup U|\ge5$, but the union of some
three edge lists has size at most two. Those three lists cannot
include $L_k$, so the other three lists are all one two-set $T$.
Among their vertices two are consecutive on the four-cycle.
Name their indices $j,h$, and let $\ell$ be the remaining index
other than $k$, so that $L_j=L_h=L_\ell=T$.
Recolor the edge $u_j u_h$, whose old color is $q\notin T$,
with any $p\in T$. The resulting lists are
\[
 L'_j=L'_h=(T\setminus\{p\})\cup\{q\},\qquad
 L'_\ell=T,\qquad L'_k=L_k.
\]
Thus $L'_j\cup L'_h\cup L'_\ell=T\cup\{q\}$ has size three.
Every subfamily of three updated edge lists containing $L'_k$
has a union of size at least three because $|L'_k|=|L_k|\ge3$.
The old colors in $T$ all remain in $L'_\ell$, and $L_k$ is
unchanged. Therefore $U'=U\cup\{q\}$, and consequently
\[
 |U'|\ge |U|\ge4,\qquad
 |R\cup U'|\ge |R\cup U|\ge5.
\]

These cases exhaust the possible failures of \eqref{red:hall}.
The criterion now completes the center and its four spokes, with no
recoloring needed if the inequalities already held.
\end{proof}

\begin{proposition}
\label{red:minimal}
If a graph in $\mathcal T_6$ admits no total eight-coloring, a
counterexample $G$ of minimum order is four-connected, every triangle
of $G$ is facial, every degree is four, five, or six, and every
four-vertex has only six-vertices as neighbors.
\end{proposition}

\begin{proof}
All but the final assertion are Corollary~\ref{red:first-minimal}.
A four-vertex with a neighbor of degree at most five would have the
smaller reduction in Lemma~\ref{red:low-neighbour}.  Minimality would
color that reduction and Lemma~\ref{red:low-neighbour} would restore $G$, a contradiction.
\end{proof}

% LOCAL_WHEEL_SECTION
\section{Extension at a degree-four vertex}
\label{wheel:section}

Restoring a deleted vertex of degree four requires distinct colors
for its four incident edges and a further color for the vertex itself.
A system of distinct representatives for the four spoke lists solves
only the first requirement.
Theorem~\ref{wheel:local} characterizes the obstruction when only
the rim edges may be recolored. It forces the structure
used in Section~\ref{shell:section} and supplies the local step in
Theorem~\ref{restore:main}.

Throughout this section, $G$ is a graph with maximum
degree at most six. A vertex $v$ has degree four, with distinct
neighbors $u_0,u_1,u_2,u_3$, and the edges $u_i u_{i+1}$ form a cycle.
We use the wheel terminology of Section~\ref{prelim:graphs},
with this cycle as the specified rim. Subscripts are read modulo
four unless a path order is specified. Initially, a proper total
coloring $\psi$ of $G-v$ is given from an eight-element color set
$\Omega$. The center and all four spokes are uncolored.
No planarity assumption is needed here.

\subsection{Spoke lists and Hall repair}
\label{wheel:lists}

Write $h_i=\psi(u_i)$ and define the \emph{spoke list}
\[
 L_i=\Omega\setminus
 \bigl(\{h_i\}\cup\{\psi(e):e\in E(G-v),\ u_i\in V(e)\}\bigr).
\]
Thus $L_i$ consists exactly of the colors available for $vu_i$ before
the color of $v$ and those of the other spokes are prescribed.  Since
$\psi$ is a proper total coloring,
\begin{equation}
 \lvert L_i\rvert=8-d_G(u_i)\geq2.
 \label{wheel:list-size}
\end{equation}
Let $\mathcal{H}$ denote the set of colors on the rim vertices
in the supplied coloring, and put
\[
 \mathcal{H}=\{h_0,h_1,h_2,h_3\},\qquad
 R=\Omega\setminus\mathcal{H},\qquad c_i=\psi(u_i u_{i+1}).
\]
The set $\mathcal{H}$ may have fewer than four elements when
some rim vertices have the same color. In each local argument,
$\mathcal{H}$ and $R$ refer to these initial sets, even when a
later restoration changes a rim-vertex color.

A colored edge $e$ is called \emph{$\mathcal{H}$-colored} if its current color
belongs to $\mathcal{H}$, and \emph{$R$-colored} if its current
color belongs to $R$. For the supplied coloring $\psi$, these
conditions mean $\psi(e)\in\mathcal{H}$ and $\psi(e)\in R$,
respectively. After a recoloring, the terminology refers to the
new edge color, while the sets $\mathcal{H}$ and $R$ remain fixed.

To separate the fixed restrictions from the current rim-edge colors,
put
\[
 F_i=\{h_i\}\cup
 \{\psi(e):e\in E(G-v),\ u_i\in V(e),\ e\text{ is nonrim}\}.
\]
Thus $F_i$ records the colors forbidden by $u_i$ and its incident
nonrim edges. Define the \emph{rim color set} $A_i=\Omega\setminus F_i$.
Properness gives $c_{i-1},c_i\notin F_i$ and $c_{i-1}\ne c_i$, so
\[
 L_i=A_i\setminus\{c_{i-1},c_i\},\qquad
 A_i=L_i\cup\{c_{i-1},c_i\}.
\]
If $\psi'$ is a proper recoloring of rim edges only, write
$c'_j=\psi'(u_j u_{j+1})$ and use primes for the corresponding
color sets. All vertex and nonrim-edge colors remain fixed, giving
$F'_i=F_i$, $A'_i=A_i$, and
$L'_i=A_i\setminus\{c'_{i-1},c'_i\}$.
Thus $A_i$ is invariant under such recolorings, although $L_i$ may
change. The identity $A'_i=L'_i\cup\{c'_{i-1},c'_i\}$ must use
the list and rim-edge colors from the same coloring. If a vertex
or nonrim edge is also recolored, $A_i$ need not remain fixed.

When $d_G(u_i)=6$, there are three incident
nonrim edges, so four fixed forbidden colors altogether: the color
$h_i$ and those three edge colors.  In particular, in this case
$\lvert L_i\rvert=2$ and $\lvert A_i\rvert=4$.

For example, suppose $\Omega=\{1,\ldots,8\}$,
$F_i=\{1,2,3,4\}$, and $(c_{i-1},c_i)=(5,6)$.
Then $L_i=\{7,8\}$ and $A_i=\{5,6,7,8\}$.
If $u_{i-1}u_i$ can be properly recolored with $7$, then
$L'_i=\{5,8\}$, but $A'_i=\{5,8\}\cup\{7,6\}=A_i$.

The following update rule governs the rim recolorings. Suppose the
current color of a rim edge $u_i u_j$ is $x$, and
$p\in L_i\cap L_j$ for the \emph{current} lists.  Recoloring that edge
with $p$ is proper: $p$ differs from the colors of both endpoints and
all other edges incident with either endpoint.  At its endpoints the
lists change by
\begin{equation}
 L_i\longmapsto (L_i\setminus\{p\})\cup\{x\},\qquad
 L_j\longmapsto (L_j\setminus\{p\})\cup\{x\}.
 \label{wheel:update}
\end{equation}
The other lists remain unchanged.  All sequences below use this rule
with the current lists, rather than the lists at the beginning of the
sequence.

\begin{lemma}
\label{wheel:path}
Let $P=u_0u_1u_2u_3$ be the path consisting of three consecutive
rim edges. At most two edges of $P$ can be recolored so that the
four resulting spoke lists have distinct representatives.
Every vertex retains its color, as does every edge of $G-v$
outside $E(P)$, including the remaining rim edge $u_3u_0$.
If two path edges change, they are adjacent and receive different
new colors.
\end{lemma}

\begin{proof}
By Hall's theorem (Theorem~\ref{prelim:hall}), applied to the indexed
family $(L_i)_{i=0}^3$, the four lists have distinct
representatives if and only if the union of every subfamily has size
at least the size of that subfamily.  Equation~\eqref{wheel:list-size}
already gives the inequalities for one or two lists.  Hence a failure
is witnessed either by three lists whose union has size two, or by
the union of all four lists having size at most three.

First suppose that three lists have union $T=\{a,b\}$.  Each of these
lists equals $T$.  If their positions are consecutive on the path,
reverse the order if necessary and write
$L_0=L_1=L_2=T$.  Let $x$ and $y$ be the old colors of $u_0u_1$ and
$u_1u_2$.  They are distinct, and both lie outside $T$.  Recolor
$u_0u_1$ with $a$, then $u_1u_2$ with $b$.  The second change is legal
because the first removes only $a$, not $b$, from the list at $u_1$.
The first three lists become $\{b,x\}$, $\{x,y\}$, and
$\{a,y\}$.
Every pair of these lists has union of size at least three, and their
total union has size four.  Every triple among all four lists contains
at least two of these three lists.  All the Hall inequalities therefore
hold, independently of the unchanged fourth list.

If the three positions are not consecutive, reversal reduces the
configuration to $L_0=L_1=L_3=T$.  Let $x$ be the old color of
$u_0u_1$.  Suppose first that $L_2$ contains a color outside
$T\cup\{x\}$.  Recolor $u_0u_1$ with $a$.  The new lists at its
endpoints are both $\{b,x\}$.  A triple containing both of these
lists also contains either $T$, supplying $a$, or $L_2$, supplying a
color outside $\{a,b,x\}$.  A triple containing only one changed
list contains both $L_2$ and $T$, and already contains
$\{a,b,x\}$.  The full union has size at least four.  This gives
all Hall inequalities.

It remains in this configuration to consider
$L_2\subseteq T\cup\{x\}$.  Since $\lvert L_2\rvert\geq2$, choose
$q\in L_2\cap T$, and write $T=\{p,q\}$.  Recolor $u_0u_1$ with
$p$.  If $z$ was the color of $u_1u_2$, then
$z\notin T\cup\{x\}$: it was forbidden at $u_1$, and differed from
the adjacent edge color $x$.  Recolor $u_1u_2$ with $q$, which is
still available at both endpoints.  The four lists are now
\[
 \{q,x\},\qquad \{x,z\},\qquad
 (L_2\setminus\{q\})\cup\{z\},\qquad \{p,q\}.
\]
Their full union contains the four different colors $p,q,x,z$.
The first and second lists together contain $q,x,z$, so either triple
containing them satisfies the corresponding Hall inequality.
The triple omitting the second list
contains $p,q,x$ from the first and fourth lists.  The triple omitting
the first list contains $p,q,x,z$ from the second and fourth lists.
Again all inequalities hold.

We may now assume that every triple initially has union of size at
least three. If Hall's conditions still fail, the union of all four
lists is a three-element set $Q$. We claim that there are three
distinct indices $i,j,k\in\{0,1,2,3\}$ such that
\[
 u_i u_j\in E(P),\qquad L_i\cup L_j=Q,\qquad
 L_i\cap L_j\cap L_k\neq\varnothing.
\]
Thus the first two lists correspond to adjacent vertices of the
path, and a common color of these lists also belongs to a third
list whose index is neither $i$ nor $j$.

If some $L_i$ equals $Q$, choose any neighbor $u_j$ of $u_i$ in
$P$, and choose any index $k\notin\{i,j\}$. The sets $L_j,L_k$
are subsets of the three-element set $Q$, each of size at least
two. Hence $L_j\cap L_k\neq\varnothing$; since $L_i=Q$, this
intersection is contained in $L_i$ as well. The claim follows.

Otherwise all four lists have size two. They cannot all be equal,
since their union is $Q$. Because $P$ is connected, there are
adjacent vertices $u_i,u_j$ with $L_i\neq L_j$: equality along
every edge of $P$ would force all four lists to be equal.
These two distinct two-element subsets of $Q$ satisfy
$L_i\cup L_j=Q$ and $L_i\cap L_j=\{a\}$ for some color $a$.
Let $\{k,\ell\}=\{0,1,2,3\}\setminus\{i,j\}$
be the other two indices. If $a\in L_k\cup L_\ell$, the claim
already holds. If $a\notin L_k\cup L_\ell$, then
$L_k=L_\ell=Q\setminus\{a\}$, since each list has two elements.

The two nonempty vertex sets $\{u_i,u_j\}$ and $\{u_k,u_\ell\}$
partition $V(P)$. Since $P$ is connected, it has an edge $u_pu_q$
with $p\in\{i,j\}$ and $q\in\{k,\ell\}$. Let $r$ be the unique
index in $\{k,\ell\}\setminus\{q\}$. As $p\in\{i,j\}$, we can
write $L_p=\{a,b\}$ for some $b\in Q\setminus\{a\}$. We also
have $L_q=L_r=Q\setminus\{a\}$. Consequently
\[
 L_p\cup L_q=Q,\qquad L_p\cap L_q=\{b\},\qquad b\in L_r.
\]
Thus $p,q,r$ are distinct indices satisfying the claim, even though
$L_q=L_r$ as color sets.

Choose indices $i,j,k$ satisfying the claim and a color
$a\in L_i\cap L_j\cap L_k$. The old color $z$ of $u_i u_j$
belongs to neither endpoint list. Since $L_i\cup L_j=Q$,
we have $z\notin Q$.
Recolor $u_i u_j$ with $a$, which is available at both endpoints.
The list $L_k$ is unchanged because $k\notin\{i,j\}$, and hence
still contains $a$. The updated endpoint lists together contain
$(Q\setminus\{a\})\cup\{z\}$, so the union of all four updated
lists contains $Q\cup\{z\}$. Every subfamily of three updated lists
contains a changed list with the color $z$ and an unchanged list
with at least two colors of $Q$, giving at least three colors.
Theorem~\ref{prelim:hall} therefore supplies distinct representatives
for the four updated lists.
All constructions make at most two changes; the only constructions
with two changes use adjacent edges and different new colors.
\end{proof}

\begin{corollary}
\label{wheel:four-neighbours}
If $\lvert \mathcal{H}\rvert\leq3$, the coloring extends to $G$ after changing
only rim-edge colors and assigning colors to $v$ and its spokes.
\end{corollary}

\begin{proof}
Apply Lemma~\ref{wheel:path} to any path obtained by deleting one edge
of the rim.  Assign distinct representatives to the four spokes.
Their four colors, together with the at most three colors in $\mathcal{H}$,
exclude at most seven colors at $v$.
\end{proof}

\subsection{The two obstructions}

For the remainder of this section, assume that each of the four
neighbors $u_0,u_1,u_2,u_3$ of the degree-four vertex $v$
has degree six in $G$.

Write $C=u_0u_1u_2u_3u_0$ for the specified rim. A
\emph{local restoration of $\psi$} is a proper total coloring
$\widehat{\psi}$ of $G$ from the same color set $\Omega$ such that
\[
 \begin{aligned}
 \widehat{\psi}(x)&=\psi(x)
   &&\text{for every }x\in V(G)\setminus\{v\},\\
 \widehat{\psi}(e)&=\psi(e)
   &&\text{for every }e\in E(G-v)\setminus E(C).
 \end{aligned}
\]
Thus only the new vertex and spokes and the four rim edges may
receive new colors. All old vertices and nonrim edges remain fixed,
including any chord of $C$. Failure of this restricted restoration
does not preclude a total coloring obtained by other recolorings.

\begin{lemma}
\label{wheel:boundary}
If $\lvert \mathcal{H}\rvert=4$ and $L_i\cap \mathcal{H}\neq\varnothing$ for some $i$,
then a local restoration exists.
\end{lemma}

\begin{proof}
Choose $h\in L_i\cap \mathcal{H}$.  We will choose a rim path and perform the
repair of Lemma~\ref{wheel:path} in such a way that $h$ remains in at
least one spoke list.

Suppose first that $h$ is available at two opposite rim vertices. Use any
rim path.  At most one edge in the repair receives the color $h$,
because if two edges change, their new colors are different.  That
edge cannot meet both opposite rim vertices. By
\eqref{wheel:update}, a list loses $h$ only when an incident edge is
recolored with $h$.  Therefore $h$ survives at one of the opposite
rim vertices. This also treats availability at three or more rim vertices.

In the remaining case, $h$ is available at only one rim vertex or at two
consecutive rim vertices. Since $\lvert \mathcal{H}\rvert=4$, exactly
one rim vertex has color $h$. An edge of rim color $h$ cannot be incident
with that vertex. The other two rim edges meet, so at most one rim
edge has color $h$.  If such an edge exists, its endpoints and the
$h$-colored rim vertex exclude $h$ from three lists. Thus only the fourth
rim vertex can have $h$ available; omit the $h$-colored edge when choosing
the path. If no rim edge has color $h$ and two consecutive rim vertices
have $h$ available, omit their joining edge. If just one rim vertex has
$h$ available and no rim edge has color $h$, omit any rim edge.
In every case the chosen path has no edge of color $h$, and the
positions at which $h$ is available form an independent set in that
path.

There cannot be a first edge of the chosen path recolored with $h$.
Before that first change, no path edge has carried or released $h$,
so \eqref{wheel:update} shows that the set of positions with $h$
available has not changed.  Its independence prevents an edge from
having $h$ available at both endpoints.  Consequently the path repair
retains $h$ in this case as well.

Choose distinct representatives of the repaired lists.  If none is
$h$, replace the representative of a list containing $h$ by $h$;
distinctness is preserved, since $h$ was unused.  The four spoke
colors now meet $\mathcal{H}$.  Their union with $\mathcal{H}$ has size at most seven,
leaving a color for $v$.
\end{proof}

\begin{theorem}
\label{wheel:local}
The supplied coloring $\psi$ has no local restoration if and only
if one of the following states holds:
\begin{enumerate}
\item[\textup{(U)}] $\lvert \mathcal{H}\rvert=4$ and $A_i=R$ for every $i$;
\item[\textup{(D)}] $\lvert \mathcal{H}\rvert=4$, every rim color $c_i$ lies
in $\mathcal{H}$, and there are disjoint two-element sets $P,Q$ such that
\[
 P\cup Q=R,\qquad L_0=L_2=P,\qquad L_1=L_3=Q.
\]
\end{enumerate}
In state \textup{(D)}, the four rim colors are distinct and exhaust
$\mathcal{H}$, and
\begin{equation}
 A_i\cap A_{i+1}=\{c_i\}\qquad(i\in\mathbb Z/4\mathbb Z).
 \label{wheel:forced-rim}
\end{equation}
\end{theorem}

\begin{proof}
In state (U), each rim color set remains $R$ under all
permitted changes.  Every spoke must therefore have a color in $R$;
the center must also have a color in $R$ because its four neighbors
use all of $\mathcal{H}$.  The center and its four spokes are five pairwise
conflicting elements, which cannot be colored from the four-element
set $R$.

In state (D), the color $h_i$ cannot occur on either rim edge
incident with $u_i$. The other two rim edges share $u_{i+2}$,
so at most one of them has color $h_i$. Thus each color of
$\mathcal{H}$ occurs on at most one rim edge. Since all four rim
colors belong to $\mathcal{H}$, they are distinct and exhaust it.
Now $A_i=L_i\cup\{c_{i-1},c_i\}$, with $L_i\subseteq R$,
$c_{i-1},c_i\in\mathcal{H}$, and $R\cap\mathcal{H}=\varnothing$.
Since consecutive lists are disjoint,
\[
 A_i\cap A_{i+1}
 =(L_i\cap L_{i+1})\cup
   (\{c_{i-1},c_i\}\cap\{c_i,c_{i+1}\})=\{c_i\},
\]
proving \eqref{wheel:forced-rim}. In any proper recoloring of the
rim, the color of $u_i u_{i+1}$ must lie in the fixed intersection
$A_i\cap A_{i+1}$. Every rim edge therefore retains its color,
and the spoke lists remain $P$ or $Q$. The center and four spokes
must again receive five distinct colors from $R$, which is impossible.

Conversely, suppose that no local restoration of $\psi$ exists. Corollary
\ref{wheel:four-neighbours} gives $\lvert \mathcal{H}\rvert=4$.
Lemma~\ref{wheel:boundary} gives $L_i\subseteq R$ for every $i$.
All these lists have size two.  If an edge of rim color in $\mathcal{H}$ has
intersecting endpoint lists, recolor that edge with a common color.
Its old color is then released into both lists and lies in $\mathcal{H}$.
Lemma~\ref{wheel:boundary} restores the center, a contradiction.
Therefore the endpoint lists of every $\mathcal{H}$-colored rim edge are
disjoint, and hence complementary in the four-set $R$.

Suppose that both $\mathcal{H}$-colored and $R$-colored rim edges occur.  An
adjacent pair of the two types exists; reindex so that $u_0u_1$ has
color $b\in \mathcal{H}$ and $u_1u_2$ has color $r\in R$.  The lists $L_0$
and $L_1$ are complementary.  The two-element lists $L_1,L_2$ both
avoid $r$, so they intersect in the three-set $R\setminus\{r\}$.
Choose $p\in L_1\cap L_2$ and recolor $u_1u_2$ with $p$.  This
releases $r$ into the current list at $u_1$.  The original $L_1$
avoided $r$, so complementarity gives $r\in L_0$, whose list has
not changed.  We can now recolor $u_0u_1$ with $r$, releasing
$b\in \mathcal{H}$.  Lemma~\ref{wheel:boundary} again gives a restoration.
This contradiction excludes a mixture of the two types.

If all rim edges have colors in $R$, then at each rim vertex the two
rim colors and the two list colors are four distinct elements of
$R$.  Thus $A_i=R$ for every $i$, giving (U).  If all rim colors
belong to $\mathcal{H}$, consecutive two-element lists are complementary in
$R$; they alternate around the rim, giving (D).  These exhaust the
possibilities.
\end{proof}

For example, a state (U) may have rim-vertex colors $1,2,3,4$,
rim-edge colors $5,6,7,8$ in cyclic order, and
$A_i=R=\{5,6,7,8\}$. Distinct representatives for the four
spoke lists then use all of $R$, leaving no color for the center.
In state (D), \eqref{wheel:forced-rim} imposes the additional
restriction that no rim edge can change color while the vertex
and nonrim-edge colors stay fixed.

Lemma~\ref{wheel:boundary} shows that a proper recoloring that
keeps all old vertex colors fixed and achieves
$L_i\cap\mathcal{H}\ne\varnothing$ for some $i$ permits restoration.
The next two propositions instead allow one rim vertex to change color.

\subsection{Releasing a rim vertex}

Here the rim is required to be \emph{induced}: no edge joins two
nonconsecutive rim vertices. Since every rim vertex has degree six
in $G$, it has exactly three neighbors outside the rim and the
center, called its \emph{exterior neighbors}. Different rim vertices
may have exterior neighbors in common.

\begin{proposition}
\label{wheel:uniform}
If the rim is induced, a state \textup{(U)} in
Theorem~\ref{wheel:local} can be restored by changing one
rim-vertex color and the rim-edge colors. Every
nonrim edge and every other old vertex retains its color.
\end{proposition}

\begin{proof}
Since $A_i=\Omega\setminus F_i=R=\Omega\setminus\mathcal{H}$,
we have $F_i=\mathcal{H}$. The three nonrim edges at $u_i$
therefore have exactly the colors in $\mathcal{H}\setminus\{h_i\}$.
These are the colors in the supplied coloring, before we recolor $u_0$.

There are three exterior neighbors of $u_0$, while $\lvert R\rvert=4$.
Choose $p\in R$ absent from their \emph{current} vertex colors, and
write $R=\{p,r,s,t\}$.  Erase the rim-edge colors and make the
following assignments:
\begin{equation}
 \begin{aligned}
 \psi(u_0)&=p,&\qquad \psi(v)&=r,\\
 (\psi(vu_0),\psi(vu_1),\psi(vu_2),\psi(vu_3))
   &=(h_0,p,s,t),\\
 (\psi(u_0u_1),\psi(u_1u_2),\psi(u_2u_3),\psi(u_3u_0))
   &=(t,r,p,s).
 \end{aligned}
 \label{wheel:uniform-assignment}
\end{equation}
Here $h_0$ denotes the former color of $u_0$.

At $u_0$, the three exterior edge colors are exactly
$\mathcal{H}\setminus\{h_0\}$, so they avoid both the new vertex color $p$
and the new spoke color $h_0$.  The two rim neighbors have colors
in $\mathcal{H}$, and the exterior neighbors avoid $p$ by its choice.  The
three wheel edges incident with $u_0$ have different colors
$h_0,t,s$, none equal to $p$.
At $u_1,u_2,u_3$, the vertex color and all exterior edge colors
belong to $\mathcal{H}$.  The three wheel edges at these vertices have colors,
respectively, $p,t,r$; $s,r,p$; and $t,p,s$.  Each triple is distinct
and lies in $R$, so it avoids all fixed forbidden colors there.
The center has color $r$, whereas the four spokes have different
colors $h_0,p,s,t$, and its four neighbors have colors
$p,h_1,h_2,h_3$.  Thus it conflicts with none of them.  No other
constraint has changed.
\end{proof}

\begin{proposition}
\label{wheel:witness}
Suppose the rim is induced and the coloring is in state
\textup{(D)} of Theorem~\ref{wheel:local}.
If some $p\in L_i$ is absent from the vertex colors of the three
exterior neighbors of $u_i$, the center can be restored by changing
only $u_i$ among the old vertices and keeping every old edge fixed.
\end{proposition}

\begin{proof}
Recolor $u_i$ with $p$.  Its incident edges avoid $p$ because
$p\in L_i$.  Its two rim neighbors have colors in $\mathcal{H}$, disjoint
from $R$, and its exterior neighbors avoid $p$ by hypothesis.
Hence this change is proper. Give the center the former color $h_i$
of $u_i$. Give $vu_i$ the other element of $L_i$, give the
opposite spoke $vu_{i+2}$ the color $p$, and give the two remaining
spokes the two colors of $R\setminus L_i$.  All four spoke colors
are different.  At each unchanged rim vertex the assigned color lies
in its original spoke list; at $u_i$ it also differs from its new
vertex color $p$.  The center color $h_i$ avoids all spoke colors,
which lie in $R$, and all its neighbor colors: the old four
neighbor colors were distinct, and $h_i$ has been replaced by $p$.
\end{proof}

Suppose the rim is induced and the coloring is in state (D).
If no completion is possible by recoloring only $u_i$ among the old
vertices and keeping all old edges fixed, the contrapositive of
Proposition~\ref{wheel:witness} implies that every color $p\in L_i$ must occur
on an exterior neighbor of $u_i$. Such a neighbor is called
an \emph{exterior witness}. This gives a necessary
condition for failure; no converse is asserted. Witnesses for
different rim vertices need not be distinct: one exterior vertex
adjacent to two rim vertices may witness the same color for both.

\section{The degree-four shell}
\label{shell:section}

Exterior witness colors and two proper boundary colorings will show
that the four-vertices of a minimum counterexample have disjoint
neighborhoods. The \emph{shell} is
the wheel centered at $v$ together with the exterior vertices
labeled $x_i,y_i$ below and the triangular faces incident with
the four rim vertices $u_i$. The eight exterior labels are not
initially assumed to denote distinct vertices. Lemma~\ref{shell:distinct}
establishes their distinctness; Figure~\ref{fig:shell} depicts the
resulting configuration.

\begin{figure}[!htbp]
\centering
\begin{tikzpicture}[scale=1.02, every node/.style={font=\small}]
 \coordinate (v) at (0,0);
 \foreach \i in {0,1,2,3} {
   \pgfmathsetmacro{\ang}{90-90*\i}
   \pgfmathsetmacro{\xa}{45-90*\i}
   \coordinate (u\i) at (\ang:1.08);
   \coordinate (y\i) at (\ang:2.65);
   \coordinate (x\i) at (\xa:2.65);
 }
 \foreach \i in {0,1,2,3} {
   \pgfmathtruncatemacro{\nxt}{mod(\i+1,4)}
   \pgfmathtruncatemacro{\prv}{mod(\i+3,4)}
   \draw[thick] (v)--(u\i)--(u\nxt);
   \draw (u\i)--(x\prv) (u\i)--(y\i) (u\i)--(x\i);
   \draw[thick] (y\i)--(x\i)--(y\nxt);
   \fill (u\i) circle (1.7pt);
   \fill (x\i) circle (1.7pt);
   \fill (y\i) circle (1.7pt);
   \pgfmathsetmacro{\ang}{90-90*\i}
   \pgfmathsetmacro{\xa}{45-90*\i}
   \node at (\ang:2.97) {$y_{\i}$};
   \node at (\xa:2.96) {$x_{\i}$};
 }
 \fill (v) circle (1.7pt);
 \node[below left] at (v) {$v$};
 \node at (-0.45,1.05) {$u_0$};
 \node at (1.05,0.45) {$u_1$};
 \node at (0.45,-1.05) {$u_2$};
 \node at (-1.05,-0.45) {$u_3$};
\end{tikzpicture}
\caption{The shell configuration established in Lemma~\ref{shell:distinct}.
Its eight exterior vertices are pairwise distinct. The three exterior
neighbors of $u_i$ are $x_{i-1},y_i,x_i$.  The drawing records the
actual facial disk forced by the links; the remainder of the sphere
triangulation lies outside it.  Additional edges in that remainder
are not excluded by this figure.}
\label{fig:shell}
\end{figure}
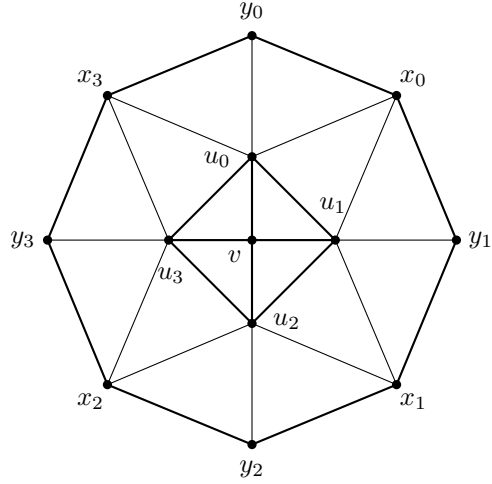
\FloatBarrier

Throughout this section, suppose $G\in\mathcal T_6$ has no total
eight-coloring and has the fewest vertices among all such graphs.
Proposition~\ref{red:minimal} applies. For the local arguments,
fix any vertex $v\in V_4(G)$, if one exists, and write its cyclic link as
$u_0u_1u_2u_3u_0$, with indices modulo four.  Every $u_i$ has degree
six.  The link is induced: a diagonal between opposite link vertices
would form with $v$ a triangle having the other two link vertices on
opposite sides, contradicting that every triangle is facial.
The final simultaneous reduction concerns the whole set $V_4(G)$
and also allows this set to be empty.

Delete $v$ and insert either diagonal into its link disk.  Because the
diagonal was absent, the result is a smaller simple sphere
triangulation.  Its diagonal endpoints retain degree six and its other
two rim vertices lose one incident edge, so it is in $\mathcal T_6$.
Minimality gives a total eight-coloring.  Discard the added diagonal,
and denote the resulting proper total coloring of $G-v$ by $\psi$.
Set
\begin{align*}
 h_i&=\psi(u_i),&
 c_i&=\psi(u_i u_{i+1}),\\
 \mathcal{H}&=\{h_0,h_1,h_2,h_3\},&
 R&=\Omega\setminus \mathcal{H},\\
 L_i&=\Omega\setminus\bigl(\{h_i\}\cup
           \{\psi(e):e\in E(G-v),\ u_i\in V(e)\}\bigr),&
 A_i&=L_i\cup\{c_{i-1},c_i\}.
\end{align*}
Each $L_i$ is a two-set. Since $G$ is a counterexample,
Theorem~\ref{wheel:local} implies that $|\mathcal{H}|=4$ and the state is
either $\mathrm U$ or $\mathrm D$. Proposition~\ref{wheel:uniform}
restores state $\mathrm U$ by allowing one old vertex to change, so
the state must be
$\mathrm D$. We retain the notation $P,Q$ of
Theorem~\ref{wheel:local} for the two alternating lists:
\begin{equation}\label{shell:D}
 \begin{gathered}
 |\mathcal{H}|=|R|=4,\qquad
 L_0=L_2=P,\quad L_1=L_3=Q,\\
 |P|=|Q|=2,\quad P\cap Q=\varnothing,\quad P\cup Q=R,\qquad
 \{c_0,c_1,c_2,c_3\}=\mathcal{H}.
 \end{gathered}
\end{equation}
The last four colors are distinct.

Let $x_i$ be the third vertex of the exterior face on the rim edge
$u_i u_{i+1}$; its other incident face has third vertex $v$.
The inducedness of the rim implies that $x_i$ is neither $v$ nor any
rim vertex.  Since $u_i$ has degree six, its complete cyclic link is
\begin{equation}\label{shell:link}
 v,\ u_{i-1},\ x_{i-1},\ y_i,\ x_i,\ u_{i+1}.
\end{equation}
This defines $y_i$.  It is exterior to the rim and different from $v$.
For a fixed $i$, the three vertices $x_{i-1},y_i,x_i$ are distinct
because the link of $u_i$ is a simple cycle. We do \emph{not} yet
assume that differently labeled exterior vertices in the links of
different $u_i$ are distinct. The faces in this link give, among
others, the edges $x_{i-1}y_i$ and $y_i x_i$.

\subsection{Exterior witnesses}

For counting purposes, call the triple $(x_{i-1},y_i,x_i)$ of
exterior neighbors of $u_i$ \emph{row $i$}. Its \emph{row index}
is $i\in\{0,1,2,3\}$, identifying the rim vertex $u_i$.
For each $p\in L_i$, the ordered pair $(i,p)$ is a \emph{demand}:
it specifies the requirement that at least one vertex in row $i$
have color $p$. A vertex $z$ in that row with $\psi(z)=p$ is an
\emph{exterior witness for $(i,p)$}. For example, since $L_0=L_2$,
each $p\in L_0$ gives two distinct demands $(0,p)$ and $(2,p)$:
the requested color is the same, but the row indices are different.

\begin{lemma}\label{shell:witnesses}
Let $G$ be the minimum counterexample fixed above, let $v\in V_4(G)$
be the chosen vertex, and use the coloring $\psi$ and shell notation
constructed above.
For every $i$ and $p\in L_i$, the demand $(i,p)$ has an exterior
witness.  Each demand has exactly one witness.  After possibly reversing
the cyclic indexing,
\begin{equation}\label{shell:orientation}
 \psi(x_i)\in L_i,
 \qquad \{\psi(x_i),\psi(y_i)\}=L_i
 \quad\text{for every }i.
\end{equation}
\end{lemma}

\begin{proof}
The rim is induced, and $\psi$ is in state \textup{(D)} by
\eqref{shell:D}. If some $p\in L_i$ were absent from the vertex colors
of the three exterior neighbors of $u_i$, Proposition~\ref{wheel:witness}
would restore $v$ and give a proper total eight-coloring of $G$,
contrary to the choice of $G$. Thus every demand has an exterior witness.

There are eight demands, two in each of the four rows. For each
$i\in\{0,1,2,3\}$ and $p\in L_i$, let
\[
 n_{i,p}=
 \bigl|\{z\in\{x_{i-1},y_i,x_i\}:\psi(z)=p\}\bigr|.
\]
The total number of demand--witness pairs is
\[
 W=\sum_{i=0}^{3}\ \sum_{p\in L_i} n_{i,p}\geq8,
\]
because every demand has at least one witness.

To bound $W$ from above, group its contributions according
to the eight labels $x_0,\ldots,x_3,\allowbreak y_0,\ldots,y_3$. The label $y_i$
appears only in row $i$ and has one color, so it contributes at most
one. The label $x_i$ appears in rows $i$ and $i+1$. Since
$L_i\cap L_{i+1}=\varnothing$, its color can meet at most one demand
in these two rows combined. Thus $W\leq4+4=8$. Equality follows:
each of the eight integers $n_{i,p}$ equals one, and each of the
eight labels contributes exactly one.

The row index is retained in this count. Different labels may still
denote the same vertex in different rows; a witness in each such row
is counted for that row's demand. Within a single row, the three
vertices are distinct by \eqref{shell:link}, so no demand--witness
pair is counted twice. The argument therefore does not assume
distinctness of all eight exterior vertices.

In particular $\psi(y_i)\in L_i$, and precisely one of the two
$x$-positions in row $i$ supplies its other color.  Put
$\epsilon_i=1$ if $\psi(x_i)\in L_i$ and $\epsilon_i=0$ if
$\psi(x_i)\in L_{i+1}$. The equality $W=8$ shows these are the
only possibilities.  Row $i$ receives
$\epsilon_i+(1-\epsilon_{i-1})$ witnesses from its $x$-positions, and
this number is one. Hence $\epsilon_i=\epsilon_{i-1}$ for every $i$,
so either all $\epsilon_i$ equal one or all equal zero.

If all equal one, then $\psi(x_i)\in L_i$ already holds for every $i$.
If all equal zero, reverse the direction of numbering around the rim.
More explicitly, with all subscripts taken modulo four, put
\[
 \widetilde u_i=u_{-i},\qquad
 \widetilde x_i=x_{-i-1},\qquad
 \widetilde y_i=y_{-i},\qquad
 \widetilde L_i=L_{-i}.
\]
The vertex $\widetilde x_i$ is the exterior third vertex on the new
rim edge $\widetilde u_i\widetilde u_{i+1}$, as required. Since the
old $\epsilon_{-i-1}$ equals zero,
\[
 \psi(\widetilde x_i)=\psi(x_{-i-1})
 \in L_{-i}=\widetilde L_i.
\]
Thus every new $\epsilon_i$ equals one. This changes only the labels,
not the graph or its coloring; henceforth we drop the tildes and use
the new indexing for all indexed quantities.

In either case, both $\psi(x_i)$ and $\psi(y_i)$ belong to $L_i$.
The edge $x_i y_i$ belongs to $G-v$ by \eqref{shell:link}, so the
proper coloring $\psi$ gives $\psi(x_i)\ne\psi(y_i)$.
As $|L_i|=2$, these two colors exhaust $L_i$, proving
\eqref{shell:orientation}.
\end{proof}

\subsection{Distinct exterior vertices}

\begin{lemma}\label{shell:distinct}
Consider the shell of the chosen vertex $v\in V_4(G)$ in the minimum
counterexample $G$, with the labels defined in \eqref{shell:link}.
The eight vertices $x_0,x_1,x_2,x_3,y_0,y_1,y_2,y_3$ are pairwise
distinct.  The rim neighbors of $x_i$ are precisely $u_i,u_{i+1}$;
the only rim neighbor of $y_i$ is $u_i$.
\end{lemma}

\begin{proof}
Consecutive $x$-vertices occur in a common simple link and are distinct.
If $x_i=x_{i+2}=z$, then $z$ is adjacent to all four rim vertices.
The triangle $zu_{i+1}u_{i+2}$ is facial, since every triangle of $G$
is facial.  Its third vertex is not $v$, so it must be the unique
exterior third vertex $x_{i+1}$ on that rim edge.  This identifies
consecutive $x$-vertices, a contradiction.  All four $x_i$ are distinct.

The two links containing $x_i$ already show that $x_i\ne y_i,y_{i+1}$.
If $x_i=y_{i+2}$, its adjacency to $u_{i+1}$ and $u_{i+2}$ and the same
facial-triangle argument give $x_i=x_{i+1}$, impossible.  If
$x_i=y_{i-1}$, use the edge $u_{i-1}u_i$ to obtain the equally
impossible $x_i=x_{i-1}$.  Hence no $x$-vertex equals a $y$-vertex.

If $y_i=y_{i+1}$, adjacency to the consecutive rim vertices forces it
to be $x_i$, which is excluded.  It remains to rule out
$y_i=y_{i+2}=z$.  In that case $z$ is adjacent to the six distinct
vertices
\[
 u_i,\ u_{i+2},\ x_{i-1},\ x_i,\ x_{i+1},\ x_{i+2}.
\]
They exhaust its neighbors because its degree is at most six.
Its simple cyclic link contains the two disjoint three-vertex paths
\[
 x_{i-1},u_i,x_i
 \qquad\text{and}\qquad
 x_{i+1},u_{i+2},x_{i+2},
\]
since the indicated triangles are faces. In this link, the middle
vertices $u_i,u_{i+2}$ already have their two incident link edges.
Each of the four endpoints needs exactly one further link edge,
so the remaining two link edges must pair these endpoints.

If each path's own endpoints were joined, those two edges would be
$x_{i-1}x_i$ and $x_{i+1}x_{i+2}$. The link would then consist of
the two vertex-disjoint triangles
\[
 x_{i-1}u_i x_i x_{i-1}
 \qquad\text{and}\qquad
 x_{i+1}u_{i+2}x_{i+2}x_{i+1}.
\]
All six vertices would already have degree two in this union,
so no further link edge could connect the triangles. This contradicts
the fact that the link of $z$ is a single simple six-cycle.
Consequently, each remaining link edge must join an endpoint of
one path to an endpoint of the other. The two possibilities for
these link edges are
\[
 \{x_{i-1}x_{i+1},\,x_i x_{i+2}\}
 \qquad\text{or}\qquad
 \{x_{i-1}x_{i+2},\,x_i x_{i+1}\}.
\]

By \eqref{shell:orientation}, the equality $y_i=y_{i+2}$, together with
$L_i=L_{i+2}$, implies $\psi(x_i)=\psi(x_{i+2})$: both are the other
member of the same two-set.  The cross-pairing containing
$x_i x_{i+2}$ is impossible because it would join identically colored
vertices. The other pairing contains $x_i x_{i+1}$. By
\eqref{shell:link}, the two paths from $x_i$ to $x_{i+1}$ in the
complete cyclic link of $u_{i+1}$ are
\[
 x_i-y_{i+1}-x_{i+1}
 \qquad\text{and}\qquad
 x_i-u_i-v-u_{i+2}-x_{i+1}.
\]
Both paths have internal vertices, so $x_i$ and $x_{i+1}$ are not
consecutive in this link. A triangular face incident with $u_{i+1}$
must have its other two vertices consecutive in that link.
Hence the triangle $u_{i+1}x_i x_{i+1}$ would not be facial.
This contradicts Proposition~\ref{red:minimal}
and excludes the final identification.

Finally, \eqref{shell:link} gives the entire neighbor set of each
rim vertex:
\[
 N_G(u_j)=\{v,u_{j-1},x_{j-1},y_j,x_j,u_{j+1}\}.
\]
Since all eight exterior vertices are pairwise distinct and lie outside
the rim, $x_i\in N_G(u_j)$ is possible only as $x_{j-1}$ or $x_j$,
giving $j=i+1$ or $j=i$. Similarly, $y_i\in N_G(u_j)$ forces
$y_i=y_j$, hence $j=i$. All these specified adjacencies do occur.
Since $N_G(v)=\{u_0,u_1,u_2,u_3\}$, we have proved
\[
 N_G(x_i)\cap N_G(v)=\{u_i,u_{i+1}\},
 \qquad
 N_G(y_i)\cap N_G(v)=\{u_i\}.
\]
\end{proof}

\begin{corollary}\label{shell:no-four-cycle}
Let $G$ be the minimum counterexample fixed above and let $v\in V_4(G)$.
There is no induced four-cycle of $G$ containing $v$.
\end{corollary}

\begin{proof}
Such a cycle would have the form $vu_i z u_j v$.  If $u_i,u_j$ were
consecutive in the rim, their rim edge would be a chord.  Thus they
must be opposite.  If $z$ were another rim vertex, $vz$ would be a
chord.  Otherwise $z$ would belong to the exterior-neighbor triples
of two opposite rim vertices.  Those triples are disjoint by
Lemma~\ref{shell:distinct}, which is impossible.
\end{proof}

\subsection{Boundary states and degree bounds}

The witness count gives color information on the exterior vertices.
To turn that information into degree bounds, we change the rim
coloring while keeping the entire exterior fixed.  The two states
below have the same missing lists, so the witness argument applies
in both states.

\begin{lemma}\label{shell:states}
Let $G$ be the minimum counterexample fixed above, let $v\in V_4(G)$
be the chosen vertex, and let $\psi$ be the coloring constructed above.
There is a second proper total coloring $\psi'$ of $G-v$ which changes
only the four rim vertices and four rim edges of $\psi$, has the same lists $L_i$,
and satisfies
\begin{equation}\label{shell:halfturn}
 h_i'=\psi'(u_i)=h_{i+2}\quad\text{for every }i.
\end{equation}
\end{lemma}

\begin{proof}
Recall that $h_i=\psi(u_i)$ and $c_i=\psi(u_i u_{i+1})$, with all
indices taken modulo four. The four $h_i$ are pairwise distinct,
and so are the four $c_i$, with both sets equal to $\mathcal H$.
Since $u_i$ is incident with the two rim edges of colors $c_{i-1},c_i$,
properness gives $h_i\in\{c_{i+1},c_{i+2}\}$.
We first show that the same one of these two formulas holds at every
index.

Suppose $h_i=c_{i+1}$ for some $i$. At the preceding index $i-1$,
the two possible formulas are $h_{i-1}=c_i$ and $h_{i-1}=c_{i+1}$.
The latter would give $h_{i-1}=h_i$, contradicting the distinctness
of the rim-vertex colors. Hence $h_{i-1}=c_i$.
Repeating this implication at indices $i-1,i-2,i-3$ proves
$h_j=c_{j+1}$ for every $j$.

If instead $h_i=c_{i+2}$, then at the next index $i+1$ the possibilities
are $h_{i+1}=c_{i+2}$ and $h_{i+1}=c_{i+3}$.
The former would give $h_{i+1}=h_i$, so $h_{i+1}=c_{i+3}$.
Repeating in the increasing-index direction gives
$h_j=c_{j+2}$ for every $j$. Thus the two possible configurations are
\[
 \begin{aligned}
 \textup{(I)}\quad &h_i=c_{i+1} &&\text{for every }i,\\
 \textup{(II)}\quad &h_i=c_{i+2} &&\text{for every }i.
 \end{aligned}
\]

We next determine the colors of the three edges from $u_i$ to its
exterior neighbors, namely $u_i x_{i-1},u_i y_i,u_i x_i$.
The vertex $u_i$ and its five incident edges in $G-v$ have six distinct
colors, whose set is
$\Omega\setminus L_i=\mathcal H\cup(R\setminus L_i)$.
The vertex and the two rim edges use the three distinct colors
$h_i,c_{i-1},c_i\in\mathcal H$. Therefore the three exterior edges
use the two colors of $R\setminus L_i$ and the single color $f_i$
defined by $\{f_i\}=\mathcal H\setminus\{h_i,c_{i-1},c_i\}$.
In configuration \textup{(I)}, $f_i=c_{i+2}$; in configuration
\textup{(II)}, $f_i=c_{i+1}$.

Define $\psi'$ by changing the rim-vertex and rim-edge colors
simultaneously. Write $h_i'=\psi'(u_i)$ and
$c_i'=\psi'(u_i u_{i+1})$, and set
\[
 \begin{aligned}
 \textup{(I)}\quad &c_i'=c_{i+1},\qquad h_i'=c_{i-1}
     &&\text{for every }i,\\
 \textup{(II)}\quad &c_i'=c_{i-1},\qquad h_i'=c_i
     &&\text{for every }i.
 \end{aligned}
\]
On every other vertex and edge of $G-v$, let $\psi'$ agree with $\psi$.
In configuration \textup{(I)}, $h_{i+2}=c_{i+3}=c_{i-1}=h_i'$.
In configuration \textup{(II)}, $h_{i+2}=c_{i+4}=c_i=h_i'$.
This proves \eqref{shell:halfturn}.

To verify that $\psi'$ is proper, first consider the vertex $u_i$
and all five edges incident with it in $G-v$. Their four colors
belonging to $\mathcal H$ are given explicitly by
\[
 (h_i',c_{i-1}',c_i',f_i)=
 \begin{cases}
  (c_{i-1},c_i,c_{i+1},c_{i+2}),&\text{in configuration \textup{(I)}},\\
  (c_i,c_{i+2},c_{i-1},c_{i+1}),&\text{in configuration \textup{(II)}}.
 \end{cases}
\]
Each tuple consists of the four distinct colors of $\mathcal H$.
The other two incident edge colors are the two distinct members of
$R\setminus L_i$, and are unchanged. Thus the color of $u_i$ and
its five incident edge colors are pairwise distinct.

For vertex adjacency, the new rim-vertex colors $h_i'=h_{i+2}$
are pairwise distinct. Every exterior neighbor of a rim vertex
has its unchanged color in $R$, by Lemma~\ref{shell:witnesses},
whereas each new rim-vertex color belongs to $\mathcal H$.
Thus adjacent vertices involving a rim vertex also have different
colors. All constraints involving only unchanged elements continue
to hold because $\psi$ was proper. This proves that $\psi'$ is a
proper total coloring of $G-v$.

Finally, let $L_i'$ be the spoke list computed from $\psi'$. The six
colors on $u_i$ and its incident edges still form
$\mathcal H\cup(R\setminus L_i)$, so
\[
 L_i'=\Omega\setminus\bigl(\mathcal H\cup(R\setminus L_i)\bigr)=L_i.
\]
\end{proof}

\begin{lemma}\label{shell:saturation}
For the chosen vertex $v\in V_4(G)$ of the minimum counterexample $G$,
use the coloring $\psi$ and shell notation constructed above.
If $u_i z$ is an edge with $z$ exterior to the rim, every color of
$L_i$ occurs on $z$ or on an edge incident with $z$.  In particular
all four colors of $R$ occur on $x_i$ or its incident edges, and
$\psi(x_i)$ occurs on an edge incident with $y_i$.
\end{lemma}

\begin{proof}
Suppose $p\in L_i$ were absent from $z$ and every edge at $z$.
Recolor $u_i z$ with $p$.  This is legal at $u_i$ by the definition
of $L_i$, and legal at $z$ by the assumed absence.  Let $t$ be the old
color of $u_i z$.  If $t\in \mathcal{H}$, the new missing list at $u_i$ meets
$\mathcal{H}$, and Lemma~\ref{wheel:boundary} restores $v$.
This is a contradiction.

Otherwise $t\in R\setminus L_i$.  After the change it is missing at
$u_i$, and it was already missing at $u_{i+1}$, because
$L_{i+1}=R\setminus L_i$.  Recolor the rim edge $u_i u_{i+1}$ with
$t$.  This uses current missing lists at both endpoints and is proper.
Its old color lies in $\mathcal{H}$ and now enters a missing list; again
Lemma~\ref{wheel:boundary} restores $v$.  Both alternatives are
impossible, proving the first assertion.

Apply the first assertion to the edge $u_i x_i$: every color of
$L_i$ occurs on $x_i$ or on an edge incident with $x_i$.
Apply it again to $u_{i+1}x_i$, using index $i+1$ and the same
exterior vertex $z=x_i$: every color of $L_{i+1}$ occurs on
$x_i$ or on an edge incident with $x_i$ as well. Since
$L_i\cup L_{i+1}=R$ by \eqref{shell:D}, these two applications give
\[
 R=L_i\cup L_{i+1}
 \subseteq \{\psi(x_i)\}\cup
 \{\psi(e):e\in E(G-v),\ x_i\in V(e)\}.
\]
Similarly, apply the first assertion to $u_i y_i$.
By \eqref{shell:orientation},
$L_i=\{\psi(x_i),\psi(y_i)\}$ with $\psi(x_i)\ne\psi(y_i)$.
Hence the required occurrence of $\psi(x_i)$ must be on an edge
incident with $y_i$, since $y_i$ itself has the different color
$\psi(y_i)$.
\end{proof}

Say that a vertex is \emph{frozen in a given total coloring} if its
color cannot be changed while all other elements retain their colors.
This refers only to an individual vertex change; it does not assert
that the entire coloring admits no more general recoloring.

\begin{lemma}\label{shell:freezing}
Consider the shell of the chosen vertex $v\in V_4(G)$ in the minimum
counterexample $G$. Let $\psi$ be the coloring constructed above and
let $\psi'$ be the coloring supplied by Lemma~\ref{shell:states}.
Each $x_i$ and $y_i$ is frozen in both $\psi$ and $\psi'$.
\end{lemma}

\begin{proof}
By Lemma~\ref{shell:witnesses}, each named exterior vertex is the unique
witness for a demand $(j,p)$ at some rim vertex.  If that exterior
vertex could be properly recolored, the color $p$ would disappear
from all exterior neighbors of $u_j$. Only an exterior vertex color
has changed, so the rim colors and all missing lists $L_i$ are
unchanged. The coloring therefore remains in state \textup{(D)}.
Since the rim is induced, Proposition~\ref{wheel:witness} would
restore $v$, contradicting the choice of $G$.

By Lemma~\ref{shell:states}, $\psi'$ is also a proper coloring in
state \textup{(D)}: the lists $L_j$ are unchanged, and its rim-vertex
and rim-edge colors each still form $\mathcal H$. The unchanged
exterior vertex colors give every demand the same unique witness.
The preceding argument therefore freezes each witness in $\psi'$ as well.
Both applications of Proposition~\ref{wheel:witness} concern $G-v$;
the auxiliary diagonal used to obtain $\psi$ was already deleted.
\end{proof}

\begin{theorem}\label{shell:separation}
Let $G\in\mathcal T_6$ be a counterexample of minimum order to total
eight-colorability. In the shell constructed above for any $v\in V_4(G)$,
each $x_i$ has degree six and
each $y_i$ has degree at least five.  Distinct four-vertices of $G$ are
at distance at least three.
\end{theorem}

\begin{proof}
For $z\in\{x_0,x_1,x_2,x_3,y_0,y_1,y_2,y_3\}$, define
\[
 D_z=\Omega\setminus\bigl(\{\psi(z)\}\cup
               \{\psi(e):e\in E(G-v),\ z\in V(e)\}\bigr).
\]
The degree of $z$ in $G-v$ equals its degree in $G$, since $v$ has only
the four rim neighbors.  Its vertex color and incident edge colors
are pairwise distinct; hence
\begin{equation}\label{shell:missing-at-exterior}
 |D_z|=7-d_G(z).
\end{equation}
By Lemma~\ref{shell:states}, $\psi'$ changes only the four rim vertices
and four rim edges. It therefore leaves $z$ and all its incident edges
unchanged, giving the same missing-color set $D_z$.
By Lemma~\ref{shell:freezing}, $z$ is frozen in each of $\psi$ and $\psi'$.
Hence every color in $D_z$ must occur on a neighbor of $z$ in each
of these two colorings. Otherwise assigning that color just to $z$
would be a proper recoloring in the coloring where it was absent
from all neighbors.

Take $z=x_i$.  Lemma~\ref{shell:saturation} implies $D_z\subseteq \mathcal{H}$.
Four distinct neighbors are already specified:
$u_i,u_{i+1},y_i,y_{i+1}$. Under $\psi$, the two rim neighbors
have colors $\{h_i,h_{i+1}\}$; under $\psi'$, their colors are
$\{h_{i+2},h_{i+3}\}$ by \eqref{shell:halfturn}.
These two pairs are disjoint and together exhaust $\mathcal H$.
For every $\alpha\in D_z$, at least one of $\psi,\psi'$ therefore
assigns color $\alpha$ to neither rim neighbor. Freezing in that
coloring requires a nonrim neighbor of color $\alpha$, whose color
is unchanged between $\psi$ and $\psi'$.
The two displayed $y$-neighbors have colors in $R$
and cannot supply it.  There is no further rim neighbor by
Lemma~\ref{shell:distinct}.  Thus all colors in $D_z$ must be supplied
by the at most $d_G(z)-4$ remaining neighbors.  Each neighbor has one
color, and so
\[
 7-d_G(x_i)\le d_G(x_i)-4.
\]
Together with $d_G(x_i)\le6$, this gives $d_G(x_i)=6$.

Now take $z=y_i$. Its only rim neighbor is $u_i$, with color
$h_i$ under $\psi$ and the distinct color $h_{i+2}$ under $\psi'$.
For each $\alpha\in D_z$, at least one of these two colorings assigns
to $u_i$ a color different from $\alpha$. Freezing in that coloring
requires another neighbor of color $\alpha$; every such neighbor
has the same color in $\psi$ and $\psi'$. There are
$d_G(z)-1$ other neighbors, but $x_i$ cannot be such a witness.
Indeed its color occurs on an edge at $y_i$ by
Lemma~\ref{shell:saturation}, and therefore is not in $D_z$.
Consequently
\[
 7-d_G(y_i)\le d_G(y_i)-2,
\]
which implies $d_G(y_i)\ge5$.  Possible repetitions of colors among
the unchanged neighbors only reduce their capacity, and do not affect
either upper bound.

Every neighbor of $v$ has degree six by Proposition~\ref{red:minimal}.
Every vertex at distance exactly two from $v$ appears in one of the
exterior triples of its complete rim links \eqref{shell:link}, and
therefore has degree at least five by the bounds just proved.  There
can be no other four-vertex within distance two of $v$.  Since $v$ was
arbitrary, distinct four-vertices are at distance at least three.
\end{proof}

\subsection{Removing the degree-four vertices}

Let $S=V_4(G)$. For each $v\in S$, let $C_v$ be its cyclic link in $G$,
a four-cycle with vertex set $N_G(v)$. Independently choose one of the
two diagonals of each $C_v$. Delete all vertices of $S$ and their
incident edges, and insert each chosen diagonal in the disk formed
by the four triangular faces originally incident with the corresponding
vertex $v$. This disk has boundary $C_v$. The vertices and edges of each
$C_v$ are retained. Denote the resulting embedded graph by $H$.
Mark each inserted diagonal and the disk bounded by its cycle $C_v$.
A \emph{marked disk} here retains the position of a
deleted vertex; it does not impose any precoloring on the core.

\begin{theorem}\label{shell:core}
Let $G\in\mathcal T_6$ be a counterexample of minimum order to total
eight-colorability. Set $S=V_4(G)$, allowing $S=\varnothing$, and
construct $H$ by the simultaneous deletion and diagonal insertion
described above.
The graph $H$ is a simple four-connected sphere triangulation
whose degrees are five and six.  It has exactly twelve five-vertices.
The marked quadrilaterals have pairwise disjoint vertex sets.  In each
marked quadrilateral, the diagonal endpoints have degree six in $H$
and the other two boundary vertices have degree five. Deleting all marked
diagonals recovers exactly $G-S$.
\end{theorem}

\begin{proof}
By Theorem~\ref{shell:separation}, vertices of $S$ are nonadjacent and
have no common neighbor.  Their closed wheel vertex sets are therefore
pairwise disjoint. At each $v\in S$, the four incident triangular faces
form a disk bounded by the induced four-cycle $C_v$. The disks have disjoint
vertices and disjoint interiors, so replacing each by its two diagonal
triangles is a legitimate operation on the sphere triangulation.
Every inserted diagonal was absent in $G$, distinct inserted diagonals
have disjoint endpoints, and no inserted edge is a loop.  The resulting
triangulation is simple.

For each $v\in S$, the four vertices of $C_v$, namely the four
neighbors of $v$ in $G$, all have degree six by
Proposition~\ref{red:minimal}. Fix such a neighbor $u\in N_G(v)$.
The vertex $u$ is retained, and deleting $v$ removes the edge $uv$.
No other deleted vertex is adjacent to $u$, because distinct vertices
of $S$ have no common neighbor. Thus $u$ loses exactly one incident
edge. It gains one new edge precisely when it is an endpoint of the
diagonal chosen in $C_v$. Consequently,
\[
 d_H(u)=
 \begin{cases}
  6-1+1=6,&\text{if }u\text{ is an endpoint of the chosen diagonal},\\
  6-1=5,&\text{if }u\text{ is either of the other two vertices of }C_v.
 \end{cases}
\]
The sets $N_G(v)$, for $v\in S$, are pairwise disjoint, so these
local degree changes affect disjoint sets of vertices.
Every original five-vertex has no neighbor in $S$ and retains degree
five. Every original six-vertex outside these neighbor sets also
retains its degree. Thus the only degrees of $H$ are five and six,
with exactly two new five-vertices for each deleted vertex $v\in S$.
Euler's identity for $G$ gives $2n_4(G)+n_5(G)=12$, and hence
\begin{equation}\label{shell:twelve}
 n_5(H)=n_5(G)+2n_4(G)=12.
\end{equation}
This proves the asserted degrees at the vertices of each marked
quadrilateral; their vertex sets are the pairwise disjoint sets
$N_G(v)$, for $v\in S$. If $S$ is empty, the construction simply gives
$H=G$, and the same identity applies.

To prove that $H$ is four-connected, we show that every triangle of
$H$ is facial, using Corollary~\ref{shell:no-four-cycle}.
Let $T$ be any triangle of $H$. It contains at most one inserted
diagonal, because distinct inserted diagonals have disjoint endpoints.
If $T$ contains none, it is an old triangle on retained vertices.
By Proposition~\ref{red:minimal}, it bounded a face in $G$. No deleted
vertex was incident with that face, so it remains a face of $H$.

Suppose instead that $T$ contains an inserted diagonal $ac$.
Let $v\in S$ be the deleted vertex whose cyclic link is
$C_v=abcda$ and whose chosen diagonal is $ac$.
Write $V(T)=\{a,c,z\}$.
If $z=b$ or $z=d$, then $T$ is respectively the triangle $abc$ or
$acd$, one of the two new faces in the disk bounded by $C_v$.
Otherwise $z\notin\{a,b,c,d\}=N_G(v)$, since $z$ is already
distinct from $a,c$. The edges $az,cz$ are old edges because
$ac$ is the only inserted diagonal in $T$.
In $G$, the edge $ac$ is absent because $C_v$ is induced, and
$vz$ is absent because $z\notin N_G(v)$.
Thus $vazcv$ is a four-cycle of $G$ whose two possible chords,
$ac$ and $vz$, are both absent. It is therefore an induced
four-cycle containing $v$, contrary to
Corollary~\ref{shell:no-four-cycle}.

Equation~\eqref{shell:twelve} ensures that $H$ has at least five vertices.
All its triangles are facial, and Lemma~\ref{red:connectivity} now
implies that it is four-connected.  Finally, the only added edges were
the marked diagonals, while precisely the old edges incident with
$S$ were removed.  Discarding those diagonals therefore gives exactly
$G-S$.
\end{proof}

The four-connectivity of $H$ supplies a Hamilton path, while its twelve
five-vertices provide suitable endpoints and the degree deficit
used in the matching argument of Theorem~\ref{core:main}.
A coloring of $H$ whose vertices
use a fixed four-element color set restricts, on discarding the marked
diagonals, to the input for Theorem~\ref{restore:main}.
The restored graph need not retain the four-element vertex color set;
the original induction hypothesis asks only for a total coloring
with eight colors.

Only induced four-cycles of $G$ through a removed center were excluded,
to ensure that an inserted diagonal creates no nonfacial triangle.
No restriction on all induced four-cycles of $H$ is needed.

\section{Simultaneous restoration}
\label{restore:section}

Theorem~\ref{wheel:local} fixes the exterior of one wheel.  When several
deleted vertices are present, a two-color interchange may change
the exterior of more than one wheel.  We therefore organize the
restoration in two phases.  During the first phase all old vertices
retain their colors, and interchanges are allowed throughout the
already colored graph.  During the second phase every remaining
wheel is uniform, all global interchanges cease, and the explicit
release of Proposition~\ref{wheel:uniform} finishes the coloring.

\begin{theorem}
\label{restore:main}
Let $G$ be a graph with $\Delta(G)\leq6$, and let $S$
be an independent set of vertices of degree four.  Suppose that:
\begin{enumerate}
\item the sets $N_G(s)$, for $s\in S$, are pairwise disjoint;
\item every vertex in $N_G(s)$ has degree six in $G$;
\item each $N_G(s)$ induces a four-cycle.
\end{enumerate}
If $G-S$ has a proper total coloring from an eight-element color set
$\Omega$ in which every vertex receives a color from one fixed
four-element set $B_0\subseteq\Omega$, then $G$ has a proper total
coloring from $\Omega$.
\end{theorem}

The conclusion does not require the final vertex colors to remain
in $B_0$.  The fixed input vertex color set is used to justify the global
interchanges in the proof.  No embedding or planarity is assumed in Theorem~\ref{restore:main}.

\begin{proof}
Call the vertices of $G-S$ \emph{old vertices}, and the vertices in
$S$ \emph{centers}, and put $R_0=\Omega\setminus B_0$.
In the first phase, let $T\subseteq S$ be the set of restored centers;
all their incident edges have also been restored. We maintain a
proper total coloring of $F=G[(V(G)\setminus S)\cup T]$ in which
every old vertex keeps its original color in $B_0$.
Initially $T=\varnothing$.

For $s\in S\setminus T$ and $u\in N_G(s)$, the hypotheses give
$u\in V(F)$ and $N_G(u)\cap S=\{s\}$. Thus $us$ is the only
edge at $u$ absent from $F$, so $d_F(u)=5$ and its spoke list has
size $8-(1+5)=2$. Adding $s$ and its four spokes to $F$ gives a
graph in which $s$ has degree four and its four neighbors have
degree six and induce a four-cycle. Theorem~\ref{wheel:local}
therefore restores $s$
whenever the current coloring is in neither state (U) nor state (D),
keeping all old vertex colors fixed and changing only rim edges
among the edges already present.

\medskip
\noindent\emph{The permissible two-color interchange.}
For distinct $p,r\in R_0$, define the \emph{two-color conflict graph}
as the subgraph of the total graph of $F$ induced by the elements
currently colored $p$ or $r$, using the terminology of
Section~\ref{prelim:coloring}.
Interchanging $p,r$ on the elements represented in a connected
component preserves properness: conflicting pairs within that
component remain differently colored, and an outside element
conflicting with one represented in the component has neither
color $p$ nor color $r$. Such an interchange keeps all old vertex
colors fixed.

Any vertex of $F$ represented in this conflict graph is a restored
center. Its neighbors in $F$ are all old, and its incident edges
have distinct colors. Its representing vertex therefore has
degree at most one in the conflict graph.
For an edge $xy\in E(F)$ currently colored $p$ or $r$, at either
endpoint the vertex and its incident edges in $F$ have pairwise
distinct colors. Thus at each endpoint, at most one of these elements
has color $r$ when $xy$ has color $p$, or color $p$ when $xy$ has
color $r$. The vertex representing $xy$ in the conflict
graph therefore has degree at most two.
Every component is consequently an isolated vertex, a path, or a cycle.

\medskip
\noindent\emph{Breaking an alternating-pair state.}
Let $s\in S\setminus T$ be in state (D), and denote its four neighbors
by $u_0,u_1,u_2,u_3$ in order around their induced four-cycle.
The four distinct colors of these rim vertices belong to $B_0$,
so the local set of vertex colors $\mathcal{H}$ equals $B_0$ and $R=R_0$.
Write
\[
 L_0=L_2=\{p,q\},\qquad L_1=L_3=\{r,t\},
 \qquad R_0=\{p,q,r,t\}.
\]
For each $i\in\{0,1,2,3\}$, let $e_i$ be the unique edge of $F$
incident with $u_i$ whose color is $r$ for even $i$ and $p$ for odd
$i$. To justify this definition for even $i$, note that $r\notin L_i$
and $u_i$ has a color in $B_0$, so $r$ must occur on an incident
edge; properness gives uniqueness. For odd $i$, use $p\notin L_i$
in the same way. Each $e_i$ is a nonrim edge, since every rim-edge
color lies in $B_0$. These edges are distinct: an edge shared by
two rim vertices would be a rim edge or a chord, and the induced
four-cycle has no chord.

Consider the conflict graph for the colors $p,r$, and write
$e_i=u_iw_i$. If $i$ is even, $e_i$ has color $r$, whereas neither
$u_i$ nor any edge of $F$ incident with $u_i$ has color $p$.
Thus any neighbor of the vertex representing $e_i$ in the conflict
graph must represent an element colored $p$: either $w_i$ itself
or an edge of $F$ incident with $w_i$.
At most one such element exists, since $w_i$ and its incident
edges have distinct colors. For odd $i$, interchange $p$ and $r$
in this argument.
Thus the vertex representing each $e_i$ has degree at most one in
the conflict graph. A path has at most two vertices of degree at
most one, an isolated component has one, and a cycle has none.
Since $e_0,e_1,e_2,e_3$ are distinct, a component containing one
of their representing vertices contains at least one and at most
two of these four vertices. Interchange $p$ and $r$ on that component.

Set $b_i=1$ if the color of $e_i$ changes and $b_i=0$ otherwise.
Exactly one or two of these four values equal $1$, so both $0$
and $1$ occur. Since $e_i$ is the only edge at $u_i$ colored
$p$ or $r$, the new spoke lists are
\[
 L_i'=
 \begin{cases}
 \{p,q\},& i\text{ even and }b_i=0,\\
 \{r,q\},& i\text{ even and }b_i=1,\\
 \{r,t\},& i\text{ odd and }b_i=0,\\
 \{p,t\},& i\text{ odd and }b_i=1.
 \end{cases}
\]
For consecutive indices modulo four, these lists satisfy
$L_i'\cap L_{i+1}'=\varnothing$ if and only if $b_i=b_{i+1}$.
Since $b_0,b_1,b_2,b_3$ include both $0$ and $1$, there is an index
$i$ with $b_i\ne b_{i+1}$. Hence $L_i'\cap L_{i+1}'\ne\varnothing$,
so state (D) no longer holds. Each rim edge $u_i u_{i+1}$ retains
its color $c_i\in B_0$, because the interchange uses only $p,r\in R_0$.
The new rim color set $A_i'=L_i'\cup\{c_{i-1},c_i\}$ therefore
contains a color in $B_0$ and cannot equal $R_0$, excluding state (U).
Theorem~\ref{wheel:local} now restores $s$ while keeping all old
vertex colors fixed. Previously restored centers may have changed
color during the interchange, but the total coloring remains proper.

\medskip
\noindent\emph{Termination of the first phase.}
While some unrestored center is not in state (U), restore it by
Theorem~\ref{wheel:local}, first using the preceding interchange
if it is in state (D). Each step increases $\lvert T\rvert$ and
keeps all old vertex colors fixed. Although the states at other
centers may change, finiteness of $S$ ensures termination, either
with all centers restored or with every remaining center in state (U).

\medskip
\noindent\emph{The final uniform phase.}
In the latter case, stop all global interchanges and apply
Proposition~\ref{wheel:uniform} to the remaining centers one at a
time. Each application changes only its rim edges, one rim-vertex
color, and the new center and spokes. For any other unrestored
center, its rim vertices and all edges incident with them keep
their colors: distinct rims have disjoint vertex sets, edges
joining them are not recolored, and no new spoke meets the other rim.
Thus each local restoration in this second phase preserves
$h_i,c_i,L_i,A_i$ at every other unrestored center: these quantities
depend only on the unchanged vertex and incident-edge colors just
described. That center remains in state (U). This preservation
applies after all global two-color interchanges have stopped;
the interchanges in the first phase may change $L_i$ and $A_i$.

The vertex recolored in one rim may be an exterior neighbor of
another. When restoring a later center, let $u_0$ be its rim vertex
to be recolored and $x_1,x_2,x_3$ its three exterior neighbors.
For the \emph{current} proper total coloring $\psi$ of $F$, choose
$p\in R\setminus\{\psi(x_1),\psi(x_2),\psi(x_3)\}$; such a
color exists because state (U) gives $\lvert R\rvert=4$.
The construction in Proposition~\ref{wheel:uniform} then restores
the center properly with $u_0$ colored $p$. Each step restores one
center and preserves state (U) at all remaining centers. Since
$S$ is finite, all vertices and edges of $G$ are eventually colored.
\end{proof}

Applying Theorem~\ref{restore:main} requires two inputs from the reduction.
The structural argument must produce
independent degree-four centers with disjoint induced rims and
degree-six rim vertices. The coloring argument on the smaller graph
must place all old vertex colors in one four-element color set.
With these inputs, Theorem~\ref{restore:main} restores all centers
simultaneously; it does not invoke an induction hypothesis requiring
intermediate graphs to retain a prescribed four-element vertex color set.

Under the structural hypotheses of Theorem~\ref{restore:main},
Corollary~\ref{ext:three-matchings} supplies the required coloring of
$G-S$ whenever $G-S$ is four-colorable and admits three pairwise
edge-disjoint matchings whose deletion leaves maximum degree at most
three. Both results hold without a planarity assumption.

\section{Proof of the main theorem}
\label{main:section}

\begin{proof}[Proof of Theorem~\ref{main:theorem}]
The empty graph and the complete graphs of order at most three have
the direct colorings described in Section~\ref{red:section}.  Every
other maximal planar graph has a sphere triangulation
embedding.  Suppose, for a contradiction, that some graph in
$\mathcal T_6$ has no total coloring with eight colors.  Choose
such a graph $G$ with the fewest vertices. Every smaller graph in
$\mathcal T_6$ therefore has a proper total eight-coloring.

Proposition~\ref{red:minimal} shows that $G$ is four-connected,
every triangle is facial, its degrees belong to $\{4,5,6\}$, and
every neighbor of a degree-four vertex has degree six. The reductions
used in that proposition preserve simplicity, triangulation and the
maximum-degree bound, so minimality applies within $\mathcal T_6$.

Put $S=V_4(G)$.  Theorem~\ref{shell:separation} proves that distinct
vertices in $S$ have distance at least three.  In particular $S$ is
independent and its neighbor sets are pairwise disjoint.  Each such
neighbor set induces its cyclic four-vertex link.  Delete $S$ and
insert one diagonal in each exposed quadrilateral.  By
Theorem~\ref{shell:core}, the resulting graph $H$ is a
four-connected sphere triangulation with degrees five and six.
This statement also covers $S=\varnothing$, when $H=G$.

Apply Corollary~\ref{dual:vertex-four} to $H$. Explicitly, its dual
$H^*$ is a simple three-connected cubic planar graph whose faces have
length five or six. Theorem~\ref{dual:kardos} supplies a Hamiltonian
cycle of $H^*$, and Lemma~\ref{dual:two-trees} supplies a proper vertex
coloring of $H$ with colors from $B_0=\{1,2,3,4\}$.
Theorem~\ref{core:main} extends this very coloring to a proper
total coloring with $\Omega=\{1,2,3,4,5,6,7,8\}$, preserving all
vertex colors in $B_0$.

Discard every inserted diagonal and its color.  The retained graph
is exactly $G-S$, and the restricted coloring is proper, with all
vertex colors in $B_0$. The set $S$ is independent; each of its
vertices has degree four and an induced rim of degree-six neighbors;
and the rims have pairwise disjoint vertex sets. Thus
Theorem~\ref{restore:main} applies and gives a proper total coloring
of $G$ from $\Omega$. If $S=\varnothing$, the coloring of $H=G$
already suffices.

This contradicts the choice of $G$.
\end{proof}

Minimality is used only in the structural reduction. The coloring of
$H$ with its restricted vertex color set is supplied independently by
Theorem~\ref{core:main}; the final coloring of $G$ need not retain
that restriction.

\medskip
\noindent\textbf{Dependency check.}
The reductions leading to $H$ are proved in
Proposition~\ref{red:minimal} and Theorems~\ref{shell:separation}
and~\ref{shell:core}. The prescribed-color extension used for $H$
is proved in Theorem~\ref{core:main}, using
Theorems~\ref{ext:main} and~\ref{core:third-matching}.
The restoration is Theorem~\ref{restore:main}. Their proofs occur
earlier in this document and do not assume the main theorem.
When $S$ is empty, $H=G$ need not be smaller; this causes no circularity,
because Theorem~\ref{core:main} is an independent extension theorem,
not the induction hypothesis. The initial vertex coloring is supplied
only by Section~\ref{dual:section}. Thus no unsupported reducibility
claim or strengthened five-palette extension is used.

\section{Further questions}
\label{scope:section}

We consider stronger extension properties and the degree-six problem
beyond triangulations. The counterexamples below delimit the local
methods, while the matching and factorization questions give sufficient
conditions for broader total-coloring results.

\subsection{Prescribed bridges}

The terminal-block induction in Theorem~\ref{ext:main} prescribes
one outgoing bridge color while preserving all vertex colors.
The corresponding statement for two bridges is false, even for
bipartite subcubic graphs.

\begin{proposition}\label{scope:two-ports}
There is a subcubic graph with a proper vertex four-coloring
and two bridges with disjoint endpoint sets, individually assigned legal colors,
for which no total five-coloring extending the prescribed vertex
coloring also extends both bridge assignments.
\end{proposition}

\begin{proof}
Start with $K_{2,3}$, with parts $\{u,v\}$ and $\{r,s,w\}$,
and attach a new leaf at each of $r,s$.  Give $u,v$ vertex color
$1$, give $r,s,w$ color $2$, and give the leaves color $3$.
This is a proper coloring from $B=\{1,2,3,4\}$.  Prescribe the
extra color $c$ on both leaf edges. These edges are bridges with
disjoint endpoint sets, and their assigned colors are legal at both endpoints.

Every edge of the $K_{2,3}$ part has list $\{3,4,c\}$.  At $r$
and $s$, the bridge color $c$ forces the other two incident edges
to receive $3,4$ in some order.  Therefore $ur,us$ must receive
the two different colors $3,4$, forcing $uw$ to have color $c$.
Likewise $vr,vs$ use $3,4$, forcing $vw$ to have color $c$.
The adjacent edges $uw,vw$ would both have color $c$, a contradiction.
\end{proof}

Either bridge assignment alone can be extended by
Proposition~\ref{ext:marked-port}, followed by coloring the other
leaf edge. That edge must avoid its two endpoint colors and the
two colored edges at its internal endpoint, so at most four of
the five colors are forbidden. The contradiction comes from
demanding both prescribed assignments at once.
Consequently an extension theorem with several prescribed boundary
edges must include a compatibility condition.

\subsection{Preserving vertex colors}

Theorem~\ref{core:main} extends every proper vertex four-coloring
of the core, whereas Theorem~\ref{main:theorem} allows the vertex
colors to change during restoration. This raises the following
stronger extension problem.

\begin{question}\label{scope:fixed-four}
For which maximal planar graphs $G$ with $\Delta(G)=6$
does the following extension property hold? For every proper vertex
coloring $\phi:V(G)\to\{1,2,3,4\}$, there exists a proper total
coloring
\[
 \psi:V(G)\cup E(G)\to\{1,\ldots,8\},
 \qquad \psi(v)=\phi(v)\quad\text{for every }v\in V(G).
\]
Only the vertex colors are prescribed. The edge colors may be
chosen from all eight colors, subject to the total-coloring constraints.
\end{question}

For a fixed $\phi$, put $\Omega=\{1,\ldots,8\}$ and assign to
each edge $uv$ the list
\begin{equation}\label{scope:fixed-lists}
 \Lambda_\phi(uv)=\Omega\setminus\{\phi(u),\phi(v)\}.
\end{equation}
Question~\ref{scope:fixed-four} asks whether these lists always admit
a proper edge coloring. Indeed, membership in $\Lambda_\phi(uv)$
is precisely the condition that the edge color avoid its two
endpoint colors; properness of the edge coloring supplies the
remaining constraints. These six-element lists are correlated through
$\phi$; the question concerns this particular family of lists.

An affirmative answer would separate the choice of vertex colors
from the subsequent edge-coloring problem. In particular, any
prescription on a subset of vertices that extends to a proper
vertex four-coloring of $G$ would also extend to a total
eight-coloring. This concerns vertex prescriptions alone; the
edge-precoloring obstruction in Proposition~\ref{scope:two-ports}
illustrates the distinction. A negative answer would identify
an obstruction to preserving vertex colors, without contradicting
Theorem~\ref{main:theorem}.
Theorem~\ref{core:main} gives an affirmative answer for the
four-connected degree-five/six subclass. Proposition~\ref{scope:eulerian}
below treats every supplied three-coloring of an Eulerian
triangulation; it does not treat a supplied coloring that uses
all four colors.

There are two specific obstacles to extending the present proof.
First, if the neighbors of a deleted degree-four vertex have
alternating colors $1,2,1,2$, both diagonals of the exposed
quadrilateral join equally colored vertices. Thus adding a
diagonal need not preserve the prescribed coloring, even when
the resulting graph satisfies the structural hypotheses of the
core theorem. Second, restoration must assign the center its
prescribed color $t=\phi(v)$, as well as preserve the old vertex
colors. Since all neighbors avoid $t$, they use at most three
colors; states (U) and (D), which require four distinct rim-vertex
colors, do not occur in this setting. Nevertheless,
Corollary~\ref{wheel:four-neighbours} does not prescribe the color
of the restored center. Its conclusion is therefore insufficient.

Even allowing arbitrary changes on the rim edges alone does not
give the required stronger restoration statement for general
graphs. To see this, take a four-wheel and attach three distinct
leaves to each rim vertex $u_i$. Give the center vertex color $3$,
the rim vertices colors $1,2,1,2$ in cyclic order, and all leaves
color $4$. On the graph with the center deleted, color the rim
edges $3,4,3,4$ in cyclic order. At $u_i$, color its three leaf
edges with the three members of
$\{1,2,7,8\}\setminus\{\phi(u_i)\}$. This is a proper total
coloring. In the notation of Section~\ref{wheel:section}, all four
sets $A_i$ equal $\{3,4,5,6\}$. If the leaf-edge colors remain
fixed and the center must receive color $3$, the four spokes
must receive distinct colors from $\{4,5,6\}$, which is impossible,
regardless of the rim-edge colors. This graph is not maximal
planar, so it is not a counterexample to
Question~\ref{scope:fixed-four}, but it rules out the unrestricted
local strengthening.

One can instead allow more old edges to be recolored, or seek
restrictions on the boundary data imposed by triangulation.
Under the structural
hypotheses of Theorem~\ref{restore:main}, one can ask whether a
total eight-coloring of $G-S$ that agrees with a prescribed
proper four-coloring $\phi$ of $G$ can be replaced by a total
eight-coloring of $G$ agreeing with $\phi$ on every vertex,
including $S$, while all old edge colors are allowed to change.
Even an affirmative answer to this restoration problem would
still require a construction of the initial coloring compatible
with $\phi$; the diagonal obstruction above prevents a direct
appeal to the present core construction. Moreover, reductions
for a counterexample with prescribed vertex colors would need
their own justification: minimality in
Theorem~\ref{main:theorem} concerns unrestricted total colorings.

\subsection{External inputs and simpler subclass proofs}

The proof above does not use the Four-Color Theorem. Its initial
vertex coloring comes from Kardo\v{s}'s theorem via
Section~\ref{dual:section}. Theorems~\ref{ext:main}, \ref{core:main}
and~\ref{restore:main} are extension results independent of that
initial construction. Avoiding the computer-assisted Hamiltonicity
input as well is a different problem, not resolved by this revision.
For Eulerian triangulations there is a different construction,
based on three vertex colors and bipartite matchings. Recall that
a connected graph is Eulerian if every vertex has even degree.
The following statement preserves a supplied three-coloring and
has a proof independent of Theorem~\ref{main:theorem}.

\begin{proposition}\label{scope:eulerian}
Let $G$ be an Eulerian maximal planar graph with
$\Delta(G)\leq6$. Then $G$ has a proper vertex three-coloring,
and every proper coloring $\phi:V(G)\to\{1,2,3\}$ extends to a
proper total coloring of $G$ with colors in $\{1,\ldots,8\}$.
\end{proposition}

\begin{proof}
The edgeless cases are immediate. For $K_3$, retain its three
distinct vertex colors and give each edge the color of the opposite
vertex. We may therefore assume that $|V(G)|\geq4$, so the fixed
sphere embedding is a triangulation with all degrees even.
The classical three-colorability theorem for Eulerian triangulations
provides a proper vertex three-coloring; Tsai and
West~\cite{tsaiwest2011} give a constructive proof independent of
the Four-Color Theorem.

Fix any such coloring $\phi$, and put $V_i=\phi^{-1}(i)$.
For $1\leq i<j\leq3$, let $G_{ij}$ be the bipartite graph on
$V_i\cup V_j$ consisting of all edges of $G$ between these two
sets. At a vertex $v\in V_i$, consecutive vertices of its cyclic
link are adjacent and have colors different from $i$. Their colors
therefore alternate between the other two colors. Consequently
$|N_G(v)\cap V_j|=d_G(v)/2\leq3$ for each $j\ne i$, and every
pair graph $G_{ij}$ has maximum degree at most three.

By K\"onig's edge-coloring theorem~\cite{konig1916}, the edges
of a bipartite graph of maximum degree at most three can be
partitioned into three matchings, allowing empty matchings.
Apply this to each $G_{ij}$ and choose one of the three
matchings as $M_{ij}$. At every degree-three vertex the three
incident edges belong to different matchings, so $M_{ij}$
covers all such vertices. It follows that
$\Delta(G_{ij}-M_{ij})\leq2$.

Give every edge of $M_{ij}$ the color $6-i-j$, the unique member
of $\{1,2,3\}\setminus\{i,j\}$. This assignment avoids both
endpoint colors. Each $M_{ij}$ is a matching, and at a vertex
in $V_i$ the two possible matching edges receive the two distinct
colors other than $i$. Thus all colored vertices and matching edges
form a proper partial total coloring.

Put $Q=G-(M_{12}\cup M_{13}\cup M_{23})$. Each vertex belongs
to two pair graphs, and its degree in each residual pair graph is
at most two; hence $\Delta(Q)\leq4$. Since $Q$ is simple,
Vizing's edge-coloring theorem gives a proper edge coloring of
$Q$ with at most five colors; see Misra and Gries~\cite{misragries1992}
for a constructive proof. Use the colors $4,5,6,7,8$ for these
edges. They avoid all vertex and matching-edge colors, while
properness within $Q$ excludes conflicts between its edges.
Together these assignments give the required total coloring and
preserve $\phi$.
\end{proof}

The construction uses only three-colorability of Eulerian
triangulations and the edge-coloring theorems of K\"onig and Vizing.
The alternation on each link also explains its limitation: a
triangulation containing an odd-degree vertex cannot have a proper
vertex three-coloring. In particular, the degree-five vertices of
the core prevent a direct application of this construction.

For non-Eulerian triangulations, Section~\ref{dual:section} already
supplies the initial vertex coloring without the Four-Color Theorem.
An elementary replacement for Kardo\v{s}'s Hamiltonicity input would
further reduce the external dependencies, but is not established here.
The matching decomposition itself does not supply that initial coloring.

Alternatively, one could construct a total eight-coloring of the core
directly, without restricting the number of colors on its vertices.
This alone would not suffice: Theorem~\ref{restore:main} uses the
restricted vertex color set throughout its first phase. In a state
(D), the four rim colors exhaust that set, so neither of the two
interchanged colors occurs on an old vertex. The only vertex-elements
in the corresponding conflict graph are restored centers, each of
degree at most one; edge-elements have degree at most two. Its
components are therefore paths, cycles or isolated vertices. If the
interchanged colors can occur on old vertices, a vertex-element can
have several neighboring vertices of the other color, and the degree-two bound on the
conflict graph no longer follows. The endpoint argument used to
break state (D) then fails. This route requires both a direct core
coloring and a replacement for the simultaneous restoration argument
that controls these more general two-color components.

\begin{remark}\label{scope:five-palette-research}
The common input vertex color set in Theorem~\ref{restore:main} can in
fact be enlarged to five colors. To see this, let $B$ be that set, with
$|B|\le5$. In an alternating state (D), the four distinct rim-vertex
colors form $\mathcal H\subseteq B$, so $|B\setminus\mathcal H|\le1$.
Its alternating two-element lists $P,Q$ are disjoint and lie in
$\Omega\setminus\mathcal H$. Thus choose $p\in P\setminus B$ and
$r\in Q\setminus B$. Neither color occurs on an old vertex, and the
same two-color conflict-graph argument from Theorem~\ref{restore:main}
applies. Rim-edge colors in $\mathcal H$ remain unchanged, excluding
state (U) after the interchange exactly as before. The final uniform
phase already allows arbitrary current exterior vertex colors.

This observation is not used in the completed proof. To replace the
Hamiltonian-dual input by the elementary Five-Color Theorem, one would
still need a compatible core coloring. A sufficient additional statement
would be that a prescribed proper vertex coloring
$\phi:V(Q)\to\{1,\ldots,5\}$ of the relevant subcubic residual graph
extends to a total coloring with the same five colors, or equivalently
that the edge lists
$\{1,\ldots,5\}\setminus\{\phi(u),\phi(v)\}$ admit a proper edge
coloring. No such general statement is proved or assumed here.
\end{remark}

The explicit coloring of the separating-triangle family in
Proposition~\ref{red:chain-colouring} already avoids the Four-Color
Theorem. Extending that construction to the four-connected core
would require a structural substitute for its chain order;
the count of twelve degree-five vertices alone provides no such order.

\subsection{General planar graphs}
\label{scope:planar-section}

The natural next existence problem is $\chit(G)\leq8$ for every
planar graph with $\Delta(G)=6$. This is the
remaining planar case identified in the literature cited in
the introduction; see, for example, \cite{zhuzhouxu2026}.
Since triangulation can increase the maximum degree, we seek a
matching decomposition directly in the original graph. Its first
step holds without any further structural assumption.

\begin{proposition}\label{scope:first-matching}
Every planar graph $G$ with $\Delta(G)\leq6$ has
a matching covering all vertices of degree six.
\end{proposition}

\begin{proof}
Put $U=V_6(G)$ and fix $X\subseteq V(G)$. Let $C$ be an odd
component of $G-X$ whose vertices all belong to $U$, as counted
by $q_U(G-X)$ in Lemma~\ref{core:required}. If $|V(C)|=1$,
its edge cut in $G$ has size six. Otherwise $|V(C)|\geq3$,
and the simple planar edge bound gives
\[
 |\delta_G(V(C))|
 =6|V(C)|-2|E(C)|\geq12.
\]
Every edge in this cut has its other endpoint in $X$. Distinct
components have disjoint cuts, so
$6q_U(G-X)\leq\sum_{x\in X}d_G(x)\leq6|X|$.
Lemma~\ref{core:required} now supplies the matching.
\end{proof}

Deleting this matching leaves maximum degree at most five,
but a total seven-coloring of the remainder does not by itself
complete the argument. Its vertex coloring might assign equal
colors to the endpoints of a deleted matching edge. Giving that
edge an eighth color would not repair this vertex conflict.
The prescribed-color extension theorem avoids this difficulty
if two further matching deletions leave a subcubic graph.

\begin{question}\label{scope:planar-matchings}
Does every planar graph $G$ with $\Delta(G)\leq6$
contain three pairwise edge-disjoint matchings $M_1,M_2,M_3$
such that $\Delta(G-(M_1\cup M_2\cup M_3))\leq3$?
\end{question}

An affirmative answer, combined with
Corollary~\ref{ext:three-matchings}, would extend every prescribed
proper vertex four-coloring of every such planar graph to a
total eight-coloring. A separate vertex-coloring existence theorem
for the broader class would then supply
the initial coloring for the existence result. Thus this
decomposition would settle the planar degree-six case and
Question~\ref{scope:fixed-four} at the same time. The decomposition
already exists whenever $G$ has a proper edge coloring with at
most six colors: take three color classes, allowing empty
classes. In particular, Vizing's theorem supplies it when
$\Delta(G)\leq5$. The remainder need only be subcubic, so the
question is weaker than six-edge-colorability.

The exact additional matching conditions can be stated without
a spanning-path assumption. Set $F_0=G$. At step $i\in\{1,2,3\}$,
put $U_i=V_{7-i}(F_{i-1})$, choose a matching $M_i$ covering
$U_i$, and put $F_i=F_{i-1}-M_i$. Each successful step ensures
$\Delta(F_i)\leq6-i$. By Lemma~\ref{core:required}, the
condition at that step is precisely
\begin{equation}\label{scope:successive-matchings}
 q_{U_i}(F_{i-1}-X)\leq |X|
 \qquad\text{for every }X\subseteq V(G).
\end{equation}
Conversely, a decomposition as in
Question~\ref{scope:planar-matchings}, in any order, gives these
three steps. A vertex of current degree $7-i$ must lose at
least $4-i$ incident edges in the $4-i$ remaining matchings,
so each remaining matching must cover it. Proposition~\ref{scope:first-matching}
proves the first set of inequalities unconditionally. The
unproved issue is the existence of a choice of $M_1$ and then
$M_2$ that makes the next two sets of inequalities hold.

Arbitrary choices are insufficient. For example, let $P$ be
the \emph{triangular prism graph}: it consists of two
vertex-disjoint triangles $a_1a_2a_3a_1$ and $b_1b_2b_3b_1$,
together with the three edges $a_1b_1,a_2b_2,a_3b_3$ and
no other edges. Thus $P$ is a planar three-regular graph
with six vertices and nine edges. Let $J=L(P)$ be its
line graph: the vertices of $J$ correspond to the edges
of $P$, and two vertices of $J$ are adjacent precisely
when the corresponding edges share an endpoint in $P$.
Consequently $J$ has nine vertices. Each edge of $P$
shares an endpoint with four distinct other edges, two
at each endpoint, so every vertex of $J$ has degree four.
Its planarity follows by drawing, in small neighborhoods
of the vertices of $P$, the triangles joining the three
incident edge midpoints.
The first two required sets $U_1,U_2$ are
empty, so choosing $M_1=M_2=\varnothing$ is allowed. The third
step would then require a perfect matching of the odd-order
graph $J$, which is impossible. This example concerns a failed
sequence of choices, not failure of the decomposition in
Question~\ref{scope:planar-matchings}.

The curvature calculation extends to general plane graphs. Let $H$ be a
connected plane graph of maximum degree at most
six, and let $S$ be a proper vertex set with $|S|\geq2$ such
that $H[S]$ and $H[V(H)\setminus S]$ are connected and nonempty.
As in Lemma~\ref{core:curvature}, let $\Omega_S$ be the face of
$H[S]$ containing the complementary induced subgraph.
For each face $f$ of $H[S]$, let $\partial f$ denote its
boundary walk and $|\partial f|$ its length, counting every
traversal of an edge. This walk need not be a simple cycle:
an edge traversed twice contributes two to its length.
Put $\ell=|\partial\Omega_S|$ and define
\[
 p_6(S)=\sum_{v\in S}(6-d_H(v)),\qquad
 \eta(S)=\sum_{f\ne\Omega_S}(|\partial f|-3),
\]
where the second sum is over all faces of $H[S]$ other than
$\Omega_S$. Each such
face has length at least three, so $\eta(S)\geq0$.
Writing $n=|S|$, $m=|E(H[S])|$ and $q$ for its number of
faces, Euler's formula and the edge-side count give
$n-m+q=2$ and $2m=\ell+3(q-1)+\eta(S)$. Thus
$m=3n-3-\ell-\eta(S)$, and summing degrees in $H$ over $S$
yields
\begin{equation}\label{scope:general-curvature}
 p_6(S)+|\delta_H(S)|=6+2\ell+2\eta(S).
\end{equation}
Every edge from $S$ to its complement leaves a vertex on
the boundary of $\Omega_S$. There are at most $\ell$ such
vertices, since the connected graph $H[S]$ has an edge.
Consequently, if $R\subseteq H$ has maximum degree at most
two and $K=H-E(R)$, at most $2\ell$ cut edges are deleted,
and
\[
 p_6(S)+|\delta_K(S)|\geq6+2\eta(S)\geq6.
\]

Nontriangular faces therefore strengthen this inequality;
their presence is not itself the obstruction to the matching
argument. The difficulty is that $p_6(S)$ now includes larger
contributions from vertices of degree at most four, and need
not satisfy the bounds used for the five/six-degree core.
Moreover, a general planar graph need not have a spanning path
of the kind used there, and deleting a linear forest need not
preserve connectivity. A productive next step would be a
selection rule for the first two matchings that controls
these degree deficits and the odd components in
\eqref{scope:successive-matchings}. The displayed curvature
bound applies only when both induced sides are connected;
applying it to components arising in a matching obstruction
requires a separate verification of that hypothesis.

\subsection{The general degree-six problem}

For arbitrary graphs, the total-coloring
conjecture is known through maximum degree five
\cite{kostochka1996,zhuzhouxu2026}. The bound eight at maximum
degree six remains open in general. Both four-colorability and
the matching decomposition can fail outside the planar setting.

It is enough to study six-regular graphs. Indeed, from any
graph of maximum degree at most six, take two
disjoint copies and join each vertex of degree less than six
to its corresponding vertex in the other copy. This increases
each deficient degree by one and preserves simplicity.
Repeating the construction at most six times gives a six-regular
graph containing the original graph as an
induced subgraph. A total eight-coloring restricts to every
subgraph. Hence the following is equivalent to the uniform
eight-color bound for all graphs of maximum
degree at most six.

\begin{question}\label{scope:regular-six}
Does every six-regular graph have a proper total
coloring with at most eight colors?
\end{question}

An affirmative answer would settle the degree-six case of
the Total Coloring Conjecture, including all planar graphs in
Section~\ref{scope:planar-section}. The regular completion
above does not preserve planarity and is used only for this
unrestricted question. It removes low-degree reductions from
the existence problem and isolates the compatibility
constraints at vertices with six incident edges.

Two examples delimit what one can expect from the present
method. The graph $K_7$ is not four-colorable, although it has
a total seven-coloring: label its vertices by the residues
modulo seven, give vertex $i$ color $2i$, and give edge $ij$
color $i+j$, with arithmetic modulo seven. Vertex colors are
distinct, the incident edge colors at a fixed vertex are
distinct, and $i+j$ equals neither $2i$ nor $2j$ when $i\ne j$.
Thus failure of four-colorability is a limitation of the
extension method, not a total-coloring obstruction.

Conversely, $K_{3,3,3}$ is three-colorable and six-regular,
but has no three-matchings decomposition with subcubic
remainder. Such a decomposition would require each of its
nine vertices to belong to all three matchings, making each
matching perfect, which is impossible on an odd number of
vertices. Nevertheless it has a total eight-coloring
preserving its natural three-coloring: choose a perfect
matching in each bipartite graph between two parts, give that
matching the vertex color of the third part, and color the
four-regular remainder with five new edge colors by Vizing's
theorem. At every vertex the two matching colors are distinct
and avoid its vertex color. This proves the asserted coloring
and shows that the three-matchings condition is not necessary.
Question~\ref{scope:planar-matchings} therefore cannot simply
be extended to all four-colorable graphs.

A different decomposition suggests a compatibility problem
that does not require four vertex colors. A \emph{2-factor}
is a spanning subgraph whose components are cycles.
Petersen's factorization theorem~\cite{petersen1891} states
that every $2r$-regular graph, for $r\geq1$, decomposes into
$r$ edge-disjoint 2-factors. In particular, every six-regular
graph has a 2-factor.

Deleting a 2-factor leaves a four-regular graph, for which a
total six-coloring exists by \cite[Theorem~1]{kostochka1996}.
Existence alone is insufficient:
the vertex colors must also be proper on the deleted cycle
edges, and odd cycles need an additional available color.
The following sufficient condition states these two
requirements explicitly.

\begin{proposition}\label{scope:factor-extension}
Let $G$ be a six-regular graph, let $F$ be a
2-factor, and put $Q=G-E(F)$. Suppose that $Q$ has a proper
total coloring $\psi$ with colors in $\{1,\ldots,6\}$ such
that its vertex coloring is proper on all of $G$. For each
vertex $v$, let $\mu(v)$ be the unique member of
$\{1,\ldots,6\}$ not used on $v$ or on an edge of $Q$
incident with $v$. If every odd cycle of $F$ contains an
edge $uv$ with $\mu(u)=\mu(v)$, then $G$ has a proper
total eight-coloring extending $\psi$.
\end{proposition}

\begin{proof}
Since $Q$ is four-regular, its four edges incident with $v$
and the vertex $v$ have five distinct colors. Thus $\mu(v)$
is well-defined. Color each even cycle of $F$ alternately
with colors $7,8$. In each odd cycle, choose an edge $uv$
as in the hypothesis and give it color $\mu(u)=\mu(v)$.
The remaining edges of that cycle form a path of even
length, which can be colored alternately with $7,8$.

The selected edge in an odd cycle avoids the colors of
both endpoints and every incident edge of $Q$. Its two
neighboring cycle edges use $7,8$, so they also avoid its
color. All other new edges use colors absent from $Q$,
and adjacent new edges receive distinct colors. Different
cycles have disjoint vertex sets. Finally, the unchanged
vertex coloring is proper on $G$ by hypothesis. These
assignments form a proper total coloring extending $\psi$.
\end{proof}

\begin{question}\label{scope:compatible-factor}
Which six-regular graphs admit a 2-factor
$F$ and a total six-coloring of $G-E(F)$ satisfying the
two compatibility conditions in
Proposition~\ref{scope:factor-extension}? In particular,
under what structural hypotheses can $F$ and the coloring
be chosen so that the vertex coloring is proper on $G$
and every odd cycle of $F$ has an edge whose endpoints
have the same missing color?
\end{question}

The two choices must be coordinated. A total six-coloring
of the remainder need not be proper across the omitted
edges, and changing it to repair vertex conflicts can
alter the missing colors on several cycles. When an
even-cycle 2-factor is available, the missing-color
condition disappears, but compatibility of the vertex
coloring still has to be proved. The graph $K_7$ cannot
satisfy the first condition, since it has no proper vertex
six-coloring, but it has the explicit total seven-coloring above.
If the two compatibility conditions could be established for
every other connected six-regular graph, then applying
Proposition~\ref{scope:factor-extension} componentwise would answer
Question~\ref{scope:regular-six}. Whether such a choice always
exists is unresolved here. Establishing it for specified classes
would extend the method. A graph admitting no such compatible
decomposition would limit this method without refuting the
Total Coloring Conjecture.


\begin{thebibliography}{99}
\bibitem{behzad1965}
Behzad, M..
{Graphs and Their Chromatic Numbers}.
Doctoral thesis, Michigan State University, 1965.

\bibitem{kostochka1996}
Kostochka, A. V..
The total chromatic number of any multigraph with maximum degree five is at most seven.
\emph{Discrete Mathematics}, 162 (1996), 199--214.
DOI: \url{https://doi.org/10.1016/0012-365X(95)00286-6}.

\bibitem{sufangzhu2025}
Su, R., Fang, G., Zhu, E..
The {Total Coloring Conjecture} holds for planar graphs without three special subgraphs.
arXiv preprint, 2025.
\url{https://arxiv.org/abs/2507.12737v1}.

\bibitem{zhuzhouxu2026}
Zhu, E., Zhou, Y., Xu, J..
Total-coloring of planar graphs with maximum degree 6 and without prescribed 4-cycles.
arXiv preprint, 2026.
\url{https://arxiv.org/abs/2609.19503v1}.

\bibitem{fenglin2013}
Feng, Y., Lin, W..
A concise proof for total coloring subcubic graphs.
\emph{Information Processing Letters}, 113 (2013), 664--665.
DOI: \url{https://doi.org/10.1016/j.ipl.2013.06.003}.

\bibitem{zhuxu2017}
Zhu, E., Xu, J..
A sufficient condition for planar graphs with maximum degree 6 to be totally 8-colorable.
\emph{Discrete Applied Mathematics}, 223 (2017), 148--153.
DOI: \url{https://doi.org/10.1016/j.dam.2017.01.036}.

\bibitem{zhouetal2022}
Zhou, Y., Zhao, D., Ma, M., Xu, J..
Total coloring of recursive maximal planar graphs.
\emph{Theoretical Computer Science}, 909 (2022), 12--18.
DOI: \url{https://doi.org/10.1016/j.tcs.2022.01.024}.

\bibitem{cranstonwest2013}
Cranston, D. W., West, D. B..
A guide to the discharging method.
arXiv preprint, 2013.
Version 1; the statement cited here is Theorem~4.7 of that version.
\url{https://arxiv.org/abs/1306.4434v1}.

\bibitem{tutte1954}
Tutte, W. T..
A short proof of the factor theorem for finite graphs.
\emph{Canadian Journal of Mathematics}, 6 (1954), 347--352.
DOI: \url{https://doi.org/10.4153/CJM-1954-033-3}.

\bibitem{thomassen1983}
Thomassen, C..
A theorem on paths in planar graphs.
\emph{Journal of Graph Theory}, 7 (1983), 169--176.
See also the correction in \cite{chiba1986}.
DOI: \url{https://doi.org/10.1002/jgt.3190070205}.

\bibitem{chiba1986}
Chiba, N., Nishizeki, T..
A theorem on paths in planar graphs.
\emph{Journal of Graph Theory}, 10 (1986), 449--450.
DOI: \url{https://doi.org/10.1002/jgt.3190100404}.

\bibitem{kardos2020}
Kardo\v{s}, F..
A Computer-Assisted Proof of the Barnette--Goodey Conjecture: Not Only Fullerene Graphs Are Hamiltonian.
\emph{SIAM Journal on Discrete Mathematics}, 34 (2020), 62--100.
DOI: \url{https://doi.org/10.1137/140984737}.

\bibitem{hall1935}
Hall, P..
On representatives of subsets.
\emph{Journal of the London Mathematical Society}, s1-10 (1935), 26--30.
DOI: \url{https://doi.org/10.1112/jlms/s1-10.37.26}.

\bibitem{tsaiwest2011}
Tsai, M.-T., West, D. B..
A new proof of 3-colorability of {Eulerian} triangulations.
\emph{Ars Mathematica Contemporanea}, 4 (2011), 73--77.
DOI: \url{https://doi.org/10.26493/1855-3974.193.8e7}.

\bibitem{konig1916}
K{\"o}nig, D..
{\"U}ber {Graphen} und ihre {Anwendung} auf {Determinantentheorie} und {Mengenlehre}.
\emph{Mathematische Annalen}, 77 (1916), 453--465.
DOI: \url{https://doi.org/10.1007/BF01456961}.

\bibitem{misragries1992}
Misra, J., Gries, D..
A constructive proof of {Vizing}'s theorem.
\emph{Information Processing Letters}, 41 (1992), 131--133.
DOI: \url{https://doi.org/10.1016/0020-0190(92)90041-S}.

\bibitem{petersen1891}
Petersen, J..
Die {Theorie} der regul{\"a}ren graphs.
\emph{Acta Mathematica}, 15 (1891), 193--220.
DOI: \url{https://doi.org/10.1007/BF02392606}.

\end{thebibliography}
\end{document}